\documentclass[10pt,a4paper]{amsart}
\usepackage[english]{babel}
\usepackage[T1]{fontenc}
\usepackage[latin1]{inputenc}
\usepackage{graphicx}
\usepackage{amsmath,amssymb,amsthm,mathrsfs,,txfonts,ifthen,xr}
\usepackage{soul}
\usepackage{array, multirow}
\usepackage{stmaryrd}
\usepackage{mathrsfs}
\usepackage{color}
\usepackage{hyperref}
\usepackage{enumerate}
\usepackage[all]{xy}

\theoremstyle{plain}
\newtheorem{theo}{Theorem}[section]
\newtheorem{lemme}[theo]{Lemma}
\newtheorem{prop}[theo]{Proposition}
\newtheorem{coro}[theo]{Corollary}

\theoremstyle{definition}
\newtheorem{defn}[theo]{Definition}

\newtheorem{nota}[theo]{Notation} 

\newtheorem{conj}[theo]{Conjecture}

\theoremstyle{remark}
\newtheorem{rema}[theo]{Remark}

\def\Sp{{\rm Sp}}
\def\J{{\rm J}}
\def\U{{\rm U}}
\def\M{{\rm M}}
\def\det{{\rm det}}
\def\O{{\rm O}}
\def\GL{{\rm GL}}

\def\tr{{\rm tr}}
\def\G{{\rm G}}
\def\C{{\rm C}}
\def\N{{\rm N}}
\def\S{{\rm S}}
\def\exp{{\rm exp}}
\def\dim{{\rm dim}}
\def\diag{{\rm diag}}
\def\Stab{{\rm Stab}}
\def\Im{{\rm Im}}
\def\Op{{\rm Op}}
\def\Hom{{\rm Hom}}

\def\supp{{\rm supp}}
\def\Ad{{\rm Ad}}
\def\rk{{\rm rk}}
\def\H{{\rm H}}
\def\Re{{\rm Re}}
\def\reg{{\rm reg}}
\def\Z{{\rm Z}}
\def\ch{{\rm ch}}
\def\sign{{\rm sign}}
\def\K{{\rm K}}

\def\P{{\rm P}}
\def\Id{{\rm Id}}
\def\Mat{{\rm Mat}}
\def\E{{\rm E}}
\def\A{{\rm A}}
\def\T{{\rm T}}
\def\Chc{{\rm Chc}}

\def\D{{\rm D}}
\def\L{{\rm L}}
\def\ch{{\rm ch}}
\def\sh{{\rm sh}}
\def\th{{\rm th}}

\def\i{{\rm i}}
\def\V{{\rm V}}
\def\W{{\rm W}}
\def\M{{\rm M}}

\def\X{{\rm X}}
\def\Y{{\rm Y}}

\def\st{{\rm st}}
\def\min{{\rm min}}

\def\R{{\rm R}}
\def\Eig{{\rm Eigen}}
\def\Spec{{\rm Spec}}
\def\Stab{{\rm Stab}}
\def\Gar{{\rm Gar}}
\def\Ann{{\rm Ann}}
\def\Cl{{\rm Cl}}
\def\d{{\rm d}}
\def\st{{\rm st}}

\allowdisplaybreaks

\title{Transfer of characters for discrete series representations of the unitary groups in the equal rank case via the Cauchy-Harish-Chandra integral}

\author{Allan Merino}
\address{ Department of Mathematics \\ National University of Singapore \\ Block S17 \\ 10, Lower Kent Ridge Road \\ Singapore 119076 \\ Republic of Singapore}
\email{matafm@nus.edu.sg}
\date{}

\begin{document}

\maketitle

\begin{abstract}

As conjectured by T. Przebinda, the transfer of characters in the Howe's correspondence should be obtained via the Cauchy-Harish-Chandra integral. In this paper, we prove that the conjecture holds for the dual pair $(\G = \U(p, q), \G' = \U(r, s))$, $p+q = r+s$, starting with a discrete series representation $\Pi$ of $\widetilde{\U}(p, q)$.

\end{abstract}

\tableofcontents

\section{Introduction}

Let $\W$ be a finite dimensional vector space over $\mathbb{R}$ endowed with a non-degenerate, skew-symmetric, bilinear form $\langle\cdot,\cdot\rangle$, $\Sp(\W)$ be the corresponding group of isometries, $\widetilde{\Sp}(\W)$ be the metaplectic cover of $\Sp(\W)$ (see \cite[Definition~4.18]{TOM6}) and $(\omega, \mathscr{H})$ be the corresponding Weil representation (see \cite[Section~4.8]{TOM6}). For every irreducible reductive dual pair $(\G, \G')$ in $\Sp(\W)$, R. Howe proved (see \cite[Theorem~1]{HOW1}) that there is a bijection between $\mathscr{R}(\widetilde{\G}, \omega)$ and $\mathscr{R}(\widetilde{\G'}, \omega)$ whose graph is $\mathscr{R}(\widetilde{\G}\cdot\widetilde{\G'}, \omega)$ (where $\mathscr{R}(\widetilde{\G}, \omega)$ is defined in Section \ref{SectionCauchyHarishChandra}). More precisely, to every $\Pi \in \mathscr{R}(\widetilde{\G}, \omega)$, we associate a representation finitely generated admissible representation $\Pi'_{1}$ of $\widetilde{\G'}$ which has a unique irreducible quotient $\Pi'$ such that $\Pi \otimes \Pi' \in \mathscr{R}(\widetilde{\G}\cdot\widetilde{\G'}, \omega)$. We denote by $\theta: \mathscr{R}(\widetilde{\G}, \omega) \ni \Pi \to \Pi' = \theta(\Pi) \in \mathscr{R}(\widetilde{\G'}, \omega)$ the corresponding bijection.

\noindent As proved by Harish-Chandra (see \cite[Section~5]{HAR3} or Section \ref{TransferEigenDistributions}), all the representations $\Pi$ of $\widetilde{\G}$ (resp. $\Pi'$ of $\widetilde{\G'}$) appearing in the correspondence admit a character, i.e. a $\widetilde{\G}$-invariant distribution $\Theta_{\Pi}$ on $\widetilde{\G}$ (in the sense of Laurent Schwartz) given by a locally integrable function $\Theta_{\Pi}$ on $\widetilde{\G}$ which is analytic on $\widetilde{\G}^{\reg}$ (the set of regular elements of $\widetilde{\G}$). The character $\Theta_{\Pi}$ determines the representation $\Pi$. In particular, one way to understand the Howe correspondence, i.e. to make the map $\theta$ explicit, is to understand the transfer of characters.

\noindent In his paper \cite{TOM1}, T. Przebinda conjectured that the correspondence of characters should be obtained via the so-called Cauchy-Harish-Chandra integral that he introduced in \cite{TOM1}. We recall briefly the construction of this integral. Let $\T: \widetilde{\Sp}(\W) \to \S^{*}(\W)$ be the embedding of the metaplectic group inside the space of tempered distributions on $\W$ as in \cite[Definition~4.23]{TOM6} (see also Remark \ref{RemarkDistributionT}) and $\H_{1}, \ldots, \H_{n}$ be a maximal set of non-conjugate Cartan subgroups of $\G$ which are $\iota$-invariant, where $\iota$ is a Cartan involution on $\G$. Every Cartan subgroup $\H_{i}$ can be decomposed as $\H_{i} = \T_{i}\A_{i}$, with $\T_{i}$ maximal compact in $\H_{i}$. Let $\A'_{i}$ and $\A''_{i}$ be the subgroups of $\Sp(\W)$ defined by $\A'_{i} = \C_{\Sp(\W)}(\A_{i})$ and $\A''_{i} = \C_{\Sp(\W)}(\A'_{i})$. One can easily check that $(\A'_{i}, \A''_{i})$ form a dual pair in $\Sp(\W)$, which is not irreducible in general. For every function $\varphi \in \mathscr{C}^{\infty}_{c}(\widetilde{\A'_{i}})$, we define $\Chc(\varphi)$ by
\begin{equation*}
\Chc(\varphi) = \displaystyle\int_{\A''_{i} \setminus \W_{\A''_{i}}} \T(\varphi)(w) \overline{dw},
\end{equation*}
where $\overline{dw}$ is a measure on the manifold $\A''_{i} \setminus \W_{\A''_{i}}$ defined in \cite[Section~1]{TOM1}. As mentioned in \cite[Section~2]{TOM1} (see also Section \ref{SectionCauchyHarishChandra}), $\Chc(\Psi)$ is well-defined and the corresponding map $\Chc: \mathscr{C}^{\infty}_{c}(\widetilde{\A'_{i}}) \to \mathbb{C}$ is a distribution on $\widetilde{\A'_{i}}$. For every regular element $\tilde{h}_{i} \in \widetilde{\H_{i}}^{\reg}$, we denote by $\Chc_{\tilde{h}_{i}}$ the pull-back of $\Chc$ through the map $\widetilde{\G'} \ni \tilde{g}' \to \tilde{h}_{i}\tilde{g}' \in \widetilde{\A'_{i}}$. 

\noindent Assume now that $\rk(\G) \leq \rk(\G')$. In \cite{TOM4} (see also Section \ref{TransferEigenDistributions}), F. Bernon and T. Przebinda defined a map:
\begin{equation*}
\Chc^{*}: \mathscr{D}'(\widetilde{\G})^{\widetilde{\G}} \to \mathscr{D}(\widetilde{\G'})^{\widetilde{\G'}}, 
\end{equation*}
where $\mathscr{D}'(\widetilde{\G})^{\widetilde{\G}}$ is the set of $\widetilde{\G}$-invariant distributions on $\widetilde{\G}$. More precisely, if $\Theta$ is a $\widetilde{\G}$-invariant distribution given by a locally integrable function $\Theta$ on $\widetilde{\G}$, then, for every $\varphi \in \mathscr{C}^{\infty}_{c}(\widetilde{\G'})$, we get:
\begin{equation*}
\Chc^{*}(\Theta)(\varphi) = \sum\limits_{i=1}^{n} \cfrac{1}{|\mathscr{W}(\H_{i})|} \displaystyle\int_{\widetilde{\H_{i}}^{\reg}} \Theta(\tilde{h}_{i})|\det(1-\Ad(\tilde{h}^{-1}_{i}))_{\mathfrak{g}/\mathfrak{h}_{i}}| \Chc_{\tilde{h}_{i}}(\varphi) d\tilde{h}_{i}.
\end{equation*}
The conjecture can be stated as follows:

\begin{conj}

Let $\G_{1}$ and $\G'_{1}$ be the Zariski identity components of $\G$ and $\G'$ respectively. Let $\Pi \in \mathscr{R}(\widetilde{\G}, \omega)$ satisfying ${\Theta_{\Pi}}_{|_{\widetilde{\G}/\widetilde{\G_{1}}}} = 0$ if $\G = \O(\V)$, where $\V$ is an even dimensional vector space over $\mathbb{R}$ or $\mathbb{C}$. Then, up to a constant, $\Chc^{*}(\Theta_{\Pi}) = \Theta_{\Pi'_{1}}$ on $\widetilde{\G'_{1}}$.

\end{conj}

\noindent This result is well-known if the group $\G$ is compact and had been proved recently in \cite{TOM5} in the stable range. In this paper, we investigate the case $(\G, \G') = (\U(p, q), \U(r, s))$, $p+q = r+s$ ($p$ will always be assumed to be smaller or equal than $q$, in particular, the number of non-conjugate Cartan subgroups of $\G$ is p+1 (see Remark \ref{ExampleUnitary})) and $\Pi \in \mathscr{R}(\widetilde{\G}, \omega)$ a discrete series representation of $\widetilde{\G}$. Let $\lambda$ be the Harish-Chandra parameter of $\Pi$. In this case, using Li's result (see \cite[Proposition~2.4]{LI2} or Section \ref{CommutativeDiagram}), we get that $\Pi'_{1} = \Pi'$ and using \cite{PAUL2}, we know that $\Pi' = \theta(\Pi)$ is a discrete series representations of $\widetilde{\G'}$ (with Harish-Chandra parameter $\lambda'$), and the correspondence $\lambda \to \lambda'$ is known and explicit (see \cite[Theorem~2.7]{PAUL2}).

\noindent In order to prove that, up to a constant, $\Chc^{*}(\Theta_{\Pi}) = \Theta_{\Pi'_{1}} = \Theta_{\Pi'}$, we use a parametrisation of discrete series characters provided by Harish-Chandra (see \cite[Lemma~44]{HAR5}). More precisely, it follows from \cite{TOM4} and a result of Harish-Chandra (see \cite[Theorem~2]{HAR4}) that the distribution $\Chc^{*}(\Theta_{\Pi})$ is given by locally integrable function $\Theta'_{\Pi}$ analytic on $\widetilde{\G'}^{\reg}$. Using \cite[Theorem~2.2]{TOM4}, we proved in Proposition \ref{ImportantProposition1} that the value of $\Theta'_{\Pi}$ on $\widetilde{\H'}^{\reg}$, where $\H'$ is the compact Cartan subgroup of $\G'$, is of the form:
\begin{equation*}
\Delta(\check{h}') \Theta'_{\Pi}(\check{p}(\check{h}')) = \C \sum\limits_{\sigma \in \mathscr{S}_{r} \times \mathscr{S}_{s}} \varepsilon(\sigma) (\sigma \check{h}')^{\lambda_{\Pi}}, \qquad (\check{h}' \in \check{\H}'^{\reg}),
\end{equation*}
where $\C \in \mathbb{R}$, $\check{\H}'$ is a double cover of $\H'$ (see Section \ref{SectionCauchyExplicit}) chosen such that $\rho' = \frac{1}{2}\sum\limits_{\alpha > 0} \alpha$ is analytic integral, $\check{p}$ is a map from $\check{\H}'$ into $\widetilde{\H'}$ (which is not an isomorphism of double covers in general), and $\lambda_{\Pi}$ is a linear form on $i\mathfrak{h}'$ depending on $\Pi$ which is conjugated to $\lambda$ under $\mathscr{S}_{r+s}$. Moreover, using results of \cite{TOM2}, we proved in Proposition \ref{ImportantProposition2} that 
\begin{equation*}
\sup\limits_{\tilde{g}' \in \widetilde{\G'}^{\reg}} |\D(\tilde{g}')|^{\frac{1}{2}} |\Theta'_{\Pi}(\tilde{g}')| < \infty,
\end{equation*}
where $\D$ is the Weyl denominator defined in Notations \ref{NotationsD}. Finally, applying \cite[Theorem~1.3]{TOM4} to our particular dual pair, it follows that $z\Chc^{*}(\Theta_{\Pi}) = \chi_{\lambda_{\Pi}}(z)\Chc^{*}(\Theta_{\Pi})$ for every $z \in \Z(\mathscr{U}(\mathfrak{g}'_{\mathbb{C}}))$, where $\chi_{\lambda_{\Pi}}$ is the character of $\Z(\mathscr{U}(\mathfrak{g}'_{\mathbb{C}}))$ obtained via the linear form $\lambda_{\Pi}$ as in Remark \ref{RemarkInfinitesimalCharacter} and then, using a result of Harish-Chandra (see \cite[Lemma~44]{HAR5}) we get that $\Chc^{*}(\Theta_{\Pi})$ is the character of a discrete series representations of $\widetilde{\G'}$ with Harish-Chandra parameter $\lambda_{\Pi}$. 

\noindent In Section \ref{CommutativeDiagram}, we prove, by using results of \cite{TOM8} (see also \cite{LI2}), that $\T(\Theta_{\Pi})$ is a well-defined $\widetilde{\G}\cdot\widetilde{\G'}$-invariant distribution on $\S^{*}(\W)$ and we get in Corollary \ref{CorollarySection7} the following equality:
\begin{equation*}
\T(\overline{\Theta_{\Pi}}) = \C_{\Pi \otimes \Pi'} \T(\overline{\Chc^{*}(\Theta_{\Pi})}),
\end{equation*}
where $\C_{\Pi \otimes \Pi'}$ is a constant depending on $\Pi$ and $\Pi'$. In particular, we can hope that the following diagram often commutes (up to a constant):
\begin{equation*}\label{GroupsDiagram}
\xymatrix{ \mathscr{D}'(\widetilde{\G})^{\widetilde{\G}} \ar[rr]^{\Chc^{*}} \ar[rd]_{\T} && \mathscr{D}'(\widetilde{\G'})^{\widetilde{\G'}} \ar[ld]^{\T} \\ & \S^{*}(\W)^{\widetilde{\G}\cdot\widetilde{\G'}} }
\end{equation*}
Moreover, according to Li's result (see \cite{LI2} or Section \ref{CommutativeDiagram}), $\Pi$ can be embedded in $\omega$ as a subrepresentation, and by projecting onto the $\nu \otimes \Pi'$-isotypic component (where $\nu$ is the lowest $\widetilde{\K}$-type of $\Pi$ as in Theorem \ref{TheoremDiscreteSeries}), we get (see Equation \eqref{EquationIntroduction}) the following equality:
\begin{equation*}
\Chc^{*}(\Theta_{\Pi})(\varphi) = d_{\Pi} \tr\left(\displaystyle\int_{\widetilde{\K}} \displaystyle\int_{\widetilde{\G}} \displaystyle\int_{\widetilde{\G'}} \overline{\Theta_{\Pi_{\nu}}(\tilde{k})} \overline{\Theta_{\Pi}(\tilde{g})} \varphi(\tilde{g}')\omega(\tilde{k}\tilde{g}\tilde{g}') d\tilde{g}'d\tilde{g}d\tilde{k}\right),
\end{equation*}
where $\varphi \in \mathscr{C}^{\infty}_{c}(\widetilde{\G'})$ and $d_{\Pi}$ is the formal degree of $\Pi$ (see Remark \ref{RemarkFormalDegree}).

\bigskip

\noindent \textbf{Acknowledgements: } The motivation of this paper comes from a talk given by Wee Teck Gan at the Representation Theory and Number theory seminar at NUS in 2019. I would like to thank Tomasz Przebinda for the many useful discussions during the preparation of this paper. This research was supported by the MOE-NUS AcRF Tier 1 grants R-146-000-261-114 and R-146-000-302-114.

\section{Howe correspondence and Cauchy-Harish-Chandra integral}

\label{SectionCauchyHarishChandra}

Let $\W$ be a finite dimensional vector space over $\mathbb{R}$ endowed with a non-degenerate, skew-symmetric, bilinear form $\langle\cdot, \cdot\rangle$. We denote by $\Sp(\W)$ the corresponding group of isometries, i.e.
\begin{equation*}
\Sp(\W) = \left\{g \in \GL(\W), \langle g(w), g(w')\rangle = \langle w, w' \rangle, \thinspace (\forall w, w' \in \W)\right\},
\end{equation*}
and by $\widetilde{\Sp}(\W)$ the metaplectic group as in \cite[Definition~4.18]{TOM6}: it's a connected two-fold cover of $\Sp(\W)$. We will denote by $\pi: \widetilde{\Sp}(\W) \to \Sp(\W)$ the corresponding covering map.

\noindent We say that a pair of subgroup $(\G, \G')$ of $\Sp(\W)$ is a dual pair if $\G$ is the centralizer of $\G'$ in $\Sp(\W)$ and vice-versa. The dual pair is said to be reductive if both $\G$ and $\G'$ act reductively on $\W$ and irreducible if we cannot find an orthogonal decomposition of $\W = \W_{1} \oplus \W_{2}$ where both $\W_{1}$ and $\W_{2}$ are $\G\cdot\G'$-invariant. One can easily prove that the preimages $\widetilde{\G} = \pi^{-1}(\G)$ and $\widetilde{\G'} = \pi^{-1}(\G')$ in $\widetilde{\Sp}(\W)$ form a dual pair in $\widetilde{\Sp}(\W)$.

\noindent Let $(\omega, \mathscr{H})$ be the Weil representation of $\widetilde{\Sp}(\W)$ corresponding to a fixed unitary character of $\mathbb{R}$ and $(\omega^{\infty}, \mathscr{H}^{\infty})$ be the corresponding smooth representation (see \cite[Section~4.8]{TOM6}). For a subgroup $\widetilde{\H}$ of $\widetilde{\Sp}(\W)$, we denote by $\mathscr{R}(\widetilde{\H}, \omega)$ the set of conjugacy classes of irreducible admissible representations $(\Pi, \mathscr{H}_{\Pi})$ of $\widetilde{\H}$ which can be realized as a quotient of $\mathscr{H}^{\infty}$ by a closed $\omega^{\infty}(\widetilde{\H})$-invariant subspace.

\noindent As proved by R. Howe (see \cite[Theorem~1]{HOW1}), for every reductive dual pair $(\G, \G')$ of $\Sp(\W)$, we have a one-to-one correspondence between $\mathscr{R}(\widetilde{\G}, \omega)$ and $\mathscr{R}(\widetilde{\G}', \omega)$ whose graph is $\mathscr{R}(\widetilde{\G} \cdot \widetilde{\G'}, \omega)$. More precisely, if $\Pi \in \mathscr{R}(\widetilde{\G}, \omega)$, we denote by $\N(\Pi)$ the intersection of all the closed $\widetilde{\G}$-invariant subspaces $\mathscr{N}$ such that $\Pi \approx \mathscr{H}^{\infty} / \mathscr{N}$. Then, the space $\mathscr{H}(\Pi) = \mathscr{H}^{\infty} / \N(\Pi)$ is a $\widetilde{\G} \cdot \widetilde{\G'}$-module; more precisely, $\mathscr{H}(\Pi) = \Pi \otimes \Pi'_{1}$, where $\Pi'_{1}$ is a $\widetilde{\G'}$-module, not irreducible in general, but Howe's duality theorem says that there exists a unique irreducible quotient $\Pi'$ of $\Pi'_{1}$ with $\Pi' \in \mathscr{R}(\widetilde{\G'}, \omega)$ and $\Pi \otimes \Pi' \in \mathscr{R}(\widetilde{\G}\cdot\widetilde{\G'}, \omega)$.

\noindent We will denote by $\theta: \mathscr{R}(\widetilde{\G}, \omega) \to \mathscr{R}(\widetilde{\G'}, \omega)$ the corresponding bijection.

\begin{nota}

We use here the notations of \cite{TOM6}. We denote by $\S^{*}(\W)$ the space of tempered distributions on $\W$ and by 
\begin{equation*}
\T: \widetilde{\Sp}(\W) \to \S^{*}(\W)
\end{equation*} 
the injection of $\widetilde{\Sp}(\W)$ into $\S^{*}(\W)$ (see \cite[Definition~4.23]{TOM6}). We denote by $\Sp^{c}(\W)$ the subset of $\Sp(\W)$ given by 
\begin{equation*}
\left\{g \in \Sp(\W), \det(g-1) \neq 0\right\}
\end{equation*}
and by $\widetilde{\Sp}^{c}(\W)$ its preimage in $\widetilde{\Sp}(\W)$. 

\end{nota}

\begin{rema}

As explained in \cite{TOM6}, for every $\tilde{g} \in \widetilde{\Sp}^{c}(\W)$, the distribution $\T(\tilde{g})$ is defined by $\T(\tilde{g}) = \Theta(\tilde{g}) \chi_{c(g)} \mu_{\W}$, where $\Theta$ is the character of the Weil representation $(\omega, \mathscr{H})$ defined in \cite[Definition~4.23]{TOM6}, $\chi_{c(g)}: \W \to \mathbb{C}$ is the function on $\W$ given by $\chi_{c(g)}(w) = \chi\left(\frac{1}{4}\langle (g+1)(g-1)^{-1}w, w\rangle\right)$ with $g = \pi(\tilde{g})$ and $\mu_{\W}$ is the appropriately normalized Lebesgue measure on $\W$.

\noindent The map $\T$ can be extended to $\widetilde{\Sp(\W)}$ and to $\mathscr{C}^{\infty}_{c}(\widetilde{\Sp}(\W))$ by 
\begin{equation*}
\T(\varphi) = \displaystyle\int_{\widetilde{\Sp}(\W)} \varphi(\tilde{g})\T(\tilde{g}) d\tilde{g}, \qquad (\varphi \in \mathscr{C}^{\infty}_{c}(\widetilde{\Sp}(\W))),
\end{equation*}
where $d\tilde{g}$ is the Haar measure on $\widetilde{\Sp}(\W)$. As proved in \cite[Section~4.8]{TOM6}, for every $\varphi \in \mathscr{C}^{\infty}_{c}(\widetilde{\Sp}(\W))$, the distribution $\T(\varphi)$ on $\W$ is given by a Schwartz function on $\W$ still denoted by $\T(\varphi)$, i.e.
\begin{equation*}
\T(\varphi)(\phi) = \displaystyle\int_{\W} \T(\varphi)(w)\phi(w) d\mu_{\W}(w), \qquad (\phi \in \S(\W)).
\end{equation*}

\label{RemarkDistributionT}

\end{rema}

We now recall the construction of the Cauchy-Harish-Chandra integral introduced by T. Przebinda in \cite{TOM1}. We denote by $\H_{i}$, $1 \leq i \leq n$ be a maximal set of non-conjugate Cartan subgroups of $\G$. As explained in \cite[Section~2.3.6]{WAL}, for every $1 \leq i \leq n$, the Cartan subgroup $\H_{i}$ can be decomposed as $\H_{i} = \T_{i}\A_{i}$, with $\T_{i}$ maximal compact in $\H_{i}$ ($\A_{i}$ is called the split part of $\H_{i}$). For $1 \leq i \leq n$, we denote by $\A'_{i}$ and $\A^{''}_{i}$ the subgroups of $\Sp(\W)$ given by $\A'_{i} = \C_{\Sp(\W)}(\A_{i})$ and $\A^{''}_{i} = \C_{\Sp(\W)}(\A^{'}_{i})$. As recalled in \cite[Section~1]{TOM1}, there exists an open and dense subset $\W_{\A^{''}_{i}}$, which is $\A^{''}_{i}$-invariant and such that $\A^{''}_{i} \setminus \W_{\A^{''}_{i}}$ is a manifold, endowed with a measure $\overline{dw}$ such that for every $\phi \in \mathscr{C}^{\infty}_{c}(\W)$ such that $\supp(\phi) \subseteq \W_{\A^{''}_{i}}$, 
\begin{equation*}
\displaystyle\int_{\W_{\A^{''}_{i}}} \phi(w)d\mu_{\W}(w) = \displaystyle\int_{\A^{''}_{i} \setminus \W_{\A^{''}_{i}}} \displaystyle\int_{\A^{''}_{i}} \phi(aw)da\overline{dw}.
\end{equation*}
For every $\varphi \in \mathscr{C}^{\infty}_{c}(\widetilde{\A^{'}_{i}})$, we denote by $\Chc(\varphi)$ the following integral:
\begin{equation*}
\Chc(\varphi) = \displaystyle\int_{\A^{''}_{i} \setminus \W_{\A^{''}_{i}}} \T(\varphi)(w) \overline{dw}.
\end{equation*}

\noindent According to Remark \ref{RemarkDistributionT}, the previous integral is well-defined and as proved in \cite[Lemma~2.9]{TOM1}, the corresponding map $\Chc: \mathscr{C}^{\infty}_{c}(\widetilde{\A^{'}_{i}}) \to \mathbb{C}$ defines a distribution on $\widetilde{\A^{'}_{i}}$.

\noindent For every $\tilde{h}_{i} \in \widetilde{\H_{i}}$, we denote by $\tau_{\tilde{h}_{i}}$ the map:
\begin{equation*}
\tau_{\tilde{h}_{i}}: \widetilde{\G'} \ni \tilde{g}' \to \tilde{h}\tilde{g}' \in \widetilde{\A^{'}_{i}}
\end{equation*}
and, for $\tilde{h}_{i}$ regular, by $\Chc_{\tilde{h}_{i}} = \tau^{*}_{\tilde{h}_{i}}(\Chc)$, where $\tau^{*}_{\tilde{h}_{i}}$ is the pull-back of $\tau_{\tilde{h}_{i}}$ as defined in \cite[Theorem~8.2.4]{HOR}. In particular, for every $\tilde{h}_{i} \in \widetilde{\H_{i}}^{\reg}$, $\Chc_{\tilde{h}_{i}}$ is a well-defined distribution on $\widetilde{\G'}$.

\section{Explicit formulas of $\Chc$ for unitary groups}

\label{SectionCauchyExplicit}

Let $\V = \mathbb{C}^{p+q}$ and $\V' = \mathbb{C}^{r+s}$ be two complex vector spaces endowed with non-degenerate bilinear forms $\left(\cdot, \cdot\right)$ and $\left(\cdot, \cdot\right)'$ respectively, with $\left(\cdot, \cdot\right)$ hermitian and $\left(\cdot, \cdot\right)'$ skew-hermitian, and let $(p, q)$ (resp. $(r, s)$) be the signature of $\left(\cdot, \cdot\right)$ (resp. $\left(\cdot, \cdot\right)'$). We assume that $p+q \leq r+s$. Let $\mathscr{B}_{\V} = \left\{e_{1}, \ldots, e_{n}\right\}$, $n=p+q$ (resp. $\mathscr{B}_{\V'} = \left\{e'_{1}, \ldots, e'_{n'}\right\}$, $n' = r+s$) be a basis of $\V$ (resp. $\V'$) such that $\Mat(\left(\cdot, \cdot\right), \mathscr{B}_{\V}) = \Id_{p, q}$ (resp. $\Mat(\left(\cdot, \cdot\right)', \mathscr{B}_{\V'}) = i\Id_{r, s}$).
Let $\G$ and $\G'$ be the corresponding group of isometries, i.e.
\begin{equation*}
\G = \G(\V, \left(\cdot, \cdot\right)) \approx \left\{g \in \GL(n, \mathbb{C}), \overline{g}^{t}\Id_{p, q}g = \Id_{p, q}\right\}, \qquad \G' = \G(\V', \left(\cdot, \cdot\right)') \approx \left\{g \in \GL(n', \mathbb{C}), \overline{g}^{t}\Id_{r, s}g = \Id_{r, s}\right\}.
\end{equation*}
where $\approx$ is a Lie group isomorphism.

\noindent Let $\H$ and $\H'$ be the diagonal compact Cartan subgroups of $\G$ and $\G'$ respectively. By looking at the action of $\H$ on the space $\V$, we get a decomposition of $\V$ of the form:
\begin{equation*}
\V = \V_{1} \oplus \ldots \oplus \V_{n},
\end{equation*}
where the spaces $\V_{a}$ given by $\V_{a} = \mathbb{C}ie_{a}$ are irreducible $\H$-modules. We denote by $\J$ the element of $\mathfrak{h}$ such that $\J = i\Id_{\V}$ and let $\J_{j} = i\E_{j, j}$. Similarly, we write
\begin{equation*}
\V' = \V'_{1} \oplus \ldots \oplus \V'_{n'},
\end{equation*}
with $\V'_{b} = \mathbb{C}ie'_{b}$, $\J'$ the element of $\mathfrak{h}'$ given by $\J' = i\Id_{\V'}$ and $\J'_{j} = i\E_{j, j}$. Let $\W = \Hom_{\mathbb{C}}(\V', \V)$ endowed with the symplectic form $\langle\cdot, \cdot\rangle$ given by:
\begin{equation*}
\langle w_{1}, w_{2}\rangle = \tr_{\mathbb{C}/\mathbb{R}}(w^{*}_{2}w_{1}), \qquad (w_{1}, w_{2} \in \W),
\end{equation*}
where $w^{*}_{2}$ is the element of $\Hom(\V, \V')$ satisfying:
\begin{equation*}
\left(w^{*}_{2}(v'), v\right) = \left(v', w_{2}(v)\right)' \qquad (v \in \V, v' \in \V').
\end{equation*}
The space $\W$ can be seen as a complex vector space by
\begin{equation}
\i w = \J \circ w \qquad (w \in \W).
\label{WComplexSpace}
\end{equation}
We define a double cover $\widetilde{\GL}_{\mathbb{C}}(\W)$ of the complex group $\GL_{\mathbb{C}}(\W)$ by:
\begin{equation*}
\widetilde{\GL}_{\mathbb{C}}(\W) = \left\{\tilde{g} = (g, \xi) \in \GL_{\mathbb{C}}(\W) \times \mathbb{C}^{\times}, \xi^{2} = \det(g)\right\}.
\end{equation*}
Because $p+q \leq r+s$, we get a natural embedding of $\mathfrak{h}_{\mathbb{C}}$ into $\mathfrak{h}'_{\mathbb{C}}$ and we denote by $\Z' = \G'^{\mathfrak{h}}$ the centralizer of $\mathfrak{h}$ in $\G'$. Let $\H'_{\mathbb{C}}$ be the complexification of $\H'$ in $\GL_{\mathbb{C}}(\W)$. In particular, $\H'_{\mathbb{C}}$ is isomorphic to 
\begin{equation*}
\mathfrak{h}'_{\mathbb{C}} / \left\{\sum\limits_{j=1}^{n'}2\pi x_{j}\J_{j}, x_{j} \in \mathbb{Z}\right\}.
\end{equation*}
We denote by $\check{\H}'_{\mathbb{C}}$ the connected two-fold cover of $\H_{\mathbb{C}}$ isomorphic to
\begin{equation*}
\mathfrak{h}'_{\mathbb{C}}/\left\{\sum\limits_{j=1}^{n'}2\pi x_{j}\J'_{j}, \sum\limits_{j=1}^{n'}x_{j} \in 2\mathbb{Z}, x_{j} \in \mathbb{Z}\right\}.
\end{equation*}
Let $p: \check{\H}'_{\mathbb{C}} \to \H_{\mathbb{C}}$ the corresponding covering map. If $\check{\H}'_{\mathbb{C}}$ is isomorphic to $\widetilde{\H}'_{\mathbb{C}}$, we may choose an isomorphism $\check{p}: \check{\H}'_{\mathbb{C}} \to \widetilde{\H}'_{\mathbb{C}}$ so that $p = \tilde{p} \circ \check{p}$. Otherwise, $\widetilde{\H}'_{\mathbb{C}}$ coincides with the direct product $\H_{\mathbb{C}, 1} \times \{\pm 1\}$. In this case, we can define $\check{p}: \check{\H}'_{\mathbb{C}} \to \tilde{\H}_{\mathbb{C}}$ to be the composition of $p$ with the inclusion $\H_{\mathbb{C}} \to \H_{\mathbb{C}} \times \{\pm 1\}$. Then, again $p = \tilde{p} \circ \check{p}$.

\begin{rema}
\begin{enumerate}
\item Let $\Psi' := \Psi'(\mathfrak{g}'_{\mathbb{C}}, \mathfrak{h}'_{\mathbb{C}})$ be a set of positive roots corresponding to $(\mathfrak{g}'_{\mathbb{C}}, \mathfrak{h}'_{\mathbb{C}})$, $\Psi'(\mathfrak{k}')$ be a the set of compact roots in $\Psi'$, where $\mathfrak{k}$ is the Lie algebra of $\K = \U(r) \times \U(s)$, and $\Psi'_{n}$ be the set of non-compact roots of $\Psi'$, i.e. $\Psi'_{n} = \Psi' \setminus \Psi'(\mathfrak{k})$. The reason why we are considering the double cover $\check{\H}'_{\mathbb{C}}$ of $\H'_{\mathbb{C}}$ is to make the form $\rho' = \frac{1}{2} \sum\limits_{\alpha \in \Psi'} \alpha$ analytic integral. For every analytic integral form $\gamma$ on $\mathfrak{h}'_{\mathbb{C}}$, we will denote by $\check{h}' \to \check{h}'^{\gamma}$ the corresponding character on $\H'_{\mathbb{C}}$.
\item We know that, up to conjugation, the number of Cartan subgroups in $\U(r, s)$ is $\min(r, s)+1$. Those Cartan subgroups can be parametrized by some particular subsets of $\Psi'_{n}$. Let $\Psi'^{\st}_{n}$ be the set of strongly orthogonal roots in $\Psi'$ (see \cite[Section~2]{SCH}).

\noindent For every $\alpha \in \Psi'^{\st}_{n}$, we denote by $c(\alpha)$ the element of $\G'_{\mathbb{C}}$ given by:
\begin{equation*}
c(\alpha) = \exp\left(\frac{\pi}{4}(X_{-\alpha} - X_{\alpha})\right).
\end{equation*}
where $X_{\alpha}$ (resp. $X_{-\alpha}$) is in $\mathfrak{g}'_{\mathbb{C}, \alpha}$ (resp. $\mathfrak{g}'_{\mathbb{C}, -\alpha}$) and normalized as in \cite[Equation~2.7]{SCH}. For every subset $\S$ of $\Psi'^{\st}_{n}$, we denote by $c(\S)$ the following element of $\G'_{\mathbb{C}}$ defined by
\begin{equation*}
c(\S) = \prod\limits_{\alpha \in \S} c(\alpha),
\end{equation*}
and let 
\begin{equation*}
\mathfrak{h}'(\S) = \mathfrak{g}' \cap \Ad(c(\S))(\mathfrak{h}'_{\mathbb{C}}).
\end{equation*}
We denote by $\H'(\S)$ the analytic subgroup of $\G'$ whose Lie algebra is $\mathfrak{h}'(\S)$. Then, $\H'(\S)$ is a Cartan subgroup of $\G'$ and one can prove that all the Cartan subgroups are of this form (up to conjugation).

\noindent For every $\S \subseteq \Psi'^{\st}_{n}$, we will denote by $\H'_{\S}$ the subgroup of $\H'_{\mathbb{C}}$ given by:
\begin{equation*}
\H'_{\S} = c(\S)^{-1}\H'(\S)c(\S).
\end{equation*}

\noindent Assume that $r \leq s$. Then, we define  $\Psi' = \left\{e_{i} - e_{j}, 1 \leq i < j \leq r+s\right\}$, where $e_{i}$ is the linear form on $\mathfrak{h}'_{\mathbb{C}} = \mathbb{C}^{r+s}$ given by $e_{i}(\lambda_{1}, \ldots, \lambda_{r+s}) = \lambda_{i}$. In this case, the set $\Psi'^{\st}_{n}$ is equal to $\{e_{t} - e_{r+t}, 1 \leq t \leq r\}$. In particular, $\H'(\emptyset) = \H'$ and if $\S_{t} = \{e_{1} - e_{r+1}, \ldots, e_{t} - e_{r+t}\}$, we get:
\begin{equation}
\H'_{\S_{t}} = \left\{h = \diag(e^{iX_{1} - X_{r+1}}, \ldots, e^{iX_{t} - X_{r+t}}, e^{iX_{t+1}}, \ldots, e^{iX_{r}}, e^{iX_{1} + X_{r+1}}, \ldots, e^{iX_{t} + X_{r+t}}, e^{iX_{r+t+1}}, \ldots, e^{iX_{r+s}}), X_{j} \in \mathbb{R}\right\}.
\label{DiagonalHSi}
\end{equation}

\label{ExampleUnitary}

\end{enumerate}
\end{rema}

\noindent Fix a subset $\S \in \Psi'^{\st}_{n}$. We denote by $\check{\H}'_{\S}$ the preimage of $\H'_{\S}$ in $\check{\H}'_{\mathbb{C}}$. For every $\varphi \in \mathscr{C}^{\infty}_{c}(\widetilde{\G'})$, we denote by $\mathscr{H}_{\S}\varphi$ the function of $\check{\H}'_{\S}$ defined by:
\begin{equation*}
\mathscr{H}_{\S}\varphi(\check{h}') = \varepsilon_{\Psi'_{\S, \mathbb{R}}}(\check{h}') \check{h}'^{\frac{1}{2} \sum\limits_{\alpha \in \Psi'} \alpha} \prod\limits_{\alpha \in \Psi'} (1 - \check{h}'^{-\alpha}) \displaystyle\int_{\G' / \H'(\S)} \varphi(g' c(\S) \check{p}(\check{h}') c(\S)^{-1} g'^{-1}) dg'\H'(\S) \qquad (\check{h}' \in \check{\H}'_{\S}),
\end{equation*}
where $\Psi'_{\S, \mathbb{R}}$ is the subset of $\Psi'$ consisting of real roots for $\H'_{\S}$ and $\varepsilon_{\Psi'_{\S, \mathbb{R}}}$ is the function defined on $\check{\H}'^{\reg}_{\S}$ by
\begin{equation*}
\varepsilon_{\Psi'_{\S, \mathbb{R}}}(\check{h}') = \sign\left(\prod\limits_{\alpha \in \Psi'_{\S, \mathbb{R}}} (1 - \check{h}'^{-\alpha})\right),
\end{equation*}

\noindent To simplify the notations, we denote by $\Delta_{\Psi'}(\check{h}')$ the quantity
\begin{equation*}
\Delta_{\Psi'}(\check{h}') = \check{h}'^{\frac{1}{2} \sum\limits_{\alpha \in \Psi'} \alpha} \prod\limits_{\alpha \in \Psi'} (1 - \check{h}'^{-\alpha}) \qquad (\check{h}' \in \check{\H}'_{\S}).
\end{equation*}
We define $\Delta_{\Phi'}$ similarly, where $\Phi' = -\Psi'$.

\begin{rema}

\begin{enumerate}
\item For every $\check{h}' \in \check{\H}'^{\reg}_{\S}$,
\begin{equation*}
\Delta_{\Phi'}(\check{h}')\Delta_{\Psi'}(\check{h}') = \prod\limits_{\alpha \in \Psi'^{+}}(1 - \check{h}'^{\alpha})(1 - \check{h}'^{-\alpha}) 
\end{equation*}
Note that if $\S = \emptyset$, we get for every $\alpha \in \Psi'$ and $\check{h}' \in \check{\H}'$ that $\overline{\check{h}'^{\alpha}} = \check{h}'^{-\alpha}$. In particular, $\Delta_{\Phi'}(\check{h}')\Delta_{\Psi'}(\check{h}') = \prod\limits_{\alpha \in \Psi'}(1 - \check{h}'^{\alpha})\overline{(1 - \check{h}'^{\alpha})} = \prod\limits_{\alpha \in \Psi'}|1 - \check{h}'^{\alpha}|^{2} = |\det(\Id - \Ad(\check{h}'))_{\mathfrak{g}'/\mathfrak{h}'}|$. Similarly, if $\S \neq \emptyset$, we get that for every $\alpha \in \Psi'$ and $\check{h}' \in \check{\H}'$, there exists $\beta \in \Phi'$, independant on $\check{h}'$, such that $\overline{\check{h}'^{\alpha}} = \check{h}'^{\beta}$. In particular, we get $\Delta_{\Phi'}(\check{h}')\Delta_{\Psi'}(\check{h}') = \prod\limits_{\alpha \in \Psi'}|1 - \check{h}'^{\alpha}|^{2}$.

\noindent For every $\check{h}' \in \check{\H}'_{\S}$, we denote by $|\Delta_{\G'}(\check{h}')|^{2} = \Delta_{\Phi'}(\check{h}')\Delta_{\Psi'}(\check{h}')$.
\item One can easily check that two Cartan subalgebras $\mathfrak{h}'(\S_{1})$ and $\mathfrak{h}'(\S_{2})$, with $\S_{1}, \S_{2} \subseteq \Psi'^{\st}_{n}$, are conjugate if and only if there exists an element of $\sigma \in \mathscr{W}$ sending $\S_{1} \cup (-\S_{1})$ onto $\S_{2} \cup (-\S_{2})$ (see \cite[Proposition~2.16]{SCH}). 
\end{enumerate}

\label{DeltaSquare}

\end{rema}

\noindent The Weyl's integration formula can be written with the previous notations as follows

\begin{prop}[Weyl's Integration Formula]

For every $\varphi \in \mathscr{C}^{\infty}_{c}(\widetilde{\G'})$, we get:
\begin{equation}
\displaystyle\int_{\widetilde{\G'}} \varphi(\tilde{g}') d\tilde{g}' = \sum\limits_{\S \in \Psi'^{\st}_{n}} m_{\S} \displaystyle\int_{\check{\H}'_{\S}} \varepsilon_{\Psi'_{\S, \mathbb{R}}}(\check{h}') \Delta_{\Phi'}(\check{h}') \mathscr{H}_{\S}\varphi(\check{h}')d\check{h}'.
\label{Equation3}
\end{equation}
where $m_{\S}$ are complex numbers. Here, the subsets $\S$ of $\Psi'^{\st}_{n}$ are defined up to equivalence (see Remark \ref{DeltaSquare}).

\end{prop}

\begin{proof}

See \cite[Section~2,~Page~3830]{TOM4}.

\end{proof}

\begin{rema}

In particular, if we fix $\S \subseteq \Psi'^{\st}_{n}$ and $\varphi \in \mathscr{C}^{\infty}_{c}(\widetilde{\G'})$ such that $\supp(\varphi) \subseteq \widetilde{\G'}\cdot\widetilde{\H'}(\S)^{\reg}$, the previous formula can be written as follow:
\begin{equation*}
\displaystyle\int_{\widetilde{\G'}} \varphi(\tilde{g}') d\tilde{g}' = m_{\S} \displaystyle\int_{\check{\H}'_{\S}} \varepsilon_{\Psi'_{\S, \mathbb{R}}}(\check{h}') \Delta_{\Phi'}(\check{h}') \mathscr{H}_{\S}\varphi(\check{h}')d\check{h}'.
\end{equation*}
\label{RemarkWeylIntegration2}
\end{rema}

\noindent Let $\H_{\mathbb{C}}, \G'_{\mathbb{C}} \subseteq \GL_{\mathbb{C}}(\W)$ the complexifications of $\H$ and $\G'$. We denote by $\G'^{\i}_{\mathbb{C}}$ the subgroup of $\G'_{\mathbb{C}}$ consisting of elements commuting with the element $\i$ introduced in Equation \eqref{WComplexSpace}.

\noindent As proved in \cite[Section~2]{TOM4}, the character $\Theta$ defined in \cite[Definition~4.23]{TOM6} extends to a rational function on $\widetilde{\H}_{\mathbb{C}}\cdot\widetilde{\G'^{\i}_{\mathbb{C}}}$ given by
\begin{equation*}
\Theta(\tilde{h}\tilde{g}') = (-1)^{u} \cfrac{\det^{\frac{1}{2}}(\tilde{h}\tilde{g}')}{\det(1 - hg')} \qquad (\tilde{h} \in \widetilde{\H}_{\mathbb{C}}, \tilde{g}' \in \widetilde{\G'^{\i}_{\mathbb{C}}}),
\end{equation*}
where $2u$ is the maximal dimension of a real subspace of $\W$ on which the symmetric form $\langle \J\cdot, \cdot\rangle$ is negative definite. More precisely, according to \cite[Proposition~2.1]{TOM4}, we get:

\begin{prop}

\label{PropositionTheta}

For every $\check{h} \in \check{\H}_{\mathbb{C}}$ and $\check{h}' \in \check{\H}'_{\mathbb{C}}$, we get:
\begin{equation*}
\det^{\frac{k}{2}}(\check{h})_{\W^{\mathfrak{h}}} \Delta_{\Psi}(\check{h})\Theta(\check{p}(\check{h})\check{p}(\check{h}'))\Delta_{\Phi}(\check{h}') = \sum\limits_{\sigma \in \mathscr{W}(\H'_{\mathbb{C}})} (-1)^{u+\alpha} \cfrac{\sign(\sigma)}{|\mathscr{W}(\Z'_{\mathbb{C}}, \H'_{\mathbb{C}})|}\cfrac{\det^{-\frac{k}{2}}(\sigma^{-1}(\check{h}'))_{\W^{\mathfrak{h}}} \Delta_{\Phi(\Z')}(\sigma^{-1}(\check{h}'))}{\det(1-p(h)p(h'))_{\sigma\W^{\mathfrak{h}}}},
\end{equation*}
where $\alpha \in \{0, -1\}$ depends only on the choice of the positive roots $\Psi$ and $\Phi'$, $k \in \{0, -1\}$ is defined by 
\begin{equation*}
k = \begin{cases} -1 & \text{ if } n'-n \in 2\mathbb{Z} \\ 0 & \text{ otherwise } \end{cases}
\end{equation*}
and $\W^{\mathfrak{h}}$ is the set of elements of $\W$ commuting with $\mathfrak{h}$.

\end{prop}

\begin{rema}

One can easily check that the space $\W^{\mathfrak{h}}$ is given by 
\begin{equation*}
\W^{\mathfrak{h}} = \sum\limits_{i=1}^{n} \Hom(\V'_{i}, \V_{i}).
\end{equation*}

\end{rema}

\noindent For every $\S \subseteq \Psi'^{\st}_{n}$, we denote by $\underline{\S}$ the subset of $\{1, \ldots, r+s\}$ given by $\underline{\S} = \left\{j, \exists \alpha \in \S \text{ such that } \alpha(\J'_{j}) \neq 0\right\}$. Let $\sigma \in \mathscr{S}_{n'}$ and $\S \subseteq \Psi'^{\st}_{n}$, we denote by $\Gamma_{\sigma, \S}$ the subset of $\mathfrak{h}'$ defined as
\begin{equation}
\Gamma_{\sigma, \S} = \left\{Y \in \mathfrak{h}', \langle Y\cdot, \cdot\rangle_{\sigma\W^{\mathfrak{h}} \cap \sum\limits_{j \notin \underline{\S}} \Hom(\V'_{j}, \V)} > 0\right\},
\label{SpaceEGamma}
\end{equation}
and let $\E_{\sigma, \S} = \widetilde{\exp}(i\Gamma_{\sigma, \S})$ the corresponding subset of $\widetilde{\H}'_{\mathbb{C}}$, where $\widetilde{\exp}$ is a choice of exponential map $\widetilde{\exp}: \mathfrak{h}'_{\mathbb{C}} \to \widetilde{\H}'_{\mathbb{C}}$ obtained by choosing an element $\widetilde{1}$ in $\pi^{-1}\left\{1\right\}$.

\begin{theo}

\label{TheoremTheta}

For every $\check{h} \in \check{\H} = \check{\H}_{\emptyset}$ and $\varphi \in \mathscr{C}(\widetilde{\G'})$, we get:
\begin{equation*}
\det^{\frac{k}{2}}(\check{h})_{\W^{\mathfrak{h}}} \Delta_{\Psi}(\check{h}) \displaystyle\int_{\widetilde{\G'}} \Theta(\check{p}(\check{h})\tilde{g}') \varphi(\tilde{g}')d\tilde{g}' =  
\end{equation*}
\begin{equation*}
\sum\limits_{\sigma \in \mathscr{W}(\H'_{\mathbb{C}})} \sum\limits_{\S \subseteq \Psi'^{\st}_{n}} \M_{\S}(\sigma)\lim\limits_{\underset{r \to 1}{r \in \E_{\sigma, \S}}} \displaystyle\int_{\check{\H}'_{\S}}\cfrac{\det^{-\frac{k}{2}}(\sigma^{-1}(\check{h}'))_{\W^{\mathfrak{h}}} \Delta_{\Phi'(\Z')}(\sigma^{-1}(\check{h}'))}{\det(1-p(h)rp(h'))_{\sigma\W^{\mathfrak{h}}}} \varepsilon_{\Phi'_{\S, \mathbb{R}}}(\check{h}') \mathscr{H}_{\S}(\varphi)(\check{h}') d\check{h}',
\end{equation*}
where $\M_{\S}(\sigma) = \cfrac{(-1)^{u}\varepsilon(\sigma)m_{\S}}{|\mathscr{W}(\Z'_{\mathbb{C}}, \H'_{\mathbb{C}})|}$.

\end{theo}

\noindent The theorem \ref{TheoremTheta} tells us how to compute $\Chc_{\tilde{h}}$ for an element $\tilde{h}$ in the compact Cartan $\H = \H(\emptyset)$. Using \cite{TOM2}, it follows that the value of $\Chc$ on the other Cartan subgroups can be computed explicitely by knowing how to do it for the compact Cartan (we will assume, without loss of generality, that $p \leq q$, in particular, the number of Cartan subgroups of $\G$, up to conjugation, is $p+1$).

\begin{nota}

\noindent For every $i \in \max(p, \min(r, s))$, we define the set $\S_{i}$ 
\begin{equation*}
\S_{i} = \begin{cases} \left\{e_{1} - e_{\alpha+1}, \ldots, e_{i} - e_{\alpha+i}\right\} & \text{ if } r \leq s \\ \left\{e_{1} - e_{\beta+1}, \ldots, e_{i} - e_{\beta+i}\right\} & \text{ otherwise } \end{cases},
\end{equation*}
where $\alpha = \begin{cases} p & \text{ if } r \leq p \\ r & \text{ otherwise } \end{cases}$ and $\beta = \begin{cases} p & \text{ if } s \leq p \\ s & \text{ otherwise } \end{cases}$.

\label{NotationsCartanSI}

\end{nota}

\noindent For every $i \in [|0, p|]$ and $j \in [|0, \min(r, s)|]$, we denote by $\H(\S_{i})$ and $\H'(\S_{j})$ the Cartan subgroups of $\G$ and $\G'$ respectively and let $\H(\S_{i}) = \T(\S_{i})\A(\S_{i})$ and $\H'(\S_{j}) = \T'(\S_{j})\A'(\S_{j})$ be the decompositions of $\H(\S_{i})$ and $\H'(\S_{j})$ as in \cite[Section~2.3.6]{WAL}. In particular, $\H(\S_{k}) = \H'(\S_{k})$ for every $k \in [|0, \min(p, \min(r, s))|]$.

\noindent To simplify, we assume that $r \leq s$. We denote by $\V_{0, i}$ the subspace of $\V$ on which $\A(\S_{i})$ acts trivially and by $\V_{1, i}$ the orthogonal complement of  $\V_{0, i}$ in $\V$. Let  $\V_{1, i} = \X_{i} \oplus \Y_{i}$ be a complete polarization of $\V_{1, i}$. We assume that we have a natural embedding of $\V_{1, i}$ into $\V'$ such that $\X_{i} \oplus \Y_{i}$ is a complete polarization with respect to $\left(\cdot, \cdot\right)'$ (i.e $i \leq r$). Let $\U_{i}$ be the orthogonal complement of $\V_{1, i}$ in $\V'$; in particular, we get a natural embedding:
\begin{equation*}
\GL(\X_{i}) \times \G(\U_{i}) \subseteq \G' = \U(r, s).
\end{equation*}
We denote by $\T_{1}(\S_{i})$ the maximal subgroup of $\T(\S_{i})$ which acts trivially on $\V_{0, i}$ and let $\T_{2}(\S_{i})$ the subgroup of $\T(\S_{i})$ such that $\T(\S_{i}) = \T_{1}(\S_{i}) \times \T_{2}(\S_{i})$ with $\T_{2}(\S_{i}) \subseteq \G(\V_{0, i})$. In particular,
\begin{equation}
\H(\S_{i}) = \T_{1}(\S_{i}) \times \A(\S_{i}) \times \T_{2}(\S_{i}).
\label{DecompositionHSi}
\end{equation}
Similarly, we get a decomposition of $\H'(\S_{i})$ os the form:
\begin{equation}
\H'(\S_{i}) = \T'_{1}(\S_{i}) \times \A'(\S_{i}) \times \T'_{2}(\S_{i}).
\label{DecompositionHSi2}
\end{equation}

\noindent Let $\eta(\S_{i})$ and $\eta'(\S_{i})$ be the nilpotent Lie subalgebras of $\mathfrak{u}(p, q)$ and $\mathfrak{u}(r, s)$ respectively given by
\begin{equation*}
\eta(\S_{i}) = \Hom(\X_{i}, \V_{0, i}) \oplus \Hom(\X_{i}, \Y_{i}) \cap \mathfrak{u}(p, q), \qquad \eta'(\S_{i}) = \Hom(\U_{i}, \X_{i}) \oplus \Hom(\X_{i}, \Y_{i}) \cap \mathfrak{u}(r, s).
\end{equation*}
We will denote by $\W_{0, i}$ the subspace of $\W$ defined by $\Hom(\U_{i}, \V_{0, i})$ and by $\P(\S_{i})$ and $\P'(\S_{i})$ the parabolic subgroups of $\G$ and $\G'$ respectively whose Levi factors $\L(\S_{i})$ and $\L'(\S_{i})$ are given by 
\begin{equation*}
\L(\S_{i}) = \GL(\X_{i}) \times \G(\V_{0, i}) \qquad \L'(\S_{i}) = \GL(\X_{i}) \times \G(\U_{i}),
\end{equation*}
and by $\N(\S_{i}) := \exp(\eta(\S_{i}))$ and $\N'(\S_{i}) := \exp(\eta'(\S_{i}))$ the unipotent radicals of $\P(\S_{i})$ and $\P'(\S_{i})$ respectively.

\begin{rema}

One can easily check that the forms on $\V_{0, i}$ and $\U_{i}$ have signature $(p-i, q-i)$ and $(r-i, s-i)$ respectively.

\end{rema}

\noindent As proved in \cite[Theorem~0.9]{TOM2}, for every $\tilde{h} = \tilde{t}_{1}\tilde{a}\tilde{t}_{2} \in \widetilde{\H}(\S_{i})^{\reg}$ (using the decomposition of $\H(\S_{i})$ given in Equation \eqref{DecompositionHSi}) and $\varphi \in \mathscr{C}^{\infty}_{c}(\widetilde{\G'})$, we get:
\begin{equation}
|\det(\Ad(\tilde{h}) - \Id)_{\eta(\S_{i})}| \Chc_{\tilde{h}}(\varphi) = 
\label{ChcNonCompact}
\end{equation}
\begin{equation*}
\C\d_{\S_{i}}(\tilde{h}) \varepsilon(\tilde{t}_{1}\tilde{a}) \displaystyle\int_{\GL(\X_{i})/\T_{1}(\S_{i}) \times \A(\S_{i})} \displaystyle\int_{\widetilde{\G}(\U_{i})} \varepsilon(\tilde{t}_{1} \tilde{a} \tilde{y}) \Chc_{\W_{0, i}}(\tilde{t}_{2}\tilde{y})\d'_{\S_{i}}(g\tilde{t}_{1}\tilde{a}g^{-1}\tilde{y}) \varphi^{\widetilde{\K'}}_{\widetilde{\N'}(\S_{i})}(g\tilde{t}_{1}\tilde{a}g^{-1}\tilde{y}) d\tilde{y} \overline{dg},
\end{equation*}
where $\C$ is a constant defined in \cite[Theorem~0.9]{TOM2}, $\varepsilon$ is the character defined in \cite[Lemma~6.3]{TOM2}, $\d_{\S_{i}}: \widetilde{\L}(\S_{i}) \to \mathbb{R}$ and $\d'_{\S_{i}}: \widetilde{\L'}(\S_{i}) \to \mathbb{R}$ are given by 
\begin{equation*}
\d_{\S_{i}}(\tilde{l}) = |\det(\Ad(\tilde{l})_{\eta(\S_{i})})|^{\frac{1}{2}}, \qquad  \d'_{\S_{i}}(\tilde{l}') = |\det(\Ad(\tilde{l}')_{\eta'(\S_{i})})|^{\frac{1}{2}}, \qquad (\tilde{l} \in \widetilde{\L}(\S_{i}), \tilde{l}' \in \widetilde{\L'}(\S_{i})),
\label{MapGamma}
\end{equation*}
and $\varphi^{\widetilde{\K'}}_{\widetilde{\N'}(\S_{i})}$ is the Harish-Chandra transform of $\varphi$, i.e. the function on $\widetilde{\L'}(\S_{i})$ defined by:
\begin{equation*}
\varphi^{\widetilde{\K'}}_{\widetilde{\N'}(\S_{i})}(\tilde{l}') = \displaystyle\int_{\widetilde{\N'}(\S_{i})} \displaystyle\int_{\widetilde{\K'}} \varphi(\tilde{k}\tilde{l}'\tilde{n}\tilde{k}^{-1}) d\tilde{k}d\tilde{n}, \qquad (\tilde{l}' \in \widetilde{\L'}(\S_{i})).
\end{equation*}
One can easily check that $(\G(\V_{0, i}), \G(\U_{i}))$ is an irreducible dual pair in $\Sp(\W_{0, i})$ of the same type of $\G, \G'$. Moreover, the element $\tilde{t}_{2}$ is contained in the compact Cartan of $\G(\V_{0, i})$. In particular, it follows from Theorem \ref{TheoremTheta} that the integral 
\begin{equation*}
\displaystyle\int_{\widetilde{\G}(\U_{i})} \varepsilon(\tilde{t}_{1}\tilde{a}\tilde{y}) \Chc_{\W_{0, i}}(\tilde{t}_{2}\tilde{y})\d'_{\S_{i}}(g\tilde{t}_{1}\tilde{a}g^{-1}\tilde{y}) \varphi^{\widetilde{\K'}}_{\widetilde{\N'}(\S_{i})}(g\tilde{t}_{1}\tilde{a}g^{-1}\tilde{y}) d\tilde{y}
\end{equation*}
can be seen as a finite sum of integrals, where the test function $\varphi$ is replaced by $\varepsilon(\tilde{y})\d'_{\S_{i}}(g\tilde{t}_{1}\tilde{a}g^{-1}\tilde{y}) \varphi^{\widetilde{\K'}}_{\widetilde{\N'}(\S_{i})}(g\tilde{t}_{1}\tilde{a}g^{-1}\tilde{y}), \tilde{y} \in \widetilde{\G}(\U_{i})$.

\begin{nota}

For every $j \in \{1, \ldots, p\}$ and $k \in \{0, \ldots, p-j\}$, we denote by $\S^{j}_{k} = \{e_{j+1} - e_{p+j+1}, \ldots, e_{j+k} - e_{p+j+k}\}$ the subset of $\Psi^{\st}_{n}(\mathfrak{g}(V_{0, j})_{\mathbb{C}}, \mathfrak{t}_{2}(\S_{j})_{\mathbb{C}})$ and by $\H(\S^{j}_{k})$ the corresponding Cartan subgroup of $\G(\V_{0, j})$. By convention, $\H(\S^{j}_{0}) = \T_{2}(\S_{j})$ is the compact Cartan subgroup of $\G(\V_{0, j})$.

\noindent Assume that $r \leq s$. For $j \in \{1, \ldots, r\}$, and $k \in \{0, \ldots, r-j\}$, we denote by $\S^{j}_{k} = \{e_{j+1} - e_{r+j+1}, \ldots, e_{j+k} - e_{r+j+k}\}$ the subset of $\Psi'^{\st}_{n}(\mathfrak{g}(\U_{j})_{\mathbb{C}}, \mathfrak{t}'_{2}(\S_{j})_{\mathbb{C}})$ and by $\H'(\S^{j}_{k})$ the corresponding Cartan subgroup of $\G(\U_{j})$. By convention, $\H'(\S^{j}_{0}) = \T'_{2}(\S_{j})$ is the compact Cartan subgroup of $\G(\U_{j})$.

\label{NotationGUi}

\end{nota}

\begin{rema}

Let $i \in \{0, \ldots, \min(p, r)\}$ and $j \leq i$. We denote by $\X_{j}, \Y_{j}, \X_{i}, \Y_{i}$ the subspaces of $\V$ as before. There exists subspaces $\widetilde{\X}_{j, i}$ and $\widetilde{\Y}_{j, i}$ such that $\X_{i} = \X_{j} \oplus \widetilde{\X}_{j, i}$ and $\Y_{i} = \Y_{j} \oplus \widetilde{\X}_{j, i}$. In particular, $\U_{j} = \U_{i} \oplus \widetilde{\X}_{j, i} \oplus \widetilde{\Y}_{j, i}$.

\noindent The Cartan subgroup $\H'(\S_{i})$ is included in Levi factor $\L'(\S_{j})$ of the parabolic $\P'(\S_{j})$ of $\G$. Let $\H'(\S^{j}_{i-j})$ be the Cartan subgroup of $\G(\U_{j})$ as in Notation \ref{NotationGUi}. As in Equation \eqref{DecompositionHSi}, we have:
\begin{equation*}
\H'(\S^{j}_{i-j}) = \T'_{1}(\S^{j}_{i-j}) \times \A'(\S^{j}_{i-j}) \times \T'_{2}(\S^{j}_{i-j}),
\end{equation*}
and then we get the following decomposition of $\H'(\S_{i})$:
\begin{equation}
\H'(\S_{i}) = \underbrace{\T'_{1}(\S_{i}) \times \A'(\S_{i})}_{\subseteq \GL(\X_{i})} \times \underbrace{\T'_{2}(\S_{i})}_{\subseteq \G(\U_{i})} = \underbrace{\underbrace{\T'_{1}(\S_{j}) \times \A'(\S_{j})}_{\subseteq \GL(\X_{j})} \times \underbrace{\underbrace{\T'_{1}(\S^{j}_{i-j}) \times \A'(\S^{j}_{i-j})}_{\subseteq \GL(\widetilde{\X}_{j, i})} \times \underbrace{\T'_{2}(\S_{i})}_{\subseteq \G(\U_{i})}}_{\subseteq \GL(\widetilde{\X}_{j, i}) \times \G(\U_{i}) \subseteq \G(\U_{j})}}_{\subseteq \GL(\X_{j}) \times \G(\U_{j}) = \L'(\S_{j})}.
\label{DecompositionHiandHj}
\end{equation}

\noindent Finally, one can see easily that $\T'_{2}(\S_{i}) = \T'_{2}(\S^{j}_{i-j})$ and then, 
\begin{equation*}
\H'(\S_{i}) = \T'_{1}(\S_{j}) \times \A'(\S_{j}) \times \H'(\S^{j}_{i-j}).
\end{equation*}

\label{RemarkDecompositionHSij}

\end{rema}

\noindent We finish this section with a technical lemma which will be useful in Section \ref{ProofConjectureDS}.

\begin{lemme}

For every $f \in \mathscr{C}^{\infty}_{c}(\widetilde{\G'})$ and $\tilde{h}' \in \widetilde{\H'}(\S_{i})$, we get:
\begin{equation*}
\displaystyle\int_{\GL(\X_{j})/\T'_{1}(\S_{j}) \times \A'(\S_{j})} \displaystyle\int_{\G(\U_{j})/\H'(\S^{j}_{i-j})} f^{\widetilde{\K'}}_{\widetilde{\N'}(\S_{j})}(g_{1}g_{2}\tilde{h}'g^{-1}_{2}g^{-1}_{1}) \overline{dg_{2}} \overline{dg_{1}} 
\end{equation*}
\begin{equation*}
= \cfrac{\D_{\L'(\S_{i})}(\tilde{h})}{\D_{\L'(\S_{j})}(\tilde{h})} \displaystyle\int_{\GL(\X_{i})/\T'_{1}(\S_{i}) \times \A'(\S_{i})} \displaystyle\int_{\G(\U_{i})/\T'_{2}(\S_{i})} f^{\widetilde{\K'}}_{\widetilde{\N'}(\S_{i})}(g_{1}g_{2}\tilde{h}'g^{-1}_{2}g^{-1}_{1}) \overline{dg_{2}} \overline{dg_{1}},
\end{equation*}
where $\D_{\L'(\S_{j})}$ and $\D_{\L'(\S_{i})}$ are given by:
\begin{equation*}
\D_{\L'(\S_{j})}(\tilde{h}') = |\det(\Id - \Ad(\tilde{h}')^{-1})_{\mathfrak{l}'(\S_{j})/\mathfrak{h}'(\S_{j})}|^{\frac{1}{2}}, \qquad \D_{\L'(\S_{i})}(\tilde{h}') = |\det(\Id - \Ad(\tilde{h}')^{-1})_{\mathfrak{l}'(\S_{i})/\mathfrak{h}'(\S_{i})}|^{\frac{1}{2}}.
\end{equation*}

\label{IntegrationLemmaProposition67}

\end{lemme}

\begin{proof}

As explained in \cite[Appendix~A]{TOM2}, we have:
\begin{eqnarray*}
\displaystyle\int_{\widetilde{\G'}/\widetilde{\H'}(\S_{i})} f(g\tilde{h}g^{-1})\overline{dg} & = & \cfrac{\D_{\L'(\S_{i})}(\tilde{h})}{\D_{\L'(\S_{0})}(\tilde{h})} \displaystyle\int_{\L'(\S_{i})/\H'(\S_{i})} f^{\widetilde{\K'}}_{\widetilde{\N'}(\S_{i})}(l'\tilde{h}l'^{-1}) \overline{dl'} \\ 
 & = & \cfrac{\D_{\L'(\S_{i})}(\tilde{h})}{\D_{\L'(\S_{0})}(\tilde{h})} \displaystyle\int_{\GL(\X_{i})/\T'_{1}(\S_{i}) \times \A'(\S_{i})} \displaystyle\int_{\G(\U_{i}) / \T_{2}(\S_{i})} f^{\widetilde{\K'}}_{\widetilde{\N'}(\S_{i})}(g_{1}g_{2}\tilde{h}g^{-1}_{2}g^{-1}_{1})\overline{dg_{2}}\overline{dg_{1}}
\end{eqnarray*}
where $\D_{\L'(\S_{0})}(\tilde{h}) = \D_{\G'}(\tilde{h}') = |\det(\Ad(\tilde{h})^{-1} - \Id)_{\mathfrak{g}'/\mathfrak{h}'(\S_{i})}|^{\frac{1}{2}}$. Similarly, using that $\H'(\S_{i}) \subseteq \P'(\S_{j})$, we get:
\begin{eqnarray*}
\displaystyle\int_{\widetilde{\G'}/\widetilde{\H'}(\S_{j})} f(g\tilde{h}g^{-1})\overline{dg} & = & \cfrac{\D_{\L'(\S_{j})}(\tilde{h})}{\D_{\L'(\S_{0})}(\tilde{h})} \displaystyle\int_{\L'(\S_{j})/\H'(\S_{i})} f^{\widetilde{\K'}}_{\widetilde{\N'}(\S_{j})}(l'\tilde{h}l'^{-1}) \overline{dl'} \\ 
 & = & \cfrac{\D_{\L'(\S_{j})}(\tilde{h})}{\D_{\L'(\S_{0})}(\tilde{h})} \displaystyle\int_{\GL(\X_{j})/\T'_{1}(\S_{j}) \times \A'(\S_{j})} \displaystyle\int_{\G(\U_{j}) / \H'(\S^{j}_{i-j})} f^{\widetilde{\K'}}_{\widetilde{\N'}(\S_{j})}(g_{1}g_{2}\tilde{h}g^{-1}_{2}g^{-1}_{1})\overline{dg_{2}}\overline{dg_{1}},
\end{eqnarray*}
and the lemma follows.

\end{proof}

\section{Transfer of invariant eigendistributions}

\label{TransferEigenDistributions}

We start this section by recalling the notion of invariant eigendistributions. We keep the notations of Appendix \ref{AppendixA}. Let $\G$ be a connected real reductive Lie group, $\mathscr{D}'(\G)$ be the space of distributions of $\G$, i.e. the continuous linear forms on $\mathscr{C}^{\infty}_{c}(\G)$ and $\D^{\G}_{\G}(\G)$ the space of bi-invariant differential operators on $\G$ as in Notations \ref{NotationDGG}.

\noindent For every $f \in \mathscr{C}^{\infty}_{c}(\G)$ and $g \in \G$, we denote by $f^{g}$ the function of $\mathscr{C}^{\infty}_{c}(\G)$ defined by $f^{g}(x) = f(gxg^{-1}), x \in \G$. We say that $T \in \mathscr{D}'(\G)$ is a $\G$-invariant distribution if $T(f^{g}) = T(f)$ for every $f \in \mathscr{C}^{\infty}_{c}(\G)$ and $g \in \G$.

\begin{defn}

A distribution $T$ on $\G$ is an eigendistribution if there exists an algebra homomorphism $\chi_{T}: \D^{\G}_{\G}(\G) \to \mathbb{C}$ such that $D(T) = \chi_{T}(D)T$ for every $D \in \D^{\G}_{\G}(\G)$.

\end{defn}

\noindent As proved by Harish-Chandra (see \cite[Theorem~2]{HAR4}), for every invariant eigendistribution $T$ on $\G$, there exists a locally integrable function $f_{T}$ on $\G$ which is analytic on $\G^{\reg}$ such that $T = f_{T}$, i.e. for every $\varphi \in \mathscr{C}^{\infty}_{c}(\G)$, 
\begin{equation*}
T(\varphi) = \displaystyle\int_{\G} f_{T}(g)\varphi(g)dg.
\end{equation*}

\begin{rema}
\begin{enumerate}
\item Using the isomorphism defined in Appendix \ref{AppendixA}, Theorem \ref{IsomorphismEnveloppingDGG}, an eigendistribution $T$ is an invariant distribution such that there exists a character $\chi_{T}$ of $\Z(\mathscr{U}(\mathfrak{g}_{\mathbb{C}}))$ such that $zT = \chi_{T}(z)T$ for every $z \in \Z(\mathscr{U}(\mathfrak{g}_{\mathbb{C}}))$.
\item Let $(\Pi, \mathscr{H})$ be a representation of $\G$. Following \cite{HAR1}, we say that the representation $\Pi$ is permissible if $\Pi(z)$ is a scalar multiple of the unit operator for every $z \in \Z(\G) \cap \D$, where $\D$ is the analytic subgroup of $\G$ corresponding to $\Z(\mathfrak{k})$ ($\mathfrak{k}$ being the Lie algebra of a maximal compact subgroup $\K$ of $\G$). A permissible representation is said quasi-simple if there exists an homomorphism $\chi$ of $\Z(\mathscr{U}(\mathfrak{g}_{\mathbb{C}}))$ into $\mathbb{C}$ such that $d\Pi(z)(\eta) = \chi(z)\eta$ for every $z \in \Z(\mathscr{U}(\mathfrak{g}_{\mathbb{C}}))$ and $\eta$ in the Garding space $\Gar(\Pi, \mathscr{H})$ (for the definition of $\Gar(\Pi, \mathscr{H})$, see \cite[Part~II]{HAR1}). In particular, for such representations, Harish-Chandra proved that for every $\varphi \in \mathscr{C}^{\infty}_{c}(\G)$, the operator $\Pi(\varphi)$ is a trace class operator (see \cite[Section~5]{HAR3}) and the corresponding map $\Theta_{\Pi}: \mathscr{C}^{\infty}_{c}(\G) \ni \varphi \to \tr(\Pi(\varphi)) \in \mathbb{C}$ is a distribution in the sense of Laurent Schwartz (see \cite[Section~5]{HAR3}); the map $\Theta_{\Pi}$ is called the global character of $\Pi$. Using that $\Theta_{\Pi}$ is an invariant eigendistribution, it follows that there exists a locally integrable function $\Theta_{\Pi}$ on $\G$, analytic on $\G^{\reg}$, such that 
\begin{equation*}
\Theta_{\Pi}(\varphi) = \displaystyle\int_{\G} \Theta_{\Pi}(g) \varphi(g) dg, \qquad (\varphi \in \mathscr{C}^{\infty}_{c}(\G)).
\end{equation*}
The function $\Theta_{\Pi}$ is the character of $\Pi$. As proved in \cite{MAUT}, every irreducible unitary representation is quasi-simple, in particular, every discrete series representations (see Section \ref{SectionDiscreteSeries}) has a character, whose value on $\H^{\reg}$ is explicit (see Theorem \ref{TheoremDiscreteSeries}).
\end{enumerate}
\end{rema}

\begin{nota}

\noindent For every reductive group $\G$, we denote by $\mathscr{I}(\G)$ the space of orbital integrals on $\G$ as in \cite[Section~3]{BOU}. Roughly speaking, the set $\mathscr{I}(\G)$ is a subset of $\mathscr{C}^{\infty}(\G^{\reg})^{\G}$ satisfying 4 conditions (see \cite[Pages~579-580]{BOU}). This space is endowed with a natural topology (see \cite[Section~3.3]{BOU}). We denote by $\J_{\G}$ the map $\J_{\G}: \mathscr{C}^{\infty}_{c}(\G) \to \mathscr{I}(\G)$ given as follow: for every $\gamma \in \G^{\reg}$, there exists a unique, up to conjugation, Cartan subgroups $\H(\gamma)$ of $\G$ such that $\gamma \in \H(\gamma)$, and for every $\varphi \in \mathscr{C}^{\infty}_{c}(\G)$, we define $\J_{\G}(\varphi)(\gamma)$ by:
\begin{equation*}
\J_{\G}(\varphi)(\gamma) = |\det(\Id - \Ad(\gamma^{-1}))_{\mathfrak{g}/\mathfrak{h}(\gamma)}|^{\frac{1}{2}} \displaystyle\int_{\G/\H(\gamma)} \varphi(g\gamma g^{-1}) \overline{dg}.
\end{equation*}

\end{nota}

\noindent As proved in \cite[Theorem~3.2.1]{BOU}, the map:
\begin{equation*}
\J_{\G}: \mathscr{C}^{\infty}_{c}(\G) \to \mathscr{I}(\G)
\end{equation*}
is well-defined and surjective. We denote by $\mathscr{I}'(\G)$ the set of continuous linear forms on $\mathscr{I}(\G)$ and by $\J^{t}_{\G}: \mathscr{I}'(\G) \to \mathscr{D}'(\G)$ the transpose of $\J_{\G}$ defined by
\begin{equation*}
\J^{t}_{\G}(T)(\varphi) = T(\J_{\G}(\varphi)), \qquad (T \in \mathscr{I}'(\G), \varphi \in \mathscr{C}^{\infty}_{c}(\G)).
\end{equation*}
As proved in \cite[Theorem~3.2.1]{BOU}, $\J^{t}_{\G}(T)$ is a $\G$-invariant distribution on $\G$ and the corresponding map:
\begin{equation*}
\J^{t}_{\G}: \mathscr{I}'(\G) \to \mathscr{D}'(\G)^{\G}
\end{equation*}
is bijective.

\bigskip

\noindent Let $(\G, \G')$ be an irreducible dual pair in $\Sp(W)$ such that $\rk(\G) \leq \rk(\G')$ and $(\mathscr{I}(\widetilde{\G}), \J_{\widetilde{\G}})$, $(\mathscr{I}(\widetilde{\G'}), \J_{\widetilde{\G'}})$ be the corresponding space of orbital integrals on $\widetilde{\G}$ and $\widetilde{\G'}$ respectively. To simplify, we assume that both $\G$ and $\G'$ are connected). For every function $\varphi \in \mathscr{C}^{\infty}_{c}(\widetilde{\G'})$, we denote by $\Chc(\varphi)$ the $\widetilde{\G}$-invariant function on $\widetilde{\G}^{\reg}$ given by:
\begin{equation*}
\Chc(\varphi)(\tilde{h}_{i}) = \Chc_{\tilde{h}_{i}}(\varphi), \qquad (\tilde{h}_{i} \in \widetilde{\H_{i}}^{\reg}).
\end{equation*}
In \cite{TOM4}, the authors proved the following results:

\begin{theo}

For every $\varphi \in \mathscr{C}^{\infty}_{c}(\widetilde{\G'})$, $\Chc(\varphi) \in \mathscr{I}(\widetilde{\G})$ and the corresponding map 
\begin{equation*}
\Chc: \mathscr{C}^{\infty}_{c}(\widetilde{\G'}) \to \mathscr{I}(\widetilde{\G})
\end{equation*}
is continuous. Moreover, if $\J_{\widetilde{\G'}}(\varphi) = 0$, we get that $\Chc(\varphi) = 0$, i.e. the map $\Chc: \mathscr{C}^{\infty}_{c}(\widetilde{\G'}) \to \mathscr{I}(\widetilde{\G})$ factors through $\mathscr{I}(\widetilde{\G'})$ and we get a transfer of orbital integrals:
\begin{equation*}
\Chc: \mathscr{I}(\widetilde{\G'}) \to \mathscr{I}(\widetilde{\G}).
\end{equation*}

\end{theo}

\noindent By dualizing the previous map, we get $\Chc^{t}: \mathscr{I}'(\widetilde{\G}) \to \mathscr{I}'(\widetilde{\G'})$ given by
\begin{equation*}
\Chc^{t}(\tau)(\phi) = \tau(\Chc(\phi)) \qquad (\tau \in \mathscr{I}'(\widetilde{\G}), \phi \in \mathscr{I}(\widetilde{\G'})).
\end{equation*}
By using the isomorphisms $\J^{t}_{\widetilde{\G}}$ and $\J^{t}_{\widetilde{\G'}}$, we get a map $\Chc^{*}: \mathscr{D}'(\widetilde{\G})^{\widetilde{\G}} \to \mathscr{D}'(\widetilde{\G'})^{\widetilde{\G'}}$ given by $\Chc^{*} = \J^{t}_{\widetilde{\G'}} \circ \Chc^{t} \circ (\J^{t}_{\widetilde{\G}})^{-1}$.

\noindent We denote by $\Eig(\widetilde{\G})$ (resp. $\Eig(\widetilde{\G'}))$ the set of invariant eigendistributions on $\widetilde{\G}$ (resp. $\widetilde{\G'}$).

\begin{theo}

The map $\Chc^{*}: \mathscr{D}'(\widetilde{\G})^{\widetilde{\G}} \to \mathscr{D}'(\widetilde{\G'})^{\widetilde{\G'}}$ sends $\Eig(\widetilde{\G})^{\widetilde{\G}}$ into $\Eig(\widetilde{\G'})^{\widetilde{\G'}}$.

\label{TheoremTransferOfEigendistributions}

\end{theo}

\begin{rema}

If $\Theta$ is a distribution on $\widetilde{\G}$ given by a locally integrable function $\Theta$ on $\widetilde{\G'}$, we get for every $\varphi \in \mathscr{C}^{\infty}_{c}(\widetilde{\G'})$ the following equality:
\begin{equation*}
\Chc^{*}(\Theta)(\varphi) = \sum\limits_{i=1}^{n} \cfrac{1}{|\mathscr{W}(\H_{i})|} \displaystyle\int_{\widetilde{\H_{i}}^{\reg}} \Theta(\tilde{h}_{i})|\det(1-\Ad(\tilde{h}^{-1}_{i}))_{\mathfrak{g}/\mathfrak{h}_{i}}| \Chc(\varphi)(\tilde{h}_{i}) d\tilde{h}_{i}.
\end{equation*}
where $\H_{1}, \ldots, \H_{n}$ is a maximal set of non-conjugate Cartan subgroups of $\G$.

\label{ValueChcThetaPi}

\end{rema}

\noindent We recall the following conjecture.

\begin{conj}

Let $\G_{1}$ and $\G'_{1}$ be the Zariski identity components of $\G$ and $\G'$ respectively. Let $\Pi \in \mathscr{R}(\widetilde{\G}, \omega)$ satisfying ${\Theta_{\Pi}}_{|_{\widetilde{\G}/\widetilde{\G_{1}}}} = 0$ if $\G = \O(\V)$, where $\V$ is an even dimensional vector space over $\mathbb{R}$ or $\mathbb{C}$. Then, up to a constant, $\Chc^{*}(\Theta_{\Pi}) = \Theta_{\Pi'_{1}}$ on $\widetilde{\G'_{1}}$.

\label{ConjectureHC}

\end{conj}

\noindent In few cases, the conjecture is well-known: if $\G$ is compact (see \cite{TOM1}) and if $(\G, \G')$ is in the stable range (see \cite{TOM5}). In this paper, we are investigating the case $\rk(\G) = \rk(\G')$, with $\Pi$ a discrete series representation of $\widetilde{\G}$. We will focus our attention on the dual pair of unitary groups satisfying $\rk(\G) = \rk(\G')$, using some results of A. Paul that we recall in the next section.

\section{Discrete series representations and a result of A. Paul}

\label{SectionDiscreteSeries}

Let $\G$ be a connected real reductive Lie group.

\begin{defn}

We say that an irreducible representation $(\Pi, (\mathscr{H}, \langle\cdot,\cdot\rangle))$ is a discrete series representation if all the functions $\tau_{u, v}, u, v \in \mathscr{H}$, are in $\L^{2}(\G)$, where 
\begin{equation*}
\tau_{u, v}: \G \ni g \to \langle g(u), v \rangle \in \mathbb{C}.
\end{equation*}

\end{defn}

\begin{rema}

One can prove that the condition given in the previous definition is equivalent to say that the representation $(\Pi, \mathscr{H})$ is equivalent with a direct summand of the right regular representation of $\G$ on $\L^{2}(\G)$.

\noindent Moreover, as recalled in \cite[Section~9.3]{KNA2}, for such a representation $(\Pi, \mathscr{H})$, there exists a positive number $d_{\Pi}$ ( depending on the Haar measure $dg$ on $\G$), called the formal degree of $\Pi$, such that for every $u_{1}, u_{2}, v_{1}, v_{2} \in \mathscr{H}$,
\begin{equation*}
\displaystyle\int_{\G} \langle\Pi(g)u_{1}, v_{1}\rangle\overline{\langle\Pi(g)u_{2}, v_{2}\rangle}dg = \cfrac{\langle u_{1}, u_{2}\rangle \overline{\langle v_{1}, v_{2}} \rangle}{d_{\Pi}}.
\end{equation*}

\label{RemarkFormalDegree}

\end{rema}

\noindent In his papers \cite{HAR5} and \cite{HAR2}, Harish-Chandra gave a classification of the discrete series representations of $\G$. First of all, he proved that $\G$ has discrete series if and only if $\G$ has a compact Cartan subgroup (see \cite[Theorem~13]{HAR2}). Let $\K$ be a maximal compact subgroup of $\G$ and $\H$ a Cartan subgroup of $\K$. He also proved that the set of discrete series is indexed by a lattice of $i\mathfrak{h}^{*}$. We say few words about this now. Let $\Psi = \Psi(\mathfrak{g}_{\mathbb{C}}, \mathfrak{h}_{\mathbb{C}})$ be the set of roots of $\mathfrak{g}$, $\Psi(\mathfrak{k}) = \Psi(\mathfrak{k}_{\mathbb{C}}, \mathfrak{h}_{\mathbb{C}})$ be the set of compact roots of $\mathfrak{g}$, $\rho = \frac{1}{2} \sum\limits_{\alpha \in \Psi^{+}} \alpha$ and $\rho(\mathfrak{k}) = \frac{1}{2} \sum\limits_{\alpha \in \Phi^{+}(\mathfrak{k})} \alpha$.

\begin{nota}

For every $g \in \G$, we denote by $\D_{g}$ the function on $\mathbb{R}$ given by
\begin{equation*}
\D_{g}(t) = \det((t+1)\Id_{\mathfrak{g}} - \Ad(g)) \qquad (t \in \mathbb{R}).
\end{equation*}
In particular, $\D_{g}(t) = \sum\limits_{i=0}^{n} t^{i}\D_{i}(g)$, with $n = \dim(\G)$. The $\D_{i}'s$ are analytic on $\G$ and let $l$ be the least integer such that $\D_{l} \neq 0$. The integer $l$ is the rank of $\mathfrak{g}$. We denote by $\D(g)$ the coefficient of $t^{l}$ in the previous polynomial and by $\G^{\reg}$ the set of $g \in \G$ such that $\D(g) \neq 0$.

\label{NotationsD}

\end{nota}

\begin{theo}

Let $\lambda$ be an element of $i\mathfrak{h}^{*}$ such that $\lambda + \rho$ is analytic integral. Then, there exists a discrete series representation $(\Pi_{\lambda}, \mathscr{H}_{\lambda})$ of $\G$ such that:
\begin{enumerate}
\item The representation $\Pi_{\lambda}$ has infinitesimal character $\chi_{\lambda}$ as in Remark \ref{RemarkInfinitesimalCharacter},
\item The linear form $\nu = \lambda + \rho - 2\rho(\mathfrak{k})$ is the highest weight of the lowest $\K$-type for ${\Pi_{\lambda}}_{|_{\K}}$ and the multiplicity of the corresponding representation $\Pi_{\nu}$ in ${\Pi_{\lambda}}_{|_{\K}}$ is one.
\end{enumerate}
The parameter $\lambda$ is called the Harish-Chandra parameter of $\Pi_{\lambda}$. Moreover, if we denote by $\Theta_{\lambda}$ the distribution character of $\Pi$ and by $\Theta_{\lambda}$ the corresponding locally integrable function on $\G^{\reg}$, we get that the restriction of $\Theta_{\lambda}$ of $\Pi$ to $\H^{\reg}$ is given by the following formula
\begin{equation*}
\Theta_{\lambda}(\exp(X)) = (-1)^{\frac{\dim(\G) - \dim(\K)}{2}}\sum\limits_{w \in \mathscr{W}(\mathfrak{k})} \varepsilon(w) \cfrac{e^{(w\lambda)(X)}}{\prod\limits_{\alpha > 0} (e^{\frac{\alpha(X)}{2}} - e^{-\frac{\alpha(X)}{2}})}, \qquad (X \in \mathfrak{h}^{\reg}).
\end{equation*}

\label{TheoremDiscreteSeries}

\end{theo}

\begin{rema}

As proved in \cite{HAR5}, for every discrete series $\Pi$ of $\G$ with Harish-Chandra parameter $\lambda$, we get:
\begin{equation*}
\sup\limits_{g \in \G^{\reg}} |\D(g)|^{\frac{1}{2}} |\Theta_{\lambda}(g)| < \infty.
\end{equation*}

\label{RemarkDiscreteSeries}

\end{rema}

\noindent The previous properties of $\Theta_{\lambda}$ characterize the discrete series characters inside the space of invariant distributions of $\G$. More precisely, as proved in \cite[Lemma~44]{HAR5}, we have the following result.

\begin{theo}

Let $\Theta_{\lambda}$ be $\G$-invariant distribution on $\G$ such that:
\begin{enumerate}
\item $z\Theta_{\lambda} = \gamma(\lambda)(z)\Theta_{\lambda}, z \in \Z(\mathscr{U}(\mathfrak{g}_{\mathbb{C}}))$,
\item $\sup\limits_{g \in \G^{\reg}} |\D(g)|^{\frac{1}{2}} |\Theta_{\lambda}(g)| < \infty$,
\item $\Theta_{\lambda} = 0$ pointwise on $\H^{\reg}$ .
\end{enumerate}
Then, $\Theta_{\lambda} = 0$.

\label{TheoremDiscreteSeries2}

\end{theo}

\noindent The previous theorem will be central for us in Section \ref{ProofConjectureDS} to prove the conjecture \ref{ConjectureHC} for discrete series representations in the equal rank case. We now recall a key result of A. Paul for unitary groups. Let $(\G, \G') = (\U(p, q), \U(r, s))$ be a dual pair of unitary groups in $\Sp(2(p+q)(r+s), \mathbb{R})$. As explained in \cite[Section~1.2]{PAUL}, the double cover of $\widetilde{\U}(p, q)$ is isomorphic to
\begin{equation}
\widetilde{\U}(p, q) \approx \left\{(g, \xi) \in \U(p, q) \times \mathbb{C}^{*}, \xi^{2} = \det(g)^{r-s}\right\}.
\label{NatureDoubleCoverU}
\end{equation}
In particular, all the genuine admissible representations of $\widetilde{\U}(p, q)$ are the form $\Pi \otimes \det^{\frac{r-s}{2}}$, where $\det^{\frac{r-s}{2}}$ is the genuine character of $\widetilde{\U}(p, q)$ given by $\det^{\frac{r-s}{2}}(g, \xi) = \xi$ and $\Pi$ is an admissible representation of $\U(p, q)$. From now on, we fix $p$ and $q$ and let $r$ and $s$ vary under the condition that $p+q = r+s$. In particular, under this condition, it follows from Equation \eqref{NatureDoubleCoverU} that the double cover $\U(p, q)$ stays the same when $r$ and $s$ vary.

\noindent In \cite[Section~6]{PAUL}, A. Paul proved the following theorem:

\begin{theo}

For every genuine irreducible admissible representation $(\Pi, \mathscr{H}_{\Pi})$ of $\widetilde{\U}(p, q)$, there exists a unique pair of integers $(r, s) = (r_{\Pi}, s_{\Pi})$ such that $p+q = r+s$ with $\theta_{r, s}(\Pi) \neq 0$.

\end{theo}
 
\noindent She also obtained more precise results for discrete series representations (see \cite[Theorem~6.1]{PAUL} or \cite[Theorem~2.7]{PAUL2}).

\begin{nota}

 We fix a basis $\{e_{1}, \ldots, e_{n}\}$ of $\mathfrak{h}^{*}$. In particular, every linear form $\lambda$ on $\mathfrak{h}$ can be written as $\lambda = \sum\limits_{i=1}^{n} \lambda_{i}e_{i}$ or also as $\lambda = (\lambda_{1}, \ldots, \lambda_{n})$.
 
 \end{nota}

\begin{theo}

Let $\Pi$ be a discrete series representation of $\widetilde{\U}(p, q)$, the corresponding representation $\theta_{r_{\Pi}, s_{\Pi}}(\Pi)$ is a discrete series representation of $\widetilde{\U}(r_{\Pi}, s_{\Pi})$.

\noindent More precisely, if the Harish-Chandra parameter of $\Pi$ is of the form 
\begin{equation*}
\lambda = \lambda_{a, b} = (\alpha_{1}, \ldots, \alpha_{a}, \beta_{1}, \ldots, \beta_{p-a}, \gamma_{1}, \ldots, \gamma_{b}, \delta_{1}, \ldots, \delta_{q-b}),
\end{equation*}
with $\alpha_{i}, \beta_{j}, \gamma_{k}, \delta_{l} \in \mathbb{Z} + \frac{1}{2}$ such that $\alpha_{1} > \ldots > \alpha_{a} > 0 > \beta_{1} > \ldots > \beta_{p-a}$ and $\gamma_{1} > \ldots > \gamma_{b} > 0 > \delta_{1} > \ldots > \delta_{q-b}$, then $(r_{\Pi}, s_{\Pi}) = (a + q - b, b + p - a)$ and the corresponding Harish-Chandra parameter $\lambda' = \lambda'_{a, b}$ of $\theta_{r_{\Pi}, s_{\Pi}}(\Pi)$ is of the form:
\begin{equation*}
\lambda'_{a, b} = (\alpha_{1}, \ldots, \alpha_{a}, \delta_{1}, \ldots, \delta_{q-b}, \gamma_{1}, \ldots, \gamma_{b}, \beta_{1}, \ldots, \beta_{p-a}).
\end{equation*}

\label{TheoremDSHCP}

\end{theo}

\section{Proof of Conjecture \ref{ConjectureHC} for discrete series representations in the equal rank case}

\label{ProofConjectureDS}

In this section, we are interested in the dual pair $(\G, \G') = (\U(p, q), \U(r, s))$ such that $p+q = r+s$. Without loss of generality, we assume that $p \leq q$. In particular, the number of Cartan subgroups of $\G$, up to conjugation, is $p+1$. We denote by $n = p+q$. Let $(V = \mathbb{C}^{p+q}, \left(\cdot, \cdot\right))$ and $(V' = \mathbb{C}^{r+s}, \left(\cdot, \cdot\right)')$ be the hermitian and skew-hermitian spaces corresponding to $\G$ and $\G'$ respectively. In this case, the space $W = \Hom(V', V) = \M((r+s) \times (p+q), \mathbb{C})$ and for every $w \in W$, there exists a unique element $w^{*} \in \Hom(V, V') = \M((p+q) \times (r+s), \mathbb{C})$ such that:
\begin{equation*}
\left(w(v'), v\right) = \left(v', w^{*}(v)\right)', \qquad (v \in \V, v' \in \V').
\end{equation*}
One can prove that $w^{*} = i\Id_{p, q}\overline{w}^{t}\Id_{r, s}$ and the symplectic form $\langle\cdot, \cdot\rangle$ on $W$ 
\begin{equation*}
\langle w, w' \rangle = \Re(\tr(w'^{*}w)) = -\Im(\tr(\Id_{p, q}\overline{w'}^{t}\Id_{r, s}w)) \qquad (w, w' \in W).
\end{equation*}
Let $V_{i} = V'_{i} = \mathbb{C}e_{i}$. The subspaces $\mathfrak{h}$ and $\mathfrak{h}'$ of $\mathfrak{g}$ and $\mathfrak{g}'$ respectively given by:
\begin{equation*}
\mathfrak{h} = \mathfrak{h}' = \left\{y = (iX_{1}, \ldots, iX_{n}), X_{i} \in \mathbb{R}\right\}
\end{equation*}
are Cartan subalgebras. Moreover, we get:
\begin{equation*}
W^{\mathfrak{h}} = \bigoplus\limits_{i=1}^{n} \Hom(V_{i}, V'_{i}) = \bigoplus\limits_{i=1}^{n} i\mathbb{R}\E_{i, i}.
\end{equation*}

\noindent Let $\Pi$ be a discrete series of $\widetilde{\U}(p, q)$, $\Theta_{\Pi}$ be the corresponding element of $\mathscr{D}'(\widetilde{\G})^{\widetilde{\G}}$, $\Theta_{\Pi}$ the corresponding locally integrable function on $\widetilde{\G}$ such that $\Theta_{\Pi} = T_{\Theta_{\Pi}}$ and $\chi_{\Pi}$ the infinitesimal character of $\Pi$.

\noindent As recalled in Theorem \ref{TheoremTransferOfEigendistributions}, $\Chc^{*}(\Theta_{\Pi})$ is an element of $\Eig(\widetilde{\G'})^{\widetilde{\G'}}$. According to \cite[Theorem~2]{HAR4}, the distribution $\Theta'_{\Pi} = \Chc^{*}(\Theta_{\Pi})$ is given by a locally integrable function $\Theta'_{\Pi}$ on $\widetilde{\G'}$, analytic on $\widetilde{\G'}^{\reg}$.

\begin{nota}

From now on, we fix an element $\widetilde{-1}$ be an element of $\pi^{-1}(\{-1\})$. Let $c: \mathfrak{g}^{c} \to \G^{c}$ the Cayley transform, where $\mathfrak{g}^{c}$ and $\G^{c}$ are defined as in Section \ref{SectionCauchyHarishChandra}. As explained in \cite[Lemma~3.5]{TOM7}, there exists a unique smooth map $\tilde{c}: \mathfrak{g}^{c} \to \widetilde{\G}^{c}$ such that $\pi \circ \tilde{c} = c$ and $\tilde{c}(0) = \widetilde{-1}$.

\label{CayleyCovering}

\end{nota}

\begin{theo}

The value of $\Theta'_{\Pi}$ on the compact Cartan $\widetilde{\H'} = \widetilde{\H'}(\emptyset)$ is given by the following formula:
\begin{equation*}
\Delta_{\Psi'}(\check{h}') \Theta'_{\Pi}(\check{p}(\check{h}')) = \cfrac{\C}{|\mathscr{W}(\H)|} \sum\limits_{\sigma \in \mathscr{S}_{r+s}} \varepsilon(\sigma) \det^{\frac{1}{2}}(\sigma(\check{h}'))_{W^{\mathfrak{h}'}} \lim\limits_{\underset{r \in \E_{\sigma, \emptyset}}{r \to 1}} \displaystyle\int_{\check{\H}} \cfrac{\Theta_{\Pi}(\check{p}(\check{h}))\Delta(\check{h})\det^{\frac{1}{2}}(\check{h})}{\det(1 - p(\check{h})rp(\check{h}'))_{\sigma\W^{\mathfrak{h}}}} d\check{h} \qquad (\check{h}' \in \check{\H}'^{\reg}),
\end{equation*}
where $\H = \H(\emptyset)$ is the compact Cartan of $\G$ and $\C = \overline{\chi_{\Pi}(\widetilde{-1})}\Theta(\widetilde{-1})(-1)^{u}$.
\label{TheoremOmegaCompact}
\end{theo}

\begin{proof}

Let $\varphi$ be a function in $\mathscr{C}^{\infty}_{c}(\widetilde{\G'})$. According to Remark \ref{ValueChcThetaPi}, we get that 
\begin{equation*}
\Theta'_{\Pi}(\varphi) = \sum\limits_{i=0}^{p} \cfrac{1}{|\mathscr{W}(\H(\S_{i}))|} \displaystyle\int_{\widetilde{\H}(\S_{i})^{\reg}} \Theta_{\Pi}(\tilde{h}_{i})|\det(1-\Ad(\tilde{h}^{-1}_{i}))_{\mathfrak{g}/\mathfrak{h}_{i}}| \Chc(\varphi)(\tilde{h}_{i}) d\tilde{h}_{i}.
\end{equation*}
where $\H(\S_{i})$ is a set of Cartan subgroups as in Remark \ref{ExampleUnitary}, and let $\H = \H(\emptyset)$ the compact Cartan of $\G$. Now, if we assume that $\supp(\varphi) \subseteq \widetilde{\G'} \cdot \widetilde{\H'}$, then 
\begin{equation*}
\Theta'_{\Pi}(\varphi) = \cfrac{1}{|\mathscr{W}(\H)|} \displaystyle\int_{\widetilde{\H}^{\reg}} \Theta_{\Pi}(\tilde{h})|\det(1-\Ad(\tilde{h}^{-1}))_{\mathfrak{g}/\mathfrak{h}}| \Chc(\varphi)(\tilde{h}) d\tilde{h}
\end{equation*}
According to \cite[Equation~8]{TOM4} and Theorem \ref{TheoremTheta}, we get:
\begin{eqnarray*}
\Theta'_{\Pi}(\varphi) & = & \cfrac{1}{|\mathscr{W}(\H)|} \displaystyle\int_{\widetilde{\H}^{\reg}} \Theta_{\Pi}(\tilde{h})|\det(1-\Ad(\tilde{h}^{-1}_{i}))_{\mathfrak{g}/\mathfrak{h}}| \Chc(\varphi)(\tilde{h}) d\tilde{h} \\ 
      & = & \cfrac{(-1)^{u}\C}{|\mathscr{W}(\H)|} \displaystyle\int_{\check{\H}^{\reg}} \Theta_{\Pi}(\check{p}(\check{h}))|\Delta_{\G}(\check{h})|^{2} \left(\displaystyle\int_{\widetilde{\G'}} \Theta(\check{p}(\check{h})\tilde{g}')\varphi(\tilde{g}') d\tilde{g}'\right)d\check{h} \\
            & = & \cfrac{(-1)^{u+1}\C}{|\mathscr{W}(\H)|} \displaystyle\int_{\check{\H}^{\reg}} \Theta_{\Pi}(\check{p}(\check{h}))\Delta_{\Psi}(\check{h})\det^{\frac{1}{2}}(\check{h}) \left(\det^{-\frac{1}{2}}(\check{h})\Delta_{\Psi}(\check{h})\displaystyle\int_{\widetilde{\G'}} \Theta(\check{p}(\check{h})\tilde{g}')\varphi(\tilde{g}')d\tilde{g}' \right) d\check{h} \\
               & = & \cfrac{-\C m_{0}}{|\mathscr{W}(\H)|}  \sum\limits_{\sigma \in \mathscr{S}_{r+s}} \varepsilon(\sigma) \lim\limits_{\underset{r \in \E_{\sigma, \emptyset}}{r \to 1}} \displaystyle\int_{\check{\H}^{\reg}} \Theta_{\Pi}(\check{p}(\check{h}))\Delta_{\Psi}(\check{h})\det^{\frac{1}{2}}(\check{h}) \displaystyle\int_{\check{\H}'} \cfrac{\det^{\frac{1}{2}}(\sigma^{-1}(\check{h}'))_{W^{\mathfrak{h}}}}{\det(1-p(\check{h})rp(\check{h}'))_{\sigma\W^{\mathfrak{h}}}} \mathscr{H}_{\emptyset}(\varphi)(\check{h}')d\check{h}' d\check{h} 
\end{eqnarray*}      
With such assumptions on the support of $\varphi$, we get using Equation \eqref{Equation3}
\begin{eqnarray*}
\Theta'_{\Pi}(\varphi) & = & \displaystyle\int_{\widetilde{\G'}} \Theta'_{\Pi}(\tilde{g}') \varphi(\tilde{g}')d\tilde{g}' = m_{0}\displaystyle\int_{\check{\H}'} \overline{\Delta_{\Psi'}(\check{h}')} \mathscr{H}_{\emptyset}(\Theta'_{\Pi}\varphi)(\check{h}') d\check{h}' \\
           & = & -m_{0} \displaystyle\int_{\check{\H}'} \Theta'_{\Pi}(\check{p}(\check{h}')) \Delta_{\Psi'}(\check{h}') \mathscr{H}_{\emptyset}(\varphi)(\check{h}') d\check{h}'
\end{eqnarray*}
By identifications, we get, up to a constant, that:
\begin{equation*}
\Delta_{\Psi'}(\check{h}') \Theta'_{\Pi}(\check{p}(\check{h}')) = \cfrac{\C}{|\mathscr{W}(\H)|} \sum\limits_{\sigma \in \mathscr{S}_{r+s}} \varepsilon(\sigma) \det^{\frac{1}{2}}(\sigma^{-1}(\check{h}'))_{W^{\mathfrak{h}}} \lim\limits_{\underset{r \in \E_{\sigma, \emptyset}}{r \to 1}} \displaystyle\int_{\check{\H}^{\reg}} \cfrac{\Theta_{\Pi}(\check{p}(\check{h}))\Delta_{\Psi}(\check{h})\det^{\frac{1}{2}}(\check{h})}{\det(1-p(\check{h})rp(\sigma(\check{h}')))_{W^{\mathfrak{h}}}} d\check{h} 
\end{equation*}
and the theorem follows.
\end{proof}    

\noindent We know that the set of roots for $(\mathfrak{g}, \mathfrak{h})$ is given by
\begin{equation*}
\left\{\pm(e_{i} - e_{j}), 1 \leq i < j \leq n\right\}.
\end{equation*}
Let $\K = \U(p) \times \U(q)$ be a maximal compact subgroup of $\G$. Let $\Psi(\mathfrak{k}) = \Psi(\mathfrak{k}_{\mathbb{C}}, \mathfrak{h}_{\mathbb{C}})$ be a set of compact positive roots given by:
\begin{equation*}
\Psi(\mathfrak{k}) = \left\{e_{i} - e_{j}, 1 \leq i < j \leq p\right\} \cup \left\{e_{i} - e_{j}, p+1 \leq i < j \leq n\right\}.
\end{equation*}
The compact Weyl group $\mathscr{W}(\mathfrak{k}) = \mathscr{W}(\K, \H)$ is $\mathscr{S}_{p} \times \mathscr{S}_{q}$. Let $\lambda = \sum\limits_{i=1}^{p+q} \lambda_{i}e_{i}$ be the Harish-Chandra parameter of $\Pi$. Using Theorem \ref{TheoremDiscreteSeries}, the value of $\Theta_{\Pi}$ on $\widetilde{\H}^{\reg}$ is given by:
\begin{equation*}
\Theta_{\Pi}(\check{p}(\check{h})) = (-1)^{\alpha_{p, q}}\sum\limits_{\beta \in \mathscr{S}_{p} \times \mathscr{S}_{q}} \varepsilon(\beta) \cfrac{(\beta\check{h})^{\lambda}}{\prod\limits_{\alpha > 0} (\check{h}^{\frac{\alpha}{2}} - \check{h}^{-\frac{\alpha}{2}})}, \qquad (\check{h}' \in \check{\H}^{\reg}),
\end{equation*}
with $\alpha_{p, q} = \frac{\dim(\G) - \dim(\K)}{2} = pq$. Using that $W^{\mathfrak{h}} = \bigoplus\limits_{i=1}^{n} \Hom(V'_{i}, V_{i})$, we get:
\begin{equation*}
\det(1-p(h)rp(h'))_{\sigma\W^{\mathfrak{h}}} = \prod\limits_{i=1}^{n}\left(1 - h_{i}(rh')^{-1}_{\sigma(i)}\right) = (-1)^{n}\prod\limits_{i=1}^{n} (rh')^{-1}_{\sigma(i)} \prod\limits_{i=1}^{n} \left(h_{i} - (rh')_{\sigma(i)}\right),
\end{equation*}
and 
\begin{equation*}
\det^{\frac{1}{2}}(\sigma^{-1}(\check{h}'))_{W^{\mathfrak{h}}} = \prod\limits_{i=1}^{n} h'^{-\frac{1}{2}}_{\sigma(i)}, \qquad \det^{\frac{1}{2}}(\check{h})_{W^{\mathfrak{h}}} = \prod\limits_{i=1}^{n} h^{\frac{1}{2}}_{i}.
\end{equation*}
To simplify the notations, we will denote by $\xi$ the element of $\mathfrak{h}^{*}_{\mathbb{C}}$ given by $\xi = \sum\limits_{i=1}^{n} \frac{1}{2}e_{i}$. 
We recall a basic Cauchy integral formula.

\begin{lemme}

Let $k \in \mathbb{Z}$ and $a \in \mathbb{C}^{*} \setminus \S^{1}$. Then,
\begin{equation*}
\cfrac{1}{2i\pi}\displaystyle\int_{\S^{1}} \cfrac{z^{k}}{z-a}dz = \begin{cases} a^{k} & \text{ if } k \geq 0 \text{ and } |a| < 1 \\ -a^{k} & \text{ if } k < 0 \text{ and } |a| > 1 \\ 0 & \text{ otherwise }
\end{cases}
\end{equation*}

\label{LemmaComplexIntegrals}

\end{lemme}

\noindent For every $\check{h}' \in \check{\H}'^{\reg}$, we get from Theorem \ref{TheoremOmegaCompact}:
\begin{eqnarray*}
\Delta_{\Psi'}(\check{h}') \Theta'_{\Pi}(\check{p}(\check{h}')) & = & \widetilde{\C} \sum\limits_{\sigma \in \mathscr{S}_{r+s}} \sum\limits_{\beta \in \mathscr{S}_{p} \times \mathscr{S}_{q}} \varepsilon(\sigma) \varepsilon(\beta) \prod\limits_{i=1}^{n} h'^{-\frac{1}{2}}_{\sigma(i)} \prod\limits_{i=1}^{n} h'_{\sigma(i)}\lim\limits_{\underset{r \in \E_{\sigma, \emptyset}}{r \to 1}} \displaystyle\int_{\check{\H}} \cfrac{(\beta\check{h})^{\lambda}\prod\limits_{i=1}^{n} h^{\frac{1}{2}}_{i}}{\prod\limits_{i=1}^{n} (h_{i} - (rh')_{\sigma(i)})} d\check{h} \\
    & = & \widetilde{\C} \sum\limits_{\sigma \in \mathscr{S}_{r+s}} \sum\limits_{\beta \in \mathscr{S}_{p} \times \mathscr{S}_{q}} \varepsilon(\sigma) \varepsilon(\beta) \prod\limits_{i=1}^{n} h'^{\frac{1}{2}}_{\sigma(i)}\lim\limits_{\underset{r \in \E_{\sigma, \emptyset}}{r \to 1}} \displaystyle\int_{\check{\H}} \cfrac{(\beta\check{h})^{\lambda+ \xi}}{\prod\limits_{i=1}^{n} (h_{i} - (rh')_{\sigma(i)})} d\check{h} \\
    & = &  2 \widetilde{\C} \sum\limits_{\sigma \in \mathscr{S}_{r+s}} \sum\limits_{\beta \in \mathscr{S}_{p} \times \mathscr{S}_{q}} \varepsilon(\sigma) \varepsilon(\beta) \prod\limits_{i=1}^{n}  h'^{\frac{1}{2}}_{\sigma(i)}\lim\limits_{\underset{r \in \E_{\sigma, \emptyset}}{r \to 1}} \displaystyle\int_{\H} \cfrac{\prod\limits_{i=1}^{n} h^{\lambda_{i} + \frac{1}{2}}_{\beta^{-1}(i)}}{\prod\limits_{i=1}^{n} (h_{i} - (rh')_{\sigma(i)})} dh \\ 
         & = & \cfrac{2 \widetilde{\C} \prod\limits_{i=1}^{n} h'^{\frac{1}{2}}_{i}}{(2i\pi)^{n}} \sum\limits_{\sigma \in \mathscr{S}_{r+s}} \sum\limits_{\beta \in \mathscr{S}_{p} \times \mathscr{S}_{q}} \varepsilon(\sigma) \varepsilon(\beta) \lim\limits_{\underset{r \in \E_{\sigma, \emptyset}}{r \to 1}} \prod\limits_{i=1}^{n} \displaystyle\int_{\S^{1}}\cfrac{z^{\lambda_{i} - \frac{1}{2}}}{z - (rh')_{\sigma(\beta^{-1}(i))}} dz\, ,
\end{eqnarray*}
where $\widetilde{\C} = \cfrac{(-1)^{pq}\C}{\mathscr{W}(\H)}$.

\begin{lemme}

For every $\sigma \in \mathscr{S}_{r+s}$, the space $\E_{\sigma, \emptyset}$ is given by
\begin{equation*}
\E_{\sigma, \emptyset} = \left\{h' = (e^{-X_{1}}, \ldots, e^{-X_{n}}) \in \H'_{\mathbb{C}}, \begin{cases} X_{\sigma(i)} > 0 & \text{ if } i \in \{1, \ldots, p\} \text{ and } \sigma(i) \in \{1, \ldots, r\} \\ X_{\sigma(i)} < 0 & \text{ if } i \in \{1, \ldots, p\} \text{ and } \sigma(i) \in \{r+1, \ldots, r+s\} \\ X_{\sigma(i)} < 0 & \text{ if } i \in \{p+1, \ldots, p+q\} \text{ and } \sigma(i) \in \{1, \ldots, r\} \\ X_{\sigma(i)} > 0 & \text{ if } i \in \{p+1, \ldots, p+q\} \text{ and } \sigma(i) \in \{r+1, \ldots, r+s\} \end{cases} \right\}
\end{equation*}
\label{LemmaESigma}

\end{lemme}

\begin{proof}

Let $w = \sum\limits_{i=1}^{n} w_{i} \E_{i, i} \in W^{\mathfrak{h}}$, $\sigma \in \mathscr{S}_{r+s}$ and $y = (iX_{1}, \ldots, iX_{n}) \in \mathfrak{h}'$, with $X_{j} \in \mathbb{R}$.
Then,
\begin{eqnarray*}
\langle y(\sigma(w)), \sigma(w)\rangle & = & \langle y\left(\sum\limits_{i=1}^{n} w_{i} \E_{i, \sigma(i)}\right), \sum\limits_{j=1}^{n} w_{j} \E_{j, \sigma(j)}\rangle = - \langle \sum\limits_{i=1}^{n} w_{i}y_{\sigma(i)}\E_{i, \sigma(i)}, \sum\limits_{j=1}^{n} w_{j} \E_{j, \sigma(j)j}\rangle \\ 
       & = & \sum\limits_{i=1}^{n}\sum\limits_{j=1}^{n} \Im(\tr(\overline{w_{j}}\Id_{p, q}\E_{\sigma(j), j}\Id_{r, s}w_{i}y_{\sigma(i)}\E_{i, \sigma(i)})) \\
       & = & \sum\limits_{i=1}^{n} \Im(\tr(\Id_{p, q}\overline{w_{i}}\E_{i, \sigma(i)}\Id_{r, s}w_{i}y_{\sigma(i)}\E_{\sigma(i), i})) \\
       & = & \sum\limits_{i=1}^{p} \Im(\tr(|w_{i}|^{2}y_{\sigma(i)} \E_{i, \sigma(i)}\Id_{r, s}\E_{\sigma(i), i})) - \sum\limits_{i=p+1}^{n} \Im(\tr(|w_{i}|^{2}y_{\sigma(i)}\E_{i, \sigma(i)}\Id_{r, s}\E_{\sigma(i), i})) \\
       & = & \sum\limits_{\underset{\sigma(i) \in \{1, \ldots, r\}}{i=1}}^{p} |w_{i}|^{2}X_{\sigma(i)} - \sum\limits_{\underset{\sigma(i) \in \{r+1, \ldots, n\}}{i=1}}^{p} |w_{i}|^{2}X_{\sigma(i)} - \sum\limits_{\underset{\sigma(i) \in \{1, \ldots, r\}}{i=p+1}}^{n} |w_{i}|^{2}X_{\sigma(i)} + \sum\limits_{\underset{\sigma(i) \in \{r+1, \ldots, n\}}{i=p+1}}^{n} |w_{i}|^{2}X_{\sigma(i)}
\end{eqnarray*}
In particular, using Equation \eqref{SpaceEGamma}, we get:
\begin{equation*}
\Gamma_{\sigma, \emptyset} = \left\{y = (iX_{1}, \ldots, iX_{n}) \in \mathfrak{h}', \begin{cases} X_{\sigma(i)} > 0 & \text{ if } i \in \{1, \ldots, p\} \text{ and } \sigma(i) \in \{1, \ldots, r\} \\ X_{\sigma(i)} < 0 & \text{ if } i \in \{1, \ldots, p\} \text{ and } \sigma(i) \in \{r+1, \ldots, n\} \\ X_{\sigma(i)} < 0 & \text{ if } i \in \{p+1, \ldots, n\} \text{ and } \sigma(i) \in \{1, \ldots, r\} \\ X_{\sigma(i)} > 0 & \text{ if } i \in \{p+1, \ldots, n\} \text{ and } \sigma(i) \in \{r+1, \ldots, n\} \end{cases}\right\}
\end{equation*}
The result follows using that $\E_{\sigma, \emptyset} = \exp(i\Gamma_{\sigma, \emptyset})$.

\end{proof}

\begin{prop}

Let $\Pi \in \mathscr{R}(\widetilde{\U}(p, q), \omega)$ be a discrete series representation of Harish-Chandra parameter $\lambda_{a, b}$ as in Theorem \ref{TheoremDSHCP} and let $(r, s) = (r_{\Pi}, s_{\Pi})$ the unique integers such that $\theta_{r, s}(\Pi) \neq 0$. The value of $\Theta'_{\Pi}$ on $\widetilde{\H'}^{\reg} = \widetilde{\H'(\emptyset)}^{\reg}$ is given by
\begin{equation*}
\Delta_{\Psi'}(\check{h}') \Theta'_{\Pi}(\check{p}(\check{h}')) = 2(-1)^{n-a-b}\varepsilon(\tau_{a, b})\widetilde{\C}\sum\limits_{\sigma \in \mathscr{S}_{r} \times \mathscr{S}_{s}} \varepsilon(\sigma) (\sigma \check{h}')^{\tau_{a, b} \lambda_{a, b}},
\end{equation*}
where $\tau_{a, b} \in \mathscr{S}_{r+s}$ is defined by:
\begin{itemize}
\item If $r \leq p$, $\tau_{a, b} = (a+1, p+b+1)(a+2, p+b+2) \ldots (r, p+q)$,
\item If $p+1 \leq r \leq p+b$, $\tau_{a, b} \in \Stab_{\mathscr{S}_{r+s}}\left(\{1, \ldots, a\} \cup \{r+1, \ldots, p+b\}\right)$ and satisfies:
\begin{equation*}
\tau_{a, b}(a+1) = p+b+1, \ldots, \tau_{a, b}(r) = r+s, \tau_{a, b}(p+b+1) = a+1, \ldots, \tau_{a, b}(p+q) = r.
\end{equation*}
\item If $r \geq p + b + 1$, $\tau_{a, b} \in \Stab_{\mathscr{S}_{r+s}}\left(\{1, \ldots, a\} \cup \{p+b+1, \ldots, r\}\right)$ and satisfies
\begin{equation*}
\tau_{a, b}(a+1) = r+1, \ldots, \tau_{a, b}(p+b) = r+s, \tau_{a, b}(r+1) = a+1, \ldots, \tau_{a, b}(r+s) = p+b+1.
\end{equation*}
\end{itemize}

\label{ImportantProposition1}

\end{prop}

\begin{nota}

\noindent For every subset $\{i_{1}, \ldots, i_{k}\}$ of $\{1, \ldots, p\}$ (resp. $\{p+1, \ldots, p+q\}$, $\{1, \ldots, r\}$ or $\{r+1, \ldots, r+s\}$), we denote by $\{i_{1}, \ldots, i_{k}\}^{c}$ the set $\{1, \ldots, p\} \setminus \{i_{1}, \ldots, i_{k}\}$ (resp. $\{p+1, \ldots, p+q\} \setminus \{i_{1}, \ldots, i_{k}\}$, $\{1, \ldots, r\} \setminus \{i_{1}, \ldots, i_{k}\}$ or $\{r+1, \ldots, r+s\} \setminus \{i_{1}, \ldots, i_{k}\}$).

\noindent For two subsets $\{a_{1}, \ldots, a_{w}\}$ and $\{b_{1}, \ldots, b_{w}\}$ of $\{1, \ldots, p+q\}$, we denote by $\mathscr{S}^{\{b_{1}, \ldots, b_{w}\}}_{\{a_{1}, \ldots, a_{w}\}}$ the groups of bijections between $\{a_{1}, \ldots, a_{w}\}$ and $\{b_{1}, \ldots, b_{w}\}$.

\noindent Similarly, for every $\beta \in \mathscr{S}_{p} \times \mathscr{S}_{q}$, we denote by $\mathscr{S}(\beta)^{\{b_{1}, \ldots, b_{w}\}}_{\{a_{1}, \ldots, a_{w}\}}$ the groups of bijections between $\{\beta(a_{1}), \ldots, \beta(a_{w})\}$ and $\{b_{1}, \ldots, b_{w}\}$. Obviously,
\begin{equation*}
\mathscr{S}_{p} \times \mathscr{S}_{q} = \bigcup\limits_{\{i_{1}, \ldots, i_{a}\} \subseteq \{1, \ldots, p\}} \bigcup\limits_{\{j_{1}, \ldots, j_{b}\} \subseteq \{p+1, \ldots, p+q\}} \mathscr{S}^{\{i_{1}, \ldots, i_{a}\}^{c}}_{\{1, \ldots, p-a\}} \times \mathscr{S}^{\{i_{1}, \ldots, i_{a}\}}_{\{p-a+1, \ldots, p\}} \times \mathscr{S}^{\{j_{1}, \ldots, j_{b}\}^{c}}_{\{p+1, \ldots, p+q-b\}} \times \mathscr{S}^{\{j_{1}, \ldots, j_{b}\}}_{\{p+q-b+1, \ldots, p+q\}}
\end{equation*}
for every $1 \leq t \leq p$.

\end{nota}

\begin{proof}

To simplify the notations, we will denote by $\R(\sigma, \lambda_{a, b}, \beta), \sigma \in \mathscr{S}_{p+q}, \beta \in \mathscr{S}_{p} \times \mathscr{S}_{q}$, the following term:
\begin{equation*}
\R(\sigma, \lambda_{a, b}, \beta) = \lim\limits_{\underset{r \in \E_{\sigma, \emptyset}}{r \to 1}} \prod\limits_{i=1}^{n} \displaystyle\int_{\S^{1}}\cfrac{z^{\lambda_{i} - \frac{1}{2}}}{z - (rh')_{\sigma(\beta^{-1}(i))}}dz
\end{equation*}

\noindent According to Lemmas \ref{LemmaComplexIntegrals} and \ref{LemmaESigma}, we get that $\R(\sigma, \lambda_{a, b}, \beta) \neq 0$ if and only if
\begin{equation*}
\sigma \circ \beta^{-1} \in \bigcup\limits_{\{i_{1}, \ldots, i_{q-b}\} \subseteq \{1, \ldots, r\}} \bigcup\limits_{\{j_{1}, \ldots, j_{p-a}\} \subseteq \{r+1, \ldots, r+s\}} \mathscr{S}^{\{i_{1}, \ldots, i_{q-b}\}^{c}}_{\{1, \ldots, a\}} \times \mathscr{S}^{\{j_{1}, \ldots, j_{p-a}\}}_{\{a+1, \ldots, p\}} \times \mathscr{S}^{\{j_{1}, \ldots, j_{p-a}\}^{c}}_{\{p+1, \ldots, p+b\}} \times \mathscr{S}^{\{i_{1}, \ldots, i_{q-b}\}}_{\{p+b+1, \ldots, p+q\}}.
\end{equation*}
We first assume that $r \leq p$. In this case, using that $\left\{a+1, \ldots, p\right\} = \left\{a+1, \ldots, r\right\} \cup \left\{r+1, \ldots, p\right\}$, we get
\begin{equation*}
\mathscr{S}_{r} \times \mathscr{S}_{s} = \left(\bigcup\limits_{\{i_{1}, \ldots, i_{q-b}\} \subseteq \{1, \ldots, r\}} \bigcup\limits_{\{j_{1}, \ldots, j_{p-a}\} \subseteq \{r+1, \ldots, r+s\}} \mathscr{S}^{\{i_{1}, \ldots, i_{q-b}\}^{c}}_{\{1, \ldots, a\}} \times \mathscr{S}^{\{j_{1}, \ldots, j_{p-a}\}}_{\{a+1, \ldots, p\}} \times \mathscr{S}^{\{j_{1}, \ldots, j_{p-a}\}^{c}}_{\{p+1, \ldots, p+b\}} \times \mathscr{S}^{\{i_{1}, \ldots, i_{q-b}\}}_{\{p+b+1, \ldots, p+q\}}\right) \circ \sigma_{1},
\end{equation*}
where $\sigma_{1} = (a+1, p+b+1)(a+2, p+b+2) \ldots (r, p+q)$. For every $\beta \in \mathscr{S}_{p} \times \mathscr{S}_{q}$, there exists exactly $r!s!$ elements in $\sigma \in \mathscr{S}_{r+s}$ such that $\sigma \circ \beta^{-1} \in \mathscr{S}_{r} \times \mathscr{S}_{s} \circ \sigma_{1}$.Then, 
\begin{eqnarray*}
& & \sum\limits_{\sigma \in \mathscr{S}_{r+s}} \sum\limits_{\beta \in \mathscr{S}_{p} \times \mathscr{S}_{q}}  \varepsilon(\sigma\beta) \lim\limits_{\underset{r \in \E_{\sigma, \emptyset}}{r \to 1}} \prod\limits_{i=1}^{p+q} \displaystyle\int_{\S^{1}}\cfrac{z^{\lambda_{i} - \frac{1}{2}}}{z - (rh')_{\sigma(\beta^{-1}(i))}} dz = (-1)^{p+q-a-b}(2i\pi)^{p+q}p!q! \sum\limits_{\tau \in \mathscr{S}_{r} \times \mathscr{S}_{s} \circ \sigma_{1}} \varepsilon(\tau) \prod\limits_{i=1}^{p+q} h'^{\lambda_{i} - \frac{1}{2}}_{\tau(i)} \\
& = & (-1)^{p+q-a-b}(2i\pi)^{p+q}p!q! \sum\limits_{\tau \in \mathscr{S}_{r} \times \mathscr{S}_{s}} \varepsilon(\tau \sigma^{-1}_{1}) \prod\limits_{i=1}^{p+q} h'^{\lambda_{i} - \frac{1}{2}}_{\tau(\sigma^{-1}_{1}(i))} = (-1)^{p+q-a-b}(2i\pi)^{p+q}p!q! \sum\limits_{\tau \in \mathscr{S}_{r} \times \mathscr{S}_{s}} \varepsilon(\tau \sigma^{-1}_{1})(\sigma_{1}\tau^{-1}\check{h}')^{\lambda_{a, b} - \xi} \\ 
& = & (-1)^{p+q-a-b}(2i\pi)^{p+q}p!q! \varepsilon(\sigma_{1})\sum\limits_{\tau \in \mathscr{S}_{r} \times \mathscr{S}_{s}} \varepsilon(\tau)(\tau\check{h}')^{\sigma_{1}(\lambda_{a, b}-\xi)},
\end{eqnarray*}
where $\xi = \sum\limits_{i=1}^{n} \frac{1}{2} e_{i}$, i.e.
\begin{equation*}
\Delta_{\Psi'}(\check{h}')\Theta'_{\Pi}(\check{p}(\check{h}')) = 2(-1)^{p+q-a-b}\varepsilon(\sigma_{1})\widetilde{\C} \sum\limits_{\tau \in \mathscr{S}_{r} \times \mathscr{S}_{s}} \varepsilon(\tau)(\tau\check{h}')^{\sigma_{1}(\lambda_{a, b})}.
\end{equation*}
Now assume that $r > p$. We distinguish two cases. If $p+1 \leq r \leq p+b$, then,
\begin{equation*}
\mathscr{S}_{r} \times \mathscr{S}_{s} = \left(\bigcup\limits_{\{i_{1}, \ldots, i_{q-b}\} \subseteq \{1, \ldots, r\}} \bigcup\limits_{\{j_{1}, \ldots, j_{p-a}\} \subseteq \{r+1, \ldots, r+s\}} \mathscr{S}^{\{i_{1}, \ldots, i_{q-b}\}^{c}}_{\{1, \ldots, a\}} \times \mathscr{S}^{\{j_{1}, \ldots, j_{p-a}\}}_{\{a+1, \ldots, p\}} \times \mathscr{S}^{\{j_{1}, \ldots, j_{p-a}\}^{c}}_{\{p+1, \ldots, p+b\}} \times \mathscr{S}^{\{i_{1}, \ldots, i_{q-b}\}}_{\{p+b+1, \ldots, p+q\}} \right) \circ \eta,
\end{equation*}
for every $\eta \in \mathscr{S}_{r+s}$ satisfying $\eta\{1, \ldots, a\} = \{1, \ldots, a\}, \thinspace \eta\{r+1, \ldots, p+b\} = \{r+1, \ldots, p+b\}, \thinspace \eta\{a+1, \ldots, r\} \subseteq \{p+b+1, \ldots, r+s\}$ and $\eta\{p+b+1, \ldots, p+q\} \subseteq \{a+1, \ldots, r\}$. Let $\sigma_{2}$ be the element of $\Stab_{\mathscr{S}_{r+s}}\left(\{1, \ldots, a\} \cup \{r+1, \ldots, p+b\}\right)$ given by
\begin{equation*}
\sigma_{2}(a+1) = p+b+1, \ldots, \sigma_{2}(r) = r+s, \sigma_{2}(p+b+1) = a+1, \ldots, \sigma_{2}(p+q) = r.
\end{equation*}
This element satisfy the previous conditions and we get:
\begin{equation*}
\Delta_{\Psi'}(\check{h}')\Theta'_{\Pi}(\check{p}(\check{h}')) = 2(-1)^{p+q-a-b}\varepsilon(\sigma_{2}) \widetilde{\C} \sum\limits_{\tau \in \mathscr{S}_{r} \times \mathscr{S}_{s}} \varepsilon(\tau)(\tau\check{h}')^{\sigma_{2}(\lambda_{a, b})}.
\end{equation*}
Similarly, $r \geq p+b+1$,
\begin{equation*}
\mathscr{S}_{r} \times \mathscr{S}_{s} = \left(\bigcup\limits_{\{i_{1}, \ldots, i_{q-b}\} \subseteq \{1, \ldots, r\}} \bigcup\limits_{\{j_{1}, \ldots, j_{p-a}\} \subseteq \{r+1, \ldots, r+s\}} \mathscr{S}^{\{i_{1}, \ldots, i_{q-b}\}^{c}}_{\{1, \ldots, a\}} \times \mathscr{S}^{\{j_{1}, \ldots, j_{p-a}\}}_{\{a+1, \ldots, p\}} \times \mathscr{S}^{\{j_{1}, \ldots, j_{p-a}\}^{c}}_{\{p+1, \ldots, p+b\}} \times \mathscr{S}^{\{i_{1}, \ldots, i_{q-b}\}}_{\{p+b+1, \ldots, p+q\}}\right) \circ \eta,
\end{equation*}
for every $\eta \in \mathscr{S}_{r+s}$ satisfying $\eta\{1, \ldots, a\} = \{1, \ldots, a\}, \thinspace \eta\{p+b+1, \ldots, r\} = \{p+b+1, \ldots, r\}, \thinspace \eta\{a+1, \ldots, p+b\} \subseteq \{r+1, \ldots, r+s\}$ and $\eta\{r+1, \ldots, r+s\} \subseteq \{a+1, \ldots, p+b+1\}$. Let $\sigma_{3}$ be the element of $\Stab_{\mathscr{S}_{r+s}}\left(\{1, \ldots, a\} \cup \{p+b+1, \ldots, r\}\right)$ given by
\begin{equation*}
\sigma_{3}(a+1) = r+1, \ldots, \sigma_{3}(p+b) = r+s, \sigma_{3}(r+1) = a+1, \ldots, \sigma_{3}(r+s) = p+b+1.
\end{equation*}
This element satisfies the previous conditions and we get:
\begin{equation*}
\Delta_{\Psi'}(\check{h}')\Theta'_{\Pi}(\check{p}(\check{h}')) = 2(-1)^{p+q-a-b}\varepsilon(\sigma_{3}) \widetilde{\C} \sum\limits_{\tau \in \mathscr{S}_{r} \times \mathscr{S}_{s}} \varepsilon(\tau)(\tau\check{h}')^{\sigma_{3}(\lambda_{a, b})}.
\end{equation*}

\end{proof}

\begin{prop}

For every $\Pi \in \mathscr{R}(\widetilde{\U}(p, q), \omega)$, we get
\begin{equation*}
\sup\limits_{\tilde{g}' \in \widetilde{\G'}^{\reg}} |\D(\tilde{g}')|^{\frac{1}{2}} |\Theta'_{\Pi}(\tilde{g}')| < \infty.
\end{equation*}

\label{ImportantProposition2}

\end{prop}

\noindent We first need to introduce some notations.

\begin{nota}

Let $k \in [|1, \min(r, s)|]$, we denote by $\eta'_{\S_{k}} = \Ad(c(\S_{k})^{-1})(\eta'(\S_{k})) \subseteq \eta'^{+}_{\mathbb{C}}$, where $\eta'^{+}_{\mathbb{C}} = \bigoplus\limits_{\alpha \in \Psi'^{+}} {\mathfrak{g}'_{\mathbb{C}}}_{\alpha}$, where ${\mathfrak{g}'_{\mathbb{C}}}_{\alpha}$ is the eigenspace corresponding to $\alpha \in \Psi'$.

\noindent By keeping the notations of Section \ref{SectionCauchyExplicit}, we get that $\Psi'$ can be decomposed as follow:
\begin{equation}
\Psi' = \Psi'(\mathfrak{gl}(\X_{k})) \cup \Psi'(\mathfrak{g}(\U_{k})) \cup \Psi'(\eta'(\S_{k})), \qquad (k \in [|1, \min(p, \min(r, s))|]).
\label{DecompositionPsik}
\end{equation}

\noindent Finally, we denote by $\mathscr{W}(\mathfrak{g}(\U_{k}))$ the Weyl group corresponding to $(\mathfrak{g}(\U_{k})_{\mathbb{C}}, \mathfrak{t}_{2}(\S_{k})_{\mathbb{C}})$.

\end{nota}

\begin{lemme}

For every $\tilde{h}' \in \widetilde{\H'}(\S_{k})^{\reg}$, $\det(\Id - \Ad(\tilde{h}'))_{\eta'(\S_{k})} \in \mathbb{R}^{*}_{+}$.

\end{lemme}

\begin{proof}

To make things easier, we will consider $\S_{k} = \{e_{1} - e_{r+s-k+1}, \ldots, e_{k} - e_{r+s}\}$. As in Equation \eqref{DiagonalHSi}, we have:
\begin{equation*}
\H'_{\S_{k}} = c(\S_{k})^{-1}\H'(\S_{k})c(\S_{k}) = \left\{h' = (e^{i\theta_{1}-X_{1}}, \ldots, e^{i\theta_{k}-X_{k}}, t_{1}, \ldots, t_{r+s-2k}, e^{i\theta_{1}+X_{1}}, \ldots, e^{i\theta_{k}+X_{k}}), t_{i} \in \U(1), \theta_{i}, X_{i} \in \mathbb{R}\right\}.
\end{equation*}
We denote by $h'_{1} = c(\S_{k})^{-1}h'c(\S_{k})$. Obviously, we get:
\begin{equation*}
\det(\Id - \Ad(\tilde{h}'))_{\eta'(\S_{k})} = \det(\Id - \Ad(h'))_{\eta'(\S_{k})} = \det(\Id - \Ad(h'_{1}))_{\eta'_{\S_{k}}} = \det(\Id - \Ad(\check{h}'_{1}))_{\eta'_{\S_{k}}}.
\end{equation*}
We know that
\begin{equation*}
\det(\Id - \Ad(h_{1}))_{\eta'_{\S_{k}}} = \prod\limits_{\alpha \in \Psi'(\eta'(\S_{k}))}(1 - h^{-\alpha}_{1}),
\end{equation*}
and that 
\begin{equation*}
\Psi'(\eta'(\S_{k})) = \{e_{i} - e_{j}, 1 \leq i \leq k, k+1 \leq j \leq r+s-k\} \cup \{e_{i} - e_{j}, k+1 \leq i \leq r+s-k, r+s-k+1 \leq j \leq r+s\}.
\end{equation*}

\noindent Let $\alpha_{1} = e_{i} - e_{j}$, with $1 \leq i \leq k, k+1 \leq j \leq r+s-k$ and let $\alpha_{2} = e_{j} - e_{r+s-i+1}$. Then,
\begin{equation*}
(1 - h^{\alpha_{1}}_{1})(1 - h^{\alpha_{2}}_{1}) = (1 - e^{i\theta_{i} -X_{i}}t^{-1}_{j-k})(1 - t_{j-k}e^{-i\theta_{i} -X_{i}}) = |1 - e^{i\theta_{i} -X_{i}}t^{-1}_{j-k}|^{2},
\end{equation*}
and the result follows.

\end{proof}

\begin{proof}[Proof of Proposition \ref{ImportantProposition2}]

Without loss of generality, we can assume that $r \leq s$. We distinguish two cases. We first start with $p \leq r$. Note that in this case, $\H_{\S_{i}} = \H'_{\S_{i}}$ (resp. $\H(\S_{i}) = \H'(\S_{i})$) for every $0 \leq i \leq p$, with $\S_{i} = \{e_{1} - e_{r+1}, \ldots, e_{i} - e_{r+i}\}$ as in Notation \ref{NotationsCartanSI}.

\noindent In this case, we get, using \cite[Corollary~A.4]{TOM2}, that for every $\varphi \in \mathscr{C}^{\infty}_{c}(\widetilde{\G'})$: {\small
\begin{eqnarray}
\Theta'_{\Pi}(\varphi) & = & \displaystyle\int_{\widetilde{\G'}} \Theta'_{\Pi}(\tilde{g}') \varphi(\tilde{g}') d\tilde{g}' = \sum\limits_{i=0}^{r} m_{i} \displaystyle\int_{\check{\H}'_{\S_{i}}} \varepsilon_{\S_{i}, \mathbb{R}}(\check{h}') \Delta_{\Phi'}(\check{h}') \mathscr{H}_{\S_{i}}(\Theta'_{\Pi}\varphi)(\check{h}') d\check{h}' \nonumber \\
                              & = & \sum\limits_{i=0}^{r} m_{i} \displaystyle\int_{\check{\H}'_{\S_{i}}} \varepsilon_{\S_{i}, \mathbb{R}}(\check{h}') \Delta_{\Phi'}(\check{h}') \left( \varepsilon_{\S_{i}, \mathbb{R}}(\check{h}') \Delta_{\Psi'}(\check{h}') \displaystyle\int_{\G' / \H'(\S_{i})} \Theta'_{\Pi}(c(\S_{i})\check{p}(\check{h}')c(\S_{i})^{-1})\varphi(g'c(\S_{i})\check{p}(\check{h}')c(\S_{i})^{-1}g'^{-1}) \overline{dg'}\right)d\check{h}' \nonumber \\
                              & = & \sum\limits_{i=0}^{r} m_{i} \displaystyle\int_{\check{\H}_{\S_{i}}} \Theta'_{\Pi}(c(\S_{i})\check{p}(\check{h}')c(\S_{i})^{-1}) |\Delta_{\G'}(\check{h}')|^{2} \displaystyle\int_{\G' / \H'(\S_{i})} \varphi(g'c(\S_{i})\check{p}(\check{h}')c(\S_{i})^{-1}g'^{-1}) \overline{dg'}d\check{h}' \nonumber \\
                              & = & m_{0} \displaystyle\int_{\check{\H}'_{\S_{0}}} \Theta'_{\Pi}(\check{p}(\check{h}')) |\Delta_{\G'}(\check{h}')|^{2} \displaystyle\int_{\G' / \H'(\S_{0})} \varphi(g'\check{p}(\check{h}')g'^{-1}) \overline{dg'}d\check{h}' \nonumber \\
                              & + & \sum\limits_{i=1}^{p} m_{i} \displaystyle\int_{\check{\H}'_{\S_{i}}} \Theta'_{\Pi}(c(\S_{i})\check{p}(\check{h}')c(\S_{i})^{-1}) |\Delta_{\G'}(\check{h}')|^{2} \Lambda(c(\S_{i})\check{p}(\check{h}')c(\S_{i})^{-1})\displaystyle\int_{\L'(\S_{i}) / \H'(\S_{i})} \varphi^{\widetilde{\K'}}_{\widetilde{\N'}(\S_{i})}(l'c(\S_{i})\check{p}(\check{h}')c(\S_{i})^{-1}l'^{-1}) \overline{dl'}d\check{h}' \label{EquationH(Si)1} \\ 
                              & + & \sum\limits_{i=p+1}^{r} m_{i} \displaystyle\int_{\check{\H}'_{\S_{i}}} \Theta'_{\Pi}(c(\S_{i})\check{p}(\check{h}')c(\S_{i})^{-1}) |\Delta_{\G'}(\check{h}')|^{2} \displaystyle\int_{\L'(\S_{p}) / \H'(\S_{i})} \varphi^{\widetilde{\K'}}_{\widetilde{\N'}(\S_{p})}(g'c(\S_{i})\check{p}(\check{h}')c(\S_{i})^{-1}g'^{-1}) \overline{dg'}d\check{h}' \nonumber
\end{eqnarray}}
where $|\Delta_{\G'}(\check{h}')|^{2} = \Delta_{\Phi'}(\check{h}')\Delta_{\Psi'}(\check{h}')$ as in Remark \ref{DeltaSquare}, $\L'(\S_{i})$ is defined in Section \ref{SectionCauchyExplicit} and $\Lambda(c(\S_{i})\check{p}(\check{h}')c(\S_{i})^{-1})$ is given by:
\begin{equation*}
\Lambda(c(\S_{i})\check{p}(\check{h}')c(\S_{i})^{-1}) = \cfrac{\D_{\L'(\S_{i})}(c(\S_{i})\check{p}(\check{h}')c(\S_{i})^{-1})}{\D_{\L'(\S_{0})}(c(\S_{i})\check{p}(\check{h}')c(\S_{i})^{-1})} = \cfrac{|\det(\Id - \Ad(c(\S_{i})\check{p}(\check{h}')c(\S_{i})^{-1})^{-1})_{\mathfrak{l}'(\S_{i})/\mathfrak{h}'(\S_{i})}|^{\frac{1}{2}}}{|\det(\Id - \Ad(c(\S_{i})\check{p}(\check{h}')c(\S_{i})^{-1})^{-1})_{\mathfrak{g}'/\mathfrak{h}'(\S_{i})}|^{\frac{1}{2}}}.
\end{equation*}
Using that $\Theta'_{\Pi} = \Chc^{*}(\Theta_{\Pi})$, we get:
\begin{equation*}
\Theta'_{\Pi}(\varphi) = \sum\limits_{j=0}^{p} \displaystyle\int_{\widetilde{\H}(\S_{j})} \Theta_{\Pi}(\tilde{h}) |\det(\Id - \Ad(\tilde{h})^{-1})_{\mathfrak{g}/\mathfrak{h}(\S_{j})}|^{2} \Chc_{\tilde{h}}(\varphi)d\tilde{h}.
\end{equation*}
Using Equation \eqref{ChcNonCompact}, we get that:
\begin{eqnarray}
\Theta'_{\Pi}(\varphi) & = & \sum\limits_{j=0}^{p} \displaystyle\int_{\widetilde{\T_{1}}(\S_{j})} \displaystyle\int_{\widetilde{\A}(\S_{j})} \displaystyle\int_{\widetilde{\T_{2}}(\S_{j})} \Theta_{\Pi}(\tilde{t}_{1}\tilde{a}\tilde{t}_{2}) |\det(\Id - \Ad(\tilde{t}_{1}\tilde{a}\tilde{t}_{2})^{-1})_{\mathfrak{g}/\mathfrak{h}(\S_{j})}|^{2} \Chc_{\tilde{t}_{1}\tilde{a}\tilde{t}_{2}}(\varphi)d\tilde{t}_{2}d\tilde{a}d\tilde{t}_{1} \nonumber \\
     & = & \sum\limits_{j=0}^{p} \C_{j} \displaystyle\int_{\widetilde{\T_{1}}(\S_{j})} \displaystyle\int_{\widetilde{\A}(\S_{j})} \displaystyle\int_{\widetilde{\T_{2}}(\S_{j})} \Theta_{\Pi}(\tilde{t}_{1}\tilde{a}\tilde{t}_{2}) |\det(\Id - \Ad(\tilde{t}_{1}\tilde{a}\tilde{t}_{2})^{-1})_{\mathfrak{g}/\mathfrak{h}(\S_{j})}|^{2} \cfrac{\varepsilon(\tilde{t}_{1}\tilde{a}\tilde{y})\d_{\S_{j}}(\tilde{t}_{1}\tilde{a}\tilde{t}_{2})}{|\det(\Id - \Ad(\tilde{t}_{1}\tilde{a}\tilde{t}_{2})^{-1})_{\eta(\S_{j})}|} \label{EquationH(Si)2} \\
     & & \displaystyle\int_{\GL(\X_{j})/\T'_{1}(\S_{j}) \times \A'(\S_{j})} \displaystyle\int_{\widetilde{\G}(\U_{j})} \Chc_{\W_{0, j}}(\tilde{t}_{2}\tilde{y}) \varepsilon(\tilde{t}_{1}\tilde{a}\tilde{y}) \d'_{\S_{j}}(g\tilde{t}_{1}\tilde{a}g^{-1}\tilde{y}) \varphi^{\widetilde{\K'}}_{\N'(\S_{j})}(g\tilde{t}_{1}\tilde{a}g^{-1}\tilde{y}) d\tilde{y}\overline{dg} d\tilde{t}_{2}d\tilde{a}d\tilde{t}_{1} \nonumber 
\end{eqnarray}
Let $i \in [|1, r|]$ and $\varphi \in \mathscr{C}^{\infty}_{c}(\widetilde{\G'})$ such that $\supp(\varphi) \subseteq \widetilde{\G'}\cdot\widetilde{\H'}(\S_{i})$. On one hand, using Equation \eqref{EquationH(Si)1}, we get:
\begin{eqnarray*}
\Theta'_{\Pi}(\varphi) & = & m_{i} \displaystyle\int_{\check{\H}'_{\S_{i}}} \Theta'_{\Pi}(c(\S_{i})\check{p}(\check{h}')c(\S_{i})^{-1}) |\Delta_{\G'}(\check{h}')|^{2} \Lambda(c(\S_{i})\check{p}(\check{h}')c(\S_{i})^{-1})\displaystyle\int_{\L'(\S_{i}) / \H'(\S_{i})} \varphi^{\widetilde{\K'}}_{\widetilde{\N'}(\S_{i})}(l'c(\S_{i})\check{p}(\check{h}')c(\S_{i})^{-1}l'^{-1}) \overline{dl'}d\check{h}' \\
                                 & = & m_{i} \displaystyle\int_{\check{\T}'_{1, \S_{i}}} \displaystyle\int_{\check{\A}'_{\S_{i}}} \displaystyle\int_{\check{\T}'_{2, \S_{i}}} \Theta'_{\Pi}(c(\S_{i})\check{p}(\check{t}'_{1}\check{a}'\check{t}'_{2})c(\S_{i})^{-1}) |\Delta_{\G'}(\check{t}'_{1}\check{a}'\check{t}'_{2})|^{2} \Lambda(c(\S_{i})\check{p}(\check{t}'_{1}\check{a}'\check{t}'_{2})c(\S_{i})^{-1})\\ 
                                 & & \qquad \displaystyle\int_{\GL(\X_{i}) / \T'_{1}(\S_{i}) \times \A'(\S_{i})} \displaystyle\int_{\G(\U_{i}) / \T'_{2}(\S_{i})}\varphi^{\widetilde{\K'}}_{\widetilde{\N'}(\S_{i})}(g_{1}g_{2}c(\S_{i})\check{p}(\check{t}'_{1}\check{a}'\check{t}'_{2})c(\S_{i})^{-1}g^{-1}_{2}g^{-1}_{1}) \overline{dg_{2}}\overline{dg_{1}}d\check{t}'_{2}d\check{a}'d\check{t}'_{1}.
\end{eqnarray*} 
In particular, for every $j < i$, we get from Equation \eqref{DecompositionHiandHj}:
\begin{eqnarray*}
\Theta'_{\Pi}(\varphi) & = & m_{i} \displaystyle\int_{\check{\T}'_{1, \S_{j}}} \displaystyle\int_{\check{\A}'_{\S_{j}}} \displaystyle\int_{\check{\T}'_{1, \S^{j}_{i-j}}} \displaystyle\int_{\check{\A}'_{\S^{j}_{i-j}}} \displaystyle\int_{\check{\T}'_{2, \S_{i}}} \Theta'_{\Pi}(c(\S_{i})\check{p}(\check{t}_{j}\check{a}_{j}\check{h}_{j}\check{b}_{j}\check{t}_{i})c(\S_{i})^{-1}) |\Delta_{\G'}(\check{t}_{j}\check{a}_{j}\check{h}_{j}\check{b}_{j}\check{t}_{i})|^{2} \Lambda(c(\S_{i})\check{p}(\check{t}_{j}\check{a}_{j}\check{h}_{j}\check{b}_{j}\check{t}_{i})c(\S_{i})^{-1})\\ 
                                 & & \qquad \displaystyle\int_{\GL(\X_{i}) / \T'_{1}(\S_{i}) \times \A'(\S_{i})} \displaystyle\int_{\G(\U_{i}) / \T'_{2}(\S_{i})}\varphi^{\widetilde{\K'}}_{\widetilde{\N'}(\S_{i})}(g_{1}g_{2}c(\S_{i})\check{p}(\check{t}_{j}\check{a}_{j}\check{h}_{j}\check{b}_{j}\check{t}_{i})c(\S_{i})^{-1}g^{-1}_{2}g^{-1}_{1}) \overline{dg_{2}}\overline{dg_{1}}d\check{t}_{j}d\check{a}_{j}d\check{h}_{j}d\check{b}_{j}d\check{t}_{i}.
\end{eqnarray*}

\noindent On the other hand, it follows from Equation \eqref{EquationH(Si)2} that {\small
\begin{eqnarray*}
& & \Theta'_{\Pi}(\varphi) \\
& = & \sum\limits_{j=0}^{\min(i, p)} \C_{i} \displaystyle\int_{\widetilde{\T_{1}}(\S_{j})} \displaystyle\int_{\widetilde{\A}(\S_{j})} \displaystyle\int_{\widetilde{\T_{2}}(\S_{j})} \Theta_{\Pi}(\tilde{t}_{1}\tilde{a}\tilde{t}_{2}) |\det(\Id - \Ad(\tilde{t}_{1}\tilde{a}\tilde{t}_{2})^{-1})_{\mathfrak{g}/\mathfrak{h}(\S_{j})}|^{2} \cfrac{\varepsilon(\tilde{t}_{1}\tilde{a})\d_{\S_{j}}(\tilde{t}_{1}\tilde{a}\tilde{t}_{2})}{|\det(\Id - \Ad(\tilde{t}_{1}\tilde{a}\tilde{t}_{2}))_{\eta(\S_{j})}|} \\
     & & \displaystyle\int_{\GL(\X_{j})/\T'_{1}(\S_{j}) \times \A'(\S_{j})} \displaystyle\int_{\widetilde{\G}(\U_{j})} \Chc_{\W_{0, j}}(\tilde{t}_{2}\tilde{y}) \varepsilon(\tilde{t}_{1}\tilde{a}\tilde{y}) \d'_{\S_{j}}(g\tilde{t}_{1}\tilde{a}g^{-1}\tilde{y}) \varphi^{\widetilde{\K'}}_{\widetilde{\N'}(\S_{j})}(g\tilde{t}_{1}\tilde{a}g^{-1}\tilde{y}) d\tilde{y}\overline{dg} d\tilde{t}_{2}d\tilde{a}d\tilde{t}_{1} \\
      & = & \sum\limits_{j=0}^{i} \C_{j} \displaystyle\int_{\check{\T}_{1, \S_{j}}} \displaystyle\int_{\check{\A}_{\S_{j}}}\displaystyle\int_{\check{\T}_{2, \S_{j}}} \Theta_{\Pi}(c(\S_{j})\check{p}(\check{t}_{1}\check{a}\check{t}_{2})c(\S_{j})^{-1}) |\det(\Id - \Ad(c(\S_{j})\check{p}(\check{t}_{1}\check{a}\check{t}_{2})c(\S_{j})^{-1})^{-1})_{\mathfrak{g}/\mathfrak{h}(\S_{j})}|^{2} \cfrac{\varepsilon(c(\S_{j})\check{p}(\check{t}_{1}\check{a})c(\S_{j})^{-1})\d_{\S_{j}}(c(\S_{j})\check{p}(\check{t}_{1}\check{a}\check{t}_{2})c(\S_{j})^{-1})}{|\det(\Id - \Ad(c(\S_{j})\check{p}(\check{t}_{1}\check{a}\check{t}_{2})c(\S_{j})^{-1}))_{\eta(\S_{j})}|} \\
     & & \displaystyle\int_{\GL(\X_{j})/\T'_{1}(\S_{j}) \times \A'(\S_{j})} \displaystyle\int_{\widetilde{\G}(\U_{j})} \Chc_{\W_{0, j}}(\check{p}(\check{t}_{2})\tilde{y}) \varepsilon(c(\S_{j})\check{p}(\check{t}_{1}\check{a})c(\S_{j})^{-1}\tilde{y}) \d'_{\S_{j}}(gc(\S_{j})\check{p}(\check{t}_{1}\check{a})c(\S_{j})^{-1}g^{-1}\tilde{y}) \varphi^{\widetilde{\K'}}_{\widetilde{\N}'(\S_{j})}(gc(\S_{j})\check{p}(\check{t}_{1}\check{a})c(\S_{j})^{-1}g^{-1}\tilde{y}) d\tilde{y}\overline{dg} d\check{t}_{2}d\check{a}d\check{t}_{1} \\
  & = & \sum\limits_{j=0}^{\min(i, p)} \C_{j} \displaystyle\int_{\check{\T}_{1, \S_{j}}} \displaystyle\int_{\check{\A}_{\S_{j}}}\displaystyle\int_{\check{\T}_{2, \S_{j}}} \Theta_{\Pi}(c(\S_{j})\check{p}(\check{t}_{1}\check{a}\check{t}_{2})c(\S_{j})^{-1}) |\det(\Id - \Ad(c(\S_{j})\check{p}(\check{t}_{1}\check{a}\check{t}_{2})c(\S_{j})^{-1})^{-1})_{\mathfrak{g}/\mathfrak{h}(\S_{j})}|^{2} \cfrac{\varepsilon(c(\S_{j})\check{p}(\check{t}_{1}\check{a})c(\S_{j})^{-1})\d_{\S_{j}}(c(\S_{j})\check{p}(\check{t}_{1}\check{a}\check{t}_{2})c(\S_{j})^{-1})}{|\det(\Id - \Ad(c(\S_{j})\check{p}(\check{t}_{1}\check{a}\check{t}_{2})c(\S_{j})^{-1}))_{\eta(\S_{j})}|} \\
     & & \sum\limits_{\sigma \in \mathscr{W}(\mathfrak{g}(\U_{j}))} \varepsilon(\sigma) \lim\limits_{\underset{r \in \E_{\sigma, \S^{j}_{i-j}}}{r \to 1}} \displaystyle\int_{\GL(\X_{j})/\T'_{1}(\S_{j}) \times \A'(\S_{j})} \det(\check{t}_{2})_{\W^{\mathfrak{t}_{2, \S_{j}}}} \Delta_{\Psi'(\mathfrak{g}(\U_{j}))}(\check{t}_{2})^{-1} \displaystyle\int_{\check{\H}'_{\S^{j}_{i-j}}} \cfrac{\det^{-\frac{1}{2}}(\sigma^{-1}(\check{h}'))_{\W^{\mathfrak{t}_{2, \S_{j}}}}}{\det(1-p(\check{h}')rp(\check{t}_{2}))_{\sigma\W^{\mathfrak{t}_{2, \S_{j}}}}}  \displaystyle\int_{\widetilde{\G}(\U_{j})/\H'(\S^{j}_{i-j})}\\
     & & \Delta_{\Psi'^{+}(\mathfrak{g}(\U_{j}))}(\check{h}')\varepsilon(c(\S_{j})\check{p}(\check{t}_{1}\check{a})c(\S_{j})^{-1}g_{2}c(\S^{j}_{i-j})\check{p}(\check{h}')c(\S^{j}_{i-j})^{-1}g^{-1}_{2}) \d'_{\S_{j}}(g_{1}c(\S_{j})\check{p}(\check{t}_{1}\check{a})c(\S_{j})^{-1}g^{-1}_{1}g_{2}c(\S^{j}_{i-j})\check{p}(\check{h}')c(\S^{j}_{i-j})^{-1}g^{-1}_{2}) \\ 
     & & \varphi^{\widetilde{\K'}}_{\widetilde{\N}'(\S_{j})}(g_{1}c(\S_{j})\check{p}(\check{t}_{1}\check{a})c(\S_{j})^{-1}g^{-1}_{1}g_{2}c(\S^{j}_{i-j})\check{p}(\check{h}')c(\S^{j}_{i-j})^{-1}g^{-1}_{2}) \overline{dg_{2}} d\check{h}'\overline{dg_{1}} d\check{t}_{2}d\check{a}d\check{t}_{1}
\end{eqnarray*}}
Using Lemma \ref{IntegrationLemmaProposition67} and the equality $c(\S_{j})c(\S^{j}_{i-j}) = c(\S_{i})$, we get:{\small
\begin{eqnarray*}
& & \Theta'_{\Pi}(\varphi) \\
  & = & \sum\limits_{j=0}^{\min(i, p)} \C_{j} \displaystyle\int_{\check{\T}_{1, \S_{j}}} \displaystyle\int_{\check{\A}_{\S_{j}}}\displaystyle\int_{\check{\T}_{2, \S_{j}}} \Theta_{\Pi}(c(\S_{j})\check{p}(\check{t}_{1}\check{a}\check{t}_{2})c(\S_{j})^{-1}) |\det(\Id - \Ad(c(\S_{j})\check{p}(\check{t}_{1}\check{a}\check{t}_{2})c(\S_{j})^{-1})^{-1})_{\mathfrak{g}/\mathfrak{h}(\S_{j})}|^{2} \cfrac{\varepsilon(c(\S_{j})\check{p}(\check{t}_{1}\check{a})c(\S_{j})^{-1})\d_{\S_{j}}(c(\S_{j})\check{p}(\check{t}_{1}\check{a}\check{t}_{2})c(\S_{j})^{-1})}{|\det(\Id - \Ad(c(\S_{j})\check{p}(\check{t}_{1}\check{a}\check{t}_{2})c(\S_{j})^{-1}))_{\eta(\S_{j})}|} \\
     & &\sum\limits_{\sigma \in \mathscr{W}(\mathfrak{g}(\U_{j}))} \varepsilon(\sigma) \lim\limits_{\underset{r \in \E_{\sigma, \S^{j}_{i-j}}}{r \to 1}} \displaystyle\int_{\GL(\X_{i})/\T'_{1}(\S_{i}) \times \A'(\S_{i})} \det(\check{t}_{2})_{\W^{\mathfrak{t}_{2, \S_{j}}}} \Delta_{\Psi'(\mathfrak{g}(\U_{j}))}(\check{t}_{2})^{-1} \displaystyle\int_{\check{\H}'_{\S^{j}_{i-j}}} \cfrac{\det^{-\frac{1}{2}}(\sigma^{-1}(\check{h}'))_{\W^{\mathfrak{t}_{2, \S_{j}}}}\Delta_{\Psi'^{+}(\mathfrak{g}(\U_{j}))}(\check{h}')}{\det(1-p(\check{h}')rp(\check{t}_{2}))_{\sigma\W^{\mathfrak{t}_{2, \S_{j}}}}}  \cfrac{\D_{\L'(\S_{i})}(c(\S_{i})\check{p}(\check{t}_{1}\check{a}\check{h}')c(\S_{i})^{-1})}{\D_{\L'(\S_{j})}(c(\S_{i})\check{p}(\check{t}_{1}\check{a}\check{h}')c(\S_{i})^{-1})} \\
     & & \displaystyle\int_{\widetilde{\G}(\U_{i})/\T'_{2}(\S_{i})} \varepsilon(c(\S_{i})\check{p}(\check{t}_{1}\check{a}\check{h}')c(\S_{i})^{-1}) \d'_{\S_{j}}(g_{1}g_{2}c(\S_{i})\check{p}(\check{t}_{1}\check{a}\check{h}')c(\S_{i})g^{-1}_{2}g^{-1}_{1})  \varphi^{\widetilde{\K'}}_{\widetilde{\N'}(\S_{j})}(g_{1}g_{2}c(\S_{i})\check{p}(\check{t}_{1}\check{a}\check{h}')c(\S_{i})^{-1}g^{-1}_{2}g^{-1}_{1}) \overline{dg_{2}} d\check{h}'\overline{dg_{1}} d\check{t}_{2}d\check{a}d\check{t}_{1}\, .
\end{eqnarray*}}
\noindent As explained in Remark \ref{RemarkDecompositionHSij}, for every $0 \leq j \leq i$, we have the following decomposition $\H'(\S_{i}) = \T'_{1}(\S_{j}) \times \A'(\S_{j}) \times \T'_{1}(\S^{j}_{i-j}) \times \A'(\S^{j}_{i-j}) \times \T'_{2}(\S_{i})$. In particular, every element $h \in \H'(\S_{i})$ can be written as $h = t_{j}a_{j}h_{j}b_{j}t_{i}$. We get similar results for $\H'_{\S_{i}}$. In particular, we get that the value of $\Theta'_{\Pi}$ on $\H'(\S_{i})$ is given by:
\begin{equation*}
\Theta'_{\Pi}(c(\S_{i})\check{p}(\check{h}')c(\S_{i})^{-1}) |\Delta_{\G'}(\check{h}')|^{2} \Lambda(c(\S_{i})\check{p}(\check{h})c(\S_{i})^{-1}) = \sum\limits_{j=0}^{\min(i, p)}\C_{j} \sum\limits_{\sigma \in \mathscr{W}(\mathfrak{g}(\U_{j}))} \varepsilon(\sigma) 
\end{equation*}
\begin{equation*}
 \Delta_{\Psi'^{+}(\mathfrak{g}(\U_{j}))}(\check{h}_{j}\check{b}_{j}\check{t}_{i})\det^{-\frac{1}{2}}(\sigma^{-1}(\check{h}_{j}\check{b}_{j}\check{t}_{i}))_{\W^{\mathfrak{t}_{2, \S_{j}}}} \cfrac{\D_{\L'(\S_{i})}(c(\S_{i})\check{p}(\check{t}_{j}\check{a}_{j}\check{h}_{j}\check{b}_{j}\check{t}_{i})c(\S_{i})^{-1})}{\D_{\L'(\S_{j})}(c(\S_{i})\check{p}(\check{t}_{j}\check{a}_{j}\check{h}_{j}\check{b}_{j}\check{t}_{i})c(\S_{i})^{-1})} \varepsilon(c(\S_{j}\check{p}(\check{h}_{j}\check{b}_{j}\check{t}_{i})c(\S_{j})^{-1})  
\end{equation*}
\begin{equation*}
\lim\limits_{\underset{r \in \E_{\sigma, \S^{j}_{i-j}}}{r \to 1}} \displaystyle\int_{\check{\T}_{2, \S_{j}}} \cfrac{\d_{\S_{j}}(c(\S_{j})\check{p}(\check{t}_{j}\check{a}_{j}\check{h})c(\S_{j})^{-1})\Theta_{\Pi}(c(\S_{j})\check{p}(\check{t}_{j}\check{a}_{j}\check{h})c(\S_{j})^{-1}) |\det(\Id - \Ad(c(\S_{j})\check{p}(\check{t}_{j}\check{a}_{j}\check{h})c(\S_{j})^{-1})^{-1})_{\mathfrak{g}/\mathfrak{h}(\S_{j})}|^{2} \det(\check{h})^{\frac{1}{2}}_{\W^{\mathfrak{t}_{2, \S_{j}}}}}{\Delta_{\Psi'(\mathfrak{g}(\U_{j}))}(\check{h})\det(\Id - \Ad(c(\S_{j})\check{p}(\check{t}_{j}\check{a}_{j}\check{h})c(\S_{j})^{-1}))_{\eta(\S_{j})}\det(1-p(\check{h}_{j}\check{b}_{j}\check{t}_{i})rp(\check{h}))_{\sigma\W^{\mathfrak{t}_{2, \S_{j}}}}} d\check{h} \, .
\end{equation*}

\noindent Using Equation \eqref{DecompositionPsik}, we get
\begin{equation*}
\D_{\L'(\S_{i})}(c(\S_{i})\check{p}(\check{t}_{j}\check{a}_{j}\check{h}_{j}\check{b}_{j}\check{t}_{i})c(\S_{i})^{-1}) = \prod\limits_{\alpha \in \Psi'^{+}(\mathfrak{l}'(\S_{i}))} |1 - (\check{t}_{j}\check{a}_{j}\check{h}_{j}\check{b}_{j}\check{t}_{i})^{\alpha}| = \prod\limits_{\alpha \in \Psi'^{+}(\mathfrak{gl}'(\X_{i}))} |1 - (\check{t}_{j}\check{a}_{j}\check{h}_{j}\check{b}_{j})^{\alpha}| \prod\limits_{\alpha \in \Psi'^{+}(\mathfrak{g}(\U_{i}))} |1 - (\check{t}_{i})^{\alpha}|
\end{equation*}  
and 
\begin{equation*}
\D_{\L'(\S_{j})}(c(\S_{i})\check{p}(\check{t}_{j}\check{a}_{j}\check{h}_{j}\check{b}_{j}\check{t}_{i})c(\S_{i})^{-1}) = \prod\limits_{\alpha \in \Psi'^{+}(\mathfrak{l}'(\S_{j}))} |1 - (\check{t}_{j}\check{a}_{j}\check{h}_{j}\check{b}_{j}\check{t}_{i})^{\alpha}| = \prod\limits_{\alpha \in \Psi'^{+}(\mathfrak{gl}'(\X_{j}))} |1 - (\check{t}_{j}\check{a}_{j})^{\alpha}| \prod\limits_{\alpha \in \Psi'^{+}(\mathfrak{g}(\U_{j}))} |1 - (\check{h}_{j}\check{b}_{j}\check{t}_{i})^{\alpha}|,
\end{equation*}
so in particular
\begin{equation*} 
\cfrac{\D_{\L'(\S_{0})}(c(\S_{i})\check{p}(\check{t}_{j}\check{a}_{j}\check{h}_{j}\check{b}_{j}\check{t}_{i})c(\S_{i})^{-1})}{\D_{\L'(\S_{j})}(c(\S_{i})\check{p}(\check{t}_{j}\check{a}_{j}\check{h}_{j}\check{b}_{j}\check{t}_{i})c(\S_{i})^{-1})|\Delta_{\G'}(\check{t}_{j}\check{a}_{j}\check{h}_{j}\check{b}_{j}\check{t}_{i})|^{2}} = \cfrac{1}{\prod\limits_{\alpha \in \Psi'^{+}} |1 - (\check{t}_{j}\check{a}_{j}\check{h}_{j}\check{b}_{j}\check{t}_{i})^{\alpha}|\prod\limits_{\alpha \in \Psi'^{+}(\mathfrak{gl}'(\X_{j}))} |1 - (\check{t}_{j}\check{a}_{j})^{\alpha}| \prod\limits_{\alpha \in \Psi'^{+}(\mathfrak{g}(\U_{j}))} |1 - (\check{h}_{j}\check{b}_{j}\check{t}_{i})^{\alpha}|}.
\end{equation*}
Finally, 
\begin{equation*}
\Delta_{\Psi'}(\check{h}') \Theta'_{\Pi}(c(\S_{i})\check{p}(\check{h}')c(\S_{i})^{-1}) =
\end{equation*}
\begin{equation*} 
\sum\limits_{j=0}^{\min(i, p)} \C_{j} \sum\limits_{\sigma \in \mathscr{W}(\mathfrak{g}(\U_{j}))} \varepsilon(\sigma) \cfrac{\Delta_{\Psi'^{+}(\mathfrak{g}(\U_{j}))}(\check{h}_{j}\check{b}_{j}\check{t}_{i})\det^{-\frac{1}{2}}(\sigma^{-1}(\check{h}_{j}\check{b}_{j}\check{t}_{i}))_{\W^{\mathfrak{t}_{2, \S_{j}}}}\Delta_{\Psi'}(\check{t}_{j}\check{a}_{j}\check{h}_{j}\check{b}_{j}\check{t}_{i})\varepsilon(c(\S_{j}\check{p}(\check{h}_{j}\check{b}_{j}\check{t}_{i})c(\S_{j})^{-1})}{\prod\limits_{\alpha \in \Psi'^{+}} |1 - (\check{t}_{j}\check{a}_{j}\check{h}_{j}\check{b}_{j}\check{t}_{i})^{\alpha}|\prod\limits_{\alpha \in \Psi'^{+}(\mathfrak{gl}(\X_{j}))} |1 - (\check{t}_{j}\check{a}_{j})^{\alpha}| \prod\limits_{\alpha \in \Psi'^{+}(\mathfrak{g}(\U_{j}))} |1 - (\check{h}_{j}\check{b}_{j}\check{t}_{i})^{\alpha}|}
\end{equation*}
\begin{equation*}
\lim\limits_{\underset{r \in \E_{\sigma, \S^{j}_{i-j}}}{r \to 1}} \displaystyle\int_{\check{\T}_{2, \S_{j}}} \cfrac{\left(\Theta_{\Pi}(c(\S_{j})\check{p}(\check{t}_{j}\check{a}_{j}\check{h})c(\S_{j})^{-1}) \Delta_{\Psi'}(\check{t}_{j}\check{a}_{j}\check{h})\right)  \Delta_{\Phi'}(\check{t}_{j}\check{a}_{j}\check{h})\det(\check{h}')^{\frac{1}{2}}_{\W^{\mathfrak{t}_{2, \S_{j}}}}}{\Delta_{\Psi'(\mathfrak{g}(\U_{j}))}(\check{h})\det(\Id - \Ad(c(\S_{j})\check{p}(\check{t}_{j}\check{a}_{j}\check{h})c(\S_{j})^{-1}))_{\eta(\S_{j})}\det(1-p(\check{h}_{j}\check{b}_{j}\check{t}_{i})rp(\check{h}))_{\sigma\W^{\mathfrak{t}_{2, \S_{j}}}}} d\check{h}
\end{equation*}
Again, we get from Equation \eqref{DecompositionPsik} that
\begin{eqnarray*}
\Delta_{\Psi'}(\check{t}_{j}\check{a}_{j}\check{h}) & = & \Delta_{\Psi'(\mathfrak{gl}(\X_{i})}(\check{t}_{j}\check{a}_{j}) \Delta_{\Psi'(\mathfrak{g}(\U_{j}))}(\check{h}) \Delta_{\Psi'(\eta'(\S_{j}))}(\check{t}_{j}\check{a}_{j}\check{h}) \\
 & = & (\check{t}_{j}\check{a}_{j})^{-\rho'(\mathfrak{gl}(\X_{i})) - \rho'(\eta'(\S_{j}))} \check{h}^{- \rho'(\eta'(\S_{j}))}\prod\limits_{\alpha \in \Psi'^{+}(\mathfrak{gl}(\X_{j}))}(1 - (\check{t}_{j}\check{a}_{j})^{\alpha}) \prod\limits_{\alpha \in \Psi'^{+}(\eta'(\S_{j}))}(1 - (\check{t}_{j}\check{a}_{j}\check{h})^{\alpha}) \Delta_{\Psi'(\mathfrak{g}(\U_{j}))}(\check{h}),
\end{eqnarray*}
then,
\begin{equation*}
\Delta_{\Psi'}(\check{h}') \Theta'_{\Pi}(c(\S_{i})\check{p}(\check{h}')c(\S_{i})^{-1}) = \sum\limits_{j=0}^{\min(i, p)}\sum\limits_{\sigma \in \mathscr{W}(\mathfrak{g}(\U_{j}))} \varepsilon(\sigma) \det^{-\frac{1}{2}}(\sigma^{-1}(\check{h}_{j}\check{b}_{j}\check{t}_{i}))_{\W^{\mathfrak{t}_{2, \S_{j}}}} \varepsilon(c(\S_{j}\check{p}(\check{h}_{j}\check{b}_{j}\check{t}_{i})c(\S_{j})^{-1}) 
\end{equation*}
\begin{equation*}
\cfrac{\prod\limits_{\alpha \in \Psi'^{+}} 1 - \check{h}'^{\alpha}}{\prod\limits_{\alpha \in \Psi'^{+}} |1 - \check{h}'^{\alpha}|}\cfrac{\prod\limits_{\alpha \in \Psi'^{+}(\mathfrak{gl}(\X_{i}))} 1 - (\check{t}_{j}\check{a}_{j})^{\alpha}}{\prod\limits_{\alpha \in \Psi'^{+}(\mathfrak{gl}(\X_{i}))} |1 - (\check{t}_{j}\check{a}_{j})^{\alpha}|} \cfrac{\prod\limits_{\alpha \in \Psi'^{+}(\mathfrak{g}(\U_{j}))} 1 - (\check{t}_{j}\check{a}_{j}\check{h}_{j}\check{b}_{j}\check{t}_{i})^{\alpha}}{\prod\limits_{\alpha \in \Psi'^{+}(\mathfrak{g}(\U_{j}))} |1 - (\check{t}_{j}\check{a}_{j}\check{h}_{j}\check{b}_{j}\check{t}_{i})^{\alpha}|}
\lim\limits_{\underset{r \in \E_{\sigma, \S^{j}_{i-j}}}{r \to 1}} \displaystyle\int_{\check{\T}_{2, \S_{j}}} \cfrac{\left(\Theta_{\Pi}(c(\S_{j})\check{p}(\check{t}_{j}\check{a}_{j}\check{h})c(\S_{j})^{-1}) \Delta_{\Psi'}(\check{t}_{j}\check{a}_{j}\check{h})\right) \det(\check{h})^{\frac{1}{2}}_{\W^{\mathfrak{t}_{2, \S_{j}}}}}{\det(1-p(\check{h}_{j}\check{b}_{j}\check{t}_{i})rp(\check{h}))_{\sigma\W^{\mathfrak{t}_{2, \S_{j}}}}} d\check{h}
\end{equation*}
Using \cite[Theorem~10~35]{KNA2} and \cite[Theorem~10.48]{KNA2}, it follows that
\begin{equation*}
\Theta_{\Pi}(c(\S_{j})\check{p}(\check{t}_{j}\check{a}_{j}\check{h})c(\S_{j})^{-1}) \Delta_{\Psi'}(\check{t}_{j}\check{a}_{j}\check{h}) = \sum\limits_{w \in \mathscr{S}_{p+q}} c(w) (\check{t}_{j}\check{a}_{j}\check{h})^{w\lambda} = \sum\limits_{w \in \mathscr{S}_{p+q}} c(w) (\check{t}_{j}\check{a}_{j})^{w\lambda}\check{h}^{w\lambda},
\end{equation*}
where $c(w)$ are complex numbers. Using that
\begin{equation*}
\varepsilon(c(\S_{j}\check{p}(\check{h}_{j}\check{b}_{j}\check{t}_{i})c(\S_{j})^{-1})\cfrac{\prod\limits_{\alpha \in \Psi'^{+}} 1 - \check{h}'^{\alpha}}{\prod\limits_{\alpha \in \Psi'^{+}} |1 - \check{h}'^{\alpha}|}\cfrac{\prod\limits_{\alpha \in \Psi'^{+}(\mathfrak{gl}(\X_{i}))} 1 - (\check{t}_{j}\check{a}_{j})^{\alpha}}{\prod\limits_{\alpha \in \Psi'^{+}(\mathfrak{gl}(\X_{i}))} |1 - (\check{t}_{j}\check{a}_{j})^{\alpha}|} \cfrac{\prod\limits_{\alpha \in \Psi'^{+}(\mathfrak{g}(\U_{j}))} 1 - (\check{t}_{j}\check{a}_{j}\check{h}_{j}\check{b}_{j}\check{t}_{i})^{\alpha}}{\prod\limits_{\alpha \in \Psi'^{+}(\mathfrak{g}(\U_{j}))} |1 - (\check{t}_{j}\check{a}_{j}\check{h}_{j}\check{b}_{j}\check{t}_{i})^{\alpha}|}
\end{equation*}
is of norm $1$, it follows from Lemma \ref{LemmaComplexIntegrals} and Theorem \ref{TheoremDiscreteSeries2} that
\begin{equation*}
\lim\limits_{\underset{r \in \E_{\sigma, \S^{j}_{i-j}}}{r \to 1}} \displaystyle\int_{\check{\T}_{2, \S_{j}}} \cfrac{\left(\Theta_{\Pi}(c(\S_{j})\check{p}(\check{t}_{j}\check{a}_{j}\check{h})c(\S_{j})^{-1}) \Delta_{\Psi'}(\check{t}_{j}\check{a}_{j}\check{h})\right) \det(\check{h})^{\frac{1}{2}}_{\W^{\mathfrak{t}_{2, \S_{j}}}}}{\det(1-p(\check{h}_{j}\check{b}_{j}\check{t}_{i})rp(\check{h}))_{\sigma\W^{\mathfrak{t}_{2, \S_{j}}}}} d\check{h}
\end{equation*}
is a finite sum of bounded exponentials, and then, for every $i$,
\begin{equation*}
\sup\limits_{\tilde{h}' \in \widetilde{\H}'(\S_{i})} |\D(\tilde{h}')|^{\frac{1}{2}} |\Theta'_{\Pi}(\tilde{h}')| < \infty.
\end{equation*}

\noindent The proof is similar if $r \leq p$. Note that in this case, $\H_{\S_{i}} = \H'_{\S_{i}}$ (resp. $\H(\S_{i}) = \H'(\S_{i})$) for every $0 \leq i \leq r$, with $\S_{i} = \{e_{1} - e_{p+1}, \ldots, e_{i} - e_{p+i}\}$, and in particular, the Cauchy-Harish-Chandra integrals $\Chc_{\tilde{h}_{i}}, \tilde{h}_{i} \in \H(\S_{i}), r+1 \leq i \leq p$, does not give any contribution to $\Theta'_{\Pi}$ on the different Cartan subgroups $\widetilde{\H}'(\S_{k})$, $0 \leq k \leq r$. It concludes the proof of the proposition.

\end{proof}

\begin{lemme}

Let $(\G = \U(p, q), \G' = \U(r, s)), p+q = r+s$ and $\Pi \in \mathscr{R}(\widetilde{\U}(p, q), \omega)$ a discrete series representation with Harish-Chandra parameter $\lambda_{a, b}$ as in Theorem \ref{TheoremDSHCP}. For every $z \in \Z(\mathscr{U}(\mathfrak{gl}(r+s, \mathbb{C})))$, we get:
\begin{equation*}
z\Theta'_{\Pi} = \chi_{\lambda'_{a, b}}(z)\Theta'_{\Pi},
\end{equation*}
where $\chi_{\lambda'_{a, b}} = \lambda'_{a, b}(\gamma(z))$ as in Appendix \ref{AppendixA}, Remark \ref{RemarkInfinitesimalCharacter}.

\label{ImportantLemma1}

\end{lemme}

\begin{proof}

Obviously, in the equal rank case, $\mathscr{U}(\mathfrak{g}_{\mathbb{C}}) = \mathscr{U}(\mathfrak{g}'_{\mathbb{C}})$. It follows from \cite[Theorem~1.4]{TOM4} that $\Chc^{*}(z\overline{\Theta_{\Pi}}) = z \Chc(\overline{\Theta_{\Pi}})$. Because $\lambda_{a, b}$ and $\lambda'_{a, b}$ are conjugated under $\mathscr{S}_{r+s}$, the result follows from Theorem \ref{TheoremInfinitesimalCharacter}.

\end{proof}

\begin{coro}

For every discrete series representation $\Pi$ of $\widetilde{\U}(p, q)$, we get 
\begin{equation*}
\Chc^{*}(\Theta_{\Pi}) = \C \Theta_{\theta_{r, s}(\Pi)}.
\end{equation*}
with $\C \in \mathbb{C}$.
\label{LastCorollary}
\end{coro}

\begin{proof}

Using Theorem \ref{TheoremDiscreteSeries2}, it follows from Propositions \ref{ImportantProposition1} and \ref{ImportantProposition2} and Lemma \ref{ImportantLemma1} that, up to a scalar, $\Theta'_{\Pi} = \Chc^{*}(\Theta_{\Pi})$ is either the character of a discrete series representations of $\widetilde{\U}(r, s)$ with Harish-Chandra parameter $\tau_{a, b}(\lambda_{a, b})$ if $(r, s) = (r_{\Pi}, s_{\Pi})$ or $0$ if $(r, s) \neq (r_{\Pi}, s_{\Pi})$. The result follows from Theorem \ref{TheoremDiscreteSeries} because $\tau_{a, b}(\lambda_{a, b})$ and $\lambda'_{a, b}$ (as in Theorem \ref{TheoremDSHCP}) are conjugated under $\mathscr{S}_{r} \times \mathscr{S}_{s}$.

\end{proof}

\begin{coro}

If $(\G, \G') = (\U(p, q), \U(r, s)), p+q = r+s$ and $\Pi \in \mathscr{R}(\widetilde{\G}, \omega)$ a discrete series representation of $\widetilde{\G}$. Then, the conjecture \ref{ConjectureHC} holds.

\end{coro}

\begin{proof}

It follows from Theorem \ref{JianShuLi} because $\Pi' = \Pi'_{1}$.

\end{proof}

\section{A commutative diagram and a remark on the distribution $\Theta_{\Pi'}$}

\label{CommutativeDiagram}

We start this section by recalling a result of T. Przebinda (see \cite{TOM8}). Let $(\G, \G') = (\G(\V, \left(\cdot,\cdot\right), \G(\V', \left(\cdot,\cdot\right)')))$ be an irreducible reductive dual pair in $\Sp(\W)$. As proved in \cite{HOW1} (see also Section \ref{SectionCauchyHarishChandra}), the representations appearing in the correspondence are realized as quotients of $\mathscr{H}^{\infty}$, the set of smooths vectors of the metaplectic representation $(\omega, \mathscr{H})$. Let $\Pi \in \mathscr{R}(\widetilde{\G}, \omega)$, $\Pi'$ the corresponding element of $\mathscr{R}(\widetilde{\G'}, \omega)$ and $\N(\Pi \otimes \Pi')$ the $\omega^{\infty}(\widetilde{\G}\cdot\widetilde{\G'})$-invariant subspace of $\mathscr{H}^{\infty}$ such that $\Pi \otimes \Pi' \approx \mathscr{H}^{\infty}/\N(\Pi \otimes \Pi')$ as in Section \ref{SectionCauchyHarishChandra}.

\noindent In particular, 
\begin{equation*}
(\Pi \otimes \Pi')^{*} \approx (\mathscr{H}^{\infty} / \N(\Pi \otimes \Pi'))^{*} \approx \Ann(\N(\Pi \otimes \Pi')) = \left\{\alpha \in {\mathscr{H}^{\infty}}^{*}, \alpha(X) = 0, (\forall X \in \N(\Pi \otimes \Pi'))\right\} \subseteq {\mathscr{H}^{\infty}}^{*},
\end{equation*}
i.e. there exists a unique element, up to a constant, $\Gamma_{\Pi \otimes \Pi'} \in \Hom(\mathscr{H}^{\infty}, {\mathscr{H}^{\infty}}^{*})$ such that $\Hom_{\widetilde{\G}\cdot\widetilde{\G'}}(\omega, \Pi \otimes \Pi') = \mathbb{C} \cdot \Gamma_{\Pi \otimes \Pi'}$.

\begin{rema}

Let $\W = \X \oplus \Y$ be a complete polarization of $\W$. It is well-known that we can realize the representation $\omega$ on $\mathscr{H} = \mathscr{L}^{2}(\X)$: this is the Schrodinger model. Moreover, the space of smooth vectors of $\omega$ is the Schwartz space $\S(\X)$ of $\X$.

\noindent Using the isomorphisms $\mathscr{K}: \S^{*}(\W) \to \S^{*}(\X \times \X)$ and $\Op: \S^{*}(\X \times \X) \to \Hom(\S(\X), \S^{*}(\X))$ (see \cite[Equations~(143)~and~(146)]{TOM6}), there exists a unique distribution $f_{\Pi \otimes \Pi'} \in \S^{*}(\W)$ such that $\Gamma_{\Pi \otimes \Pi'} = \Op \circ \mathscr{K}(f_{\Pi \otimes \Pi'})$. The distribution $f_{\Pi \otimes \Pi'}$ is called the intertwining distribution corresponding to $\Pi \otimes \Pi'$.

\end{rema}

\noindent As explained in \cite[Section~2]{LI2}, the situation turns out to be slightly easier when $\dim(\V) \leq \dim(\V')$ and $(\Pi, \mathscr{H}_{\Pi})$ a discrete series representation of $\widetilde{\G}$. Under those hypothesis, the space $\mathscr{H}^{\infty} \otimes \mathscr{H}^{\infty}_{\Pi}$ has a natural structure of $\widetilde{\G}$-modules. Using the scalar products on $\mathscr{H}$ and $\mathscr{H}_{\Pi}$, we get a natural inner product $\langle\cdot, \cdot\rangle$ on $\mathscr{H}^{\infty} \otimes \mathscr{H}^{\infty}_{\Pi}$. We denote by $\langle\cdot, \cdot\rangle_{\Pi}$ the following form on $\mathscr{H}^{\infty} \otimes \mathscr{H}^{\infty}_{\Pi}$:
\begin{equation*}
\langle\Phi, \Phi'\rangle_{\Pi} = \displaystyle\int_{\widetilde{\G}} \langle\Phi, (\omega \otimes \Pi)(g)\Phi'\rangle dg, \qquad (\Phi, \Phi' \in \mathscr{H}^{\infty} \otimes \mathscr{H}^{\infty}_{\Pi}).
\end{equation*}
One can easily prove that in this context, the previous integral converges absolutely. We denote by $\R(\Pi)$ the radical of the form $\langle\cdot, \cdot\rangle_{\Pi}$ and we still denote by $\langle\cdot, \cdot\rangle_{\Pi}$ the non-degenerate form we got on $\H(\Pi) = \mathscr{H}^{\infty} \otimes \mathscr{H}^{\infty}_{\Pi} / \R(\Pi)$. The group $\widetilde{\G'}$ acts naturally on $\H(\Pi)$ and we denote by $\theta_{0}(\Pi)$ the corresponding $\widetilde{\G'}$-module.

\begin{theo}[\cite{LI2},~Section~2]

\begin{enumerate}
\item There exists $u_{0}, v_{0} \in \mathscr{H}^{\infty}$ and $x, y \in \mathscr{H}^{\infty}_{\Pi}$ such that
\begin{equation*}
\displaystyle\int_{\widetilde{\G}} \langle u_{0} \otimes x, (\omega \otimes \Pi)(\tilde{g})(v_{0} \otimes y)\rangle_{\Pi} \neq 0.
\end{equation*}
Moreover, we get:
\begin{equation*}
\displaystyle\int_{\widetilde{\G}} |\langle\omega(\tilde{g})u, v\rangle|^{2}dg < +\infty, \qquad (u, v \in \mathscr{H}^{\infty}).
\end{equation*}
\item The representation $\Pi$ can be embedded in $\omega$ as an irreducible subrepresentation and $\theta_{0}(\Pi)$ defines an irreducible unitary representation on the completion of $\H(\Pi)$ (completion with respect to $\langle\cdot, \cdot\rangle_{\Pi}$).
\item The map $\Pi \to \theta_{0}(\Pi^{*})$ coincide with Howe's duality correspondence.
\end{enumerate}

\label{JianShuLi}

\end{theo}

\noindent We get the following proposition.

\begin{prop}

Let $(\G, \G') = (\U(\V), \U(\V'))$ with $\dim(\V) \leq \dim(\V')$ and  $\Pi$ be a discrete series representation of $\widetilde{\G}$. The intertwining distribution is given by $f_{\Pi \otimes \Pi'} = \T(\overline{\Theta_{\Pi}})$, where
\begin{equation*}
\T(\overline{\Theta_{\Pi}}) = \displaystyle\int_{\widetilde{\G}} \overline{\Theta_{\Pi}(\tilde{g})} \T(\tilde{g}) d\tilde{g}.
\end{equation*}

\label{LastProposition}

\end{prop}

\begin{proof}

As explained in \cite[Theorem~3.1]{TOM8}, the previous Lemma follows if the following condition
\begin{equation*}
\displaystyle\int_{\widetilde{\G}} |\Omega(\tilde{g})||\Theta_{\Pi}(\tilde{g})| d\tilde{g} < \infty
\end{equation*}
is satisfied, where $\Omega$ is defined in Appendix \ref{AppendixB}.
Using Lemma \ref{LemmaAppendixB}, it follows that there exists $\C_{\Omega} > 0$ such that
\begin{equation*}
\displaystyle\int_{\widetilde{\G}} |\Omega(\tilde{g})||\Theta_{\Pi}(\tilde{g})| d\tilde{g} \leq \C_{\Omega} \displaystyle\int_{\widetilde{\G}} \Xi(\tilde{g})|\Theta_{\Pi}(\tilde{g})| d\tilde{g}.
\end{equation*}
Using the fact that every discrete series satisfies the strong inequality (see \cite[Section~5.1.1]{WAL}), it follows from \cite[Lemma~5.1.3]{WAL} that
\begin{equation*}
\displaystyle\int_{\widetilde{\G}} \Xi(\tilde{g})|\Theta_{\Pi}(\tilde{g})| d\tilde{g} < \infty,
\end{equation*} 
and the proposition follows.

\end{proof}

\begin{coro}

Assume that $(\G, \G') = (\U(p, q), \U(r, s))$, with $p+q = r+s$, and let $\Pi$ be a discrete series representation of $\widetilde{\G}$. Then, there exists a constant $\C_{\Pi \otimes \Pi'} \in \mathbb{C}$ such that $\T(\overline{\Theta_{\Pi}}) = \C_{\Pi \otimes \Pi'}\T(\overline{\Chc^{*}(\Theta_{\Pi})})$.

\label{CorollarySection7}

\end{coro}

\begin{proof}
The proof of this corollary follows from Corollary \ref{LastCorollary} and Proposition \ref{LastProposition}. 
\end{proof}

We finish this section with a remark concerning the global character $\Theta_{\Pi'}$, $\Pi' = \theta(\Pi)$ and $\Pi \in \mathscr{R}(\widetilde{\G}, \omega)$ a discrete series representation. We proved in Corollary \ref{LastCorollary} that $\Chc^{*}(\Theta_{\Pi}) = \Theta_{\Pi'}$ if $\rk(\G) = \rk(\G')$. But the global character $\Theta_{\Pi'}$ can be obtained via $\Theta_{\Pi}$ in a different way. 

\noindent As before, we assume that $\rk(\G) \leq \rk(\G')$. In particular, every discrete series representation $\Pi \in \mathscr{R}(\widetilde{\G}, \omega)$ is a sub-representation of $\omega$ and let $\mathscr{H}(\Pi)$ be the $\Pi$-isotypic component of $\mathscr{H}$. As explained in Proposition \ref{LastProposition}, $\T(\overline{\Theta_{\Pi}})$ is well-defined. Moreover, using \cite[Section~4.8]{TOM6}, the operator $\omega(\overline{\Theta_{\Pi}})$ is a well-defined operator of $\mathscr{H}$ and one can check that $\mathscr{P}_{\Pi} := \omega(d_{\Pi}\overline{\Theta_{\Pi}})$ is a projection operator onto $\mathscr{H}(\Pi)$. As a $\widetilde{\G}\times\widetilde{\G'}$-modules, we get $\mathscr{H}(\Pi) = \Pi \otimes \Pi'$. Let $\lambda$ be the Harish-Chandra parameter of $\Pi$ and $\nu$ the lowest $\widetilde{\K}$-type of $\Pi_{|_{\widetilde{\K}}}$. In particular, according to Theorem \ref{TheoremDiscreteSeries}, as a $\widetilde{\K}\times\widetilde{\G'}$, we get:
\begin{equation*}
\mathscr{H}(\Pi) = \bigoplus\limits_{\xi \in \widetilde{\K}_{\Pi}} m_{\xi} \Pi_{\xi} \otimes \Pi' = \Pi_{\nu} \otimes \Pi' \oplus \left(\bigoplus\limits_{\underset{\xi \in \widetilde{\K}_{\Pi}}{\xi \neq \nu}} m_{\xi} \Pi_{\xi} \otimes \Pi'\right),
\end{equation*}
where $\Pi_{\xi}$ is a $\widetilde{\K}$-module of highest weight $\xi$ and $\widetilde{\K}_{\Pi}$ is the set of irreducible representations of $\widetilde{\K}$ such that $\Hom_{\widetilde{\K}}(\mathscr{H}_{\xi}, \mathscr{H}) \neq \{0\}$. We denote by $\mathscr{H}(\Pi)(\nu)$ the $\nu$-isotypic component of $\mathscr{H}(\Pi)$. We denote by $\mathscr{P}_{\nu}: \mathscr{H}(\Pi) \to \mathscr{H}(\Pi)(\nu)$ the corresponding projection operator and let $\mathscr{P}_{\Pi, \nu} = \mathscr{P}_{\nu} \circ \mathscr{P}_{\Pi}$. Clearly, $\mathscr{P}_{\Pi, \nu} = \Pi(d_{\nu}\overline{\Theta_{\Pi_{\nu}}}) \circ \omega(d_{\Pi}\overline{\Theta_{\Pi}}) = \omega(d_{\nu}\overline{\Theta_{\Pi_{\nu}}}) \circ \omega(d_{\Pi}\overline{\Theta_{\Pi}})$. In particular, for every $\varphi \in \mathscr{C}^{\infty}_{c}(\widetilde{\G'})$, we get:
\begin{equation*}
\Theta_{\Pi'}(\varphi) = \frac{1}{d_{\nu}} \tr(\Id_{\mathscr{H}_{\nu}} \otimes \Pi'(\varphi)) = \frac{1}{d_{\nu}} \tr(\mathscr{P}_{\Pi, \nu} \circ \omega(\varphi)),
\end{equation*}
i.e.
\begin{equation*}
\Theta_{\Pi'}(\varphi) = d_{\Pi} \tr\left(\displaystyle\int_{\widetilde{\K}} \displaystyle\int_{\widetilde{\G}} \displaystyle\int_{\widetilde{\G'}} \overline{\Theta_{\Pi_{\nu}}(\tilde{k})} \overline{\Theta_{\Pi}(\tilde{g})} \varphi(\tilde{g}')\omega(\tilde{k}\tilde{g}\tilde{g}') d\tilde{g}'d\tilde{g}d\tilde{k}\right).
\end{equation*}
In particular, if $\rk(\G) = \rk(\G')$, we get using Corollary \ref{LastCorollary} that:
\begin{equation}
\sum\limits_{i=0}^{p} \cfrac{1}{|\mathscr{W}(\H_{i})|} \displaystyle\int_{\widetilde{\H}_{i}} \Theta_{\Pi}(\tilde{h}_{i}) |\det(\Id - \Ad(\tilde{h}_{i})^{-1})_{\mathfrak{g}/\mathfrak{h}_{i}}| \Chc_{\tilde{h}_{i}}(\varphi) d\tilde{h}_{i} = d_{\Pi} \tr\left(\displaystyle\int_{\widetilde{\K}} \displaystyle\int_{\widetilde{\G}} \displaystyle\int_{\widetilde{\G'}} \overline{\Theta_{\Pi_{\nu}}(\tilde{k})} \overline{\Theta_{\Pi}(\tilde{g})} \varphi(\tilde{g}')\omega(\tilde{k}\tilde{g}\tilde{g}') d\tilde{g}'d\tilde{g}d\tilde{k}\right).
\label{EquationIntroduction}
\end{equation}

\appendix

\section{Some standard isomorphisms}

\label{AppendixA}

\subsubsection{Universal envelopping algebra of $\mathfrak{g}$ as differential operators on $\G$}

Let $\M$ be a real connected manifold of dimension $n$. We denote by $\mathscr{C}^{\infty}(\M)$ the space of smooth functions and $\mathscr{C}^{\infty}_{c}(\M)$ the space of compactly supported function on $\mathscr{C}^{\infty}(\M)$.

\noindent We denote by $\mathscr{X}(\M)$ the set of derivations of $\mathscr{C}^{\infty}(\M)$, i.e.
\begin{equation*}
\mathscr{X}(\M) = \left\{X: \mathscr{C}^{\infty}(\M) \to \mathscr{C}^{\infty}(\M), X(fg) = X(f)g + fX(g)\right\}.
\end{equation*}
The space $\mathscr{X}(\M)$ is the set of $\mathscr{C}^{\infty}$-vectors fields of $\M$.

\begin{defn}

A continuous endomorphism $\D$ of $\mathscr{C}^{\infty}_{c}(\M)$ is called a differential. operator if whenever $\U$ is an open set in $\M$ and $f$ a function of $\mathscr{C}^{\infty}_{c}(\M)$. vanishing on $\U$, then $Df$ vanishes on $\U$.

\end{defn}

\begin{prop}

Let $\D$ be a differential operator on $\M$.For each $p \in \M$ and each open connected neighbourhood $\U$ of $p$ on which the local coordinates system $\Psi: x \to (x_{1}, \ldots, x_{n})$ is valid, there exists a finite set of functions $a_{\alpha}$ of class $\mathscr{C}^{\infty}$ such that for each $f \in \mathscr{C}^{\infty}_{c}(\M)$ with support contained in $\U$,
\begin{equation*}
\D f(x) = \begin{cases} \sum\limits_{\alpha = (\alpha_{1}, \ldots, \alpha_{n})} a_{\alpha}(x) \D^{\alpha} f \circ \Psi^{-1}(x) & \text{ if } x \in \U \\ 0 & \text{ otherwise}\end{cases} 
\end{equation*}

\end{prop}

\begin{proof}

The proof of this result can be found in \cite[Proposition~1]{HEL1}.

\end{proof}

\begin{nota}

We denote by $\D(\M)$ the set of differential operators.

\end{nota}

\noindent From now on, we assume that $\G = \M$ is a connected Lie group. We denote by $(\L, \mathscr{L}^{2}(\G, dg))$ the left regular representation. Obviously, the space $\mathscr{C}^{\infty}_{c}(\G)$ is $\G$-invariant.

\noindent We define an action of $\G$ on $\D(\G)$ by
\begin{equation*}
(\tau(g)\D)(f) = \L_{g} \circ \D(f \circ \L_{g^{-1}}) \qquad (g \in \G, f \in \mathscr{C}^{\infty}_{c}(\G), \D \in \D(\G)).
\end{equation*}

\begin{defn}

We say that $\D \in \D(\G)$ is left-invariant if $\tau(g)\D = \D$ for all $g \in \G$, i.e. $\L_{g} \circ \D(f) = \D(f \circ \L_{g})$.

\end{defn}

\noindent We denote by $\D_{\G}(\G)$ the set of left-invariant differential operators of $\G$. Similarly, we say that $\D$ is right invariant if $\tau_{1}(g)(\D) = \D$ for every $g \in \G$, where 
\begin{equation*}
(\tau_{1}(g)\D)(f) = \R_{g^{-1}} \circ \D(f \circ \R_{g}) \qquad (f \in \mathscr{C}^{\infty}_{c}(\G).
\end{equation*}
The operator $\D$ is is said to be bi-invariant if $\tau(g)\tau_{1}(h)(\D) = \D$ for every $g, h \in \G$, i.e. $\R_{h^{-1}}\circ \L_{g} \circ \D(\L_{g^{-1}} \circ f \circ \R_{h}) = \D(f)$ for every $f \in \mathscr{C}^{\infty}_{c}(\G)$.

\begin{nota}

We denote by $\D^{\G}(\G)$ the set of right-invariant differential operators and by $\D^{\G}_{\G}(\G)$ the set of bi-invariant differential operators.

\label{NotationDGG}

\end{nota}

\noindent We recall the following result.

\begin{theo}

The natural embedding
\begin{equation*}
\mathfrak{g} \to \D_{\G}(\G)
\end{equation*}
extends to an algebra isomorphism
\begin{equation*}
\mathscr{U}(\mathfrak{g}) \to \D_{\G}(\G).
\end{equation*}
Moreover, its restriction to $\Z(\mathscr{U}(\mathfrak{g}))$ is isomorphic $\D^{\G}_{\G}(\G)$.

\label{IsomorphismEnveloppingDGG}

\end{theo}

\begin{proof}

The proof of this result can be found in \cite{HEL2}.

\end{proof}

\subsubsection{Harish-Chandra isomorphism}

Let $\mathfrak{g}$ be a complex reductive Lie algebra and $\mathfrak{h}$ be a Cartan subalgebra of $\mathfrak{g}$.  We denote by $\mathscr{W} = \mathscr{W}(\mathfrak{g}, \mathfrak{h})$ the corresponding Weyl group.

\noindent We denote by $\eta^{+}$ the subalgebra of $\mathfrak{g}$ given by $\eta^{+} = \sum\limits_{\alpha \in \Psi^{+}(\mathfrak{g}, \mathfrak{h})} \mathbb{C} X_{\alpha}$, where $\mathfrak{g}_{\alpha} = \mathbb{C}X_{\alpha}$. Similarly, we denote by $\mathscr{N}$ and $\mathscr{P}$ the following subspaces of $\mathscr{U}(\mathfrak{g})$ given by:
\begin{equation*}
\mathscr{N} = \sum\limits_{\alpha \in \Phi^{+}(\mathfrak{g}, \mathfrak{h})} Y_{\alpha} \mathscr{U}(\mathfrak{g}) \qquad \mathscr{P} = \sum\limits_{\alpha \in \Phi^{+}(\mathfrak{g}, \mathfrak{h})} X_{\alpha} \mathscr{U}(\mathfrak{g}).
\end{equation*}
where $Y_{\alpha}$ is a basis of $\mathfrak{g}_{-\alpha}$.

\begin{lemme}

We get the following decomposition:
\begin{equation}
\mathscr{U}(\mathfrak{g}) = \mathscr{U}(\mathfrak{h}) \oplus (\mathscr{P} + \mathscr{N}).
\label{EquationFinal}
\end{equation}

\end{lemme}

\noindent We denote by $p_{1}: \mathscr{U}(\mathfrak{g}) \to \mathscr{U}(\mathfrak{h})$ the natural projection corresponding to Equation \eqref{EquationFinal}. We restrict this map to $\Z(\mathscr{U}(\mathfrak{g}))$. We denote by $\zeta_{1}$ the map:
\begin{equation*}
\zeta_{1}: \mathfrak{h} \ni h \to \zeta_{1}(h) = h - \rho(h).1 \in \S(\mathfrak{h}),
\end{equation*}
where $\rho = \frac{1}{2} \sum\limits_{\alpha \in \Phi^{+}(\mathfrak{g}, \mathfrak{h})} \alpha \in \mathfrak{h}^{*}$.

\noindent Using the universal property, we can extend the map $\zeta_{1}$ to $\S(\mathfrak{h})$. We denote by $\gamma$ the map:
\begin{equation*}
\gamma = \zeta_{1} \circ p_{1}: \Z(\mathscr{U}(\mathfrak{g})) \to \S(\mathfrak{h}).
\end{equation*}

\begin{theo}

The map $\gamma$ is an algebra homomorphism which is injective. Moreover, $\Im(\gamma) = \S(\mathfrak{h})^{\mathscr{W}}$ and then:
\begin{equation*}
\gamma: \Z(\mathscr{U}(\mathfrak{g})) \to \S(\mathfrak{h})^{\mathscr{W}}.
\end{equation*}
is a bijection.

\label{HarishChandraIsomorphism}

\end{theo}

\begin{rema}

Harish-Chandra's isomorphism classify all the possible infinitesimal character. Indeed, let $\lambda: \mathfrak{h} \to \mathbb{C}$ be a linear map. Using he universal property of the symmetric algebra \cite[Appendix~C]{GOOD}, the linear form $\lambda$ can be extended to a linear map $\lambda: \S(\mathfrak{h}) \to \mathbb{C}$ and by using the map $\lambda$, we get a map $\chi_{\lambda}: \Z(\mathscr{U}(\mathfrak{g})) \to \mathbb{C}$ given by:
\begin{equation*}
\chi_{\lambda}(z) = \lambda(\gamma(z)) \qquad (z \in \Z(\mathscr{U}(\mathfrak{g}))).
\end{equation*}
\label{RemarkInfinitesimalCharacter}
\end{rema}

\noindent We recall the following theorem.

\begin{theo}

Let $\mathfrak{g}$ be a complex reductive Lie algebra and $\mathfrak{h}$ a Cartan subalgebra of $\mathfrak{g}$. Then every homomorphism of $\Z(\mathscr{U}(\mathfrak{g}))$ into $\mathbb{C}$ sending $1$ to $1$ is of the form $\chi_{\lambda}, \lambda \in \mathfrak{h}^{*}$. If $\lambda$ and $\lambda' \in \mathfrak{h}^{*}$, then $\chi_{\lambda} = \chi_{\lambda'}$ if and only if $\lambda$ and $\lambda'$ are in the same $\mathscr{W}$-orbit.

\noindent In particular, $\Spec(\Z(\mathscr{U}(\mathfrak{g}))) \approx \mathfrak{h}^{*} / \mathscr{W}$.

\label{TheoremInfinitesimalCharacter}

\end{theo}

\noindent The proof of this result can be found in \cite{KNA1}.

\section{A general lemma for unitary groups}

\label{AppendixB}

Let $\U$ be a maximal compact subgroup of $\Sp(\W)$. It is well-known that the restriction of $\omega$ to $\widetilde{\U}$ is a direct sum of irreducible representations whose multiplicity is one. Moreover, the lowest $\widetilde{\U}$-type $\V_{\omega}$ is one-dimensional. Let $x$ be a non-zero vector in $\V_{\omega}$ and let $\Omega$ be the function on $\widetilde{\Sp}(\W)$ given by
\begin{equation*}
\Omega(\tilde{g}) = \langle\omega(\tilde{g})x, x\rangle, \qquad (\tilde{g} \in \widetilde{\Sp}(\W)).
\end{equation*} 
We denote by $\xi_{\omega}$ the (unitary) character of $\widetilde{\K}$ such that $\omega(\tilde{k})x = \xi_{\omega}(\tilde{k})x, \tilde{k} \in \widetilde{\K}$. One can check that for every $\tilde{k}_{1}, \tilde{k}_{2} \in \widetilde{\K}$ and $\tilde{g} \in \widetilde{\G}$,
\begin{equation*}
\Omega(\tilde{k}_{1}\tilde{g}\tilde{k}_{2}) = \langle\omega(\tilde{k}_{1}\tilde{g}\tilde{k}_{2})x, x\rangle = \langle\omega(\tilde{g}\tilde{k}_{2})x, \omega(\tilde{k}^{-1}_{1})x\rangle = \xi_{\omega}(k_{2}k^{-1}_{1}) \langle\omega(\tilde{g})x, x\rangle = \xi_{\omega}(k_{2}k^{-1}_{1})\Omega(\tilde{g}).
\end{equation*}
In particular, the map $\widetilde{\G} \ni \tilde{g} \to |\Omega(\tilde{g})| \in \mathbb{C}$ is $\widetilde{\K}$-bi-invariant. In particular, using the decomposition $\widetilde{\Sp}(W) = \widetilde{\K}\widetilde{\A}\widetilde{\K}$ as in \cite[Section~3.6.7]{WAL}, with $\A = \Cl(\A^{+})$, $\A^{+} = \exp(\mathfrak{a}^{+})$, $\mathfrak{a}$ the maximal split Cartan subalgebra of $\mathfrak{sp}(\W)$ and $\mathfrak{a}^{+} = \left\{H \in \mathfrak{a}, \alpha(H) > 0, \alpha \in \Psi^{+}\right\}$, we get for every $\tilde{g} = \tilde{k}_{1}\tilde{a}\tilde{k}_{2} \in \widetilde{\K}\widetilde{\A}\widetilde{\K}$ that $|\Omega(\tilde{g})| = |\Omega(\tilde{a})|$.

\noindent We denote by $\Xi$ the $\widetilde{\K}$-bi-invariant function defined in \cite[Section~4.5.3]{WAL}. 

\begin{lemme}

Let $(\G, \G') = (\U(\V), \U(\V')) \subseteq \Sp((\V \otimes \V')_{\mathbb{R}})$ be a dual pairs of unitary groups such that $\dim_{\mathbb{C}}(\V) \leq \dim_{\mathbb{C}}(\V')$. Then, there exists a constant $\C_{\Omega} > 0$ such that:
\begin{equation}
|\Omega(\tilde{g})| \leq \C_{\Omega} \Xi(\tilde{g}), \qquad (\tilde{g} \in \widetilde{\G}).
\label{EquationEstimatesAppendixB}
\end{equation}

\label{LemmaAppendixB}

\end{lemme}

\begin{rema}

As explained in \cite[Equation~6.4]{TOM8}, for an irreducible reductive dual pair $(\G, \G')$, there exist a constant $\C = \C_{d, d'} > 0$ such that 
\begin{equation*}
\left|\Omega(\tilde{c}(\X))\right| = \C|\det_{\mathbb{R}}(\Id - \X)|^{\frac{d'}{2}}|\det(i\Id - \J\X)|^{-\frac{d'}{2}}, \qquad (\X \in \mathfrak{g}^{c}),
\end{equation*}
where $d = \dim_{\mathbb{D}}(\V)$ and $d' = \dim_{\mathbb{D}}(\V')$.

\label{RemarkAppendixB}

\end{rema}

\begin{proof}

We start by determining the value of $\Omega$ for the dual pair $(\G, \G') = (\U(1, 1), \U(1))$. Let $\G = \K\A\K$ be the decomposition of $\G$ as in \cite[Section~3.6.7]{WAL}. In this case,
\begin{equation*}
\A = \left\{\begin{pmatrix} \ch(X) & \sh(X) \\ \sh(X) & \ch(X) \end{pmatrix}, X \in \mathbb{R}^{*}_{+}\right\}. 
\end{equation*}
Let $a(X) \in \A^{c}$ and $b(X) \in \mathfrak{a}^{c}$ such that $a(X) = c(b(X))$. One can easily check that 
\begin{equation*}
b(X) = \begin{pmatrix} 0 & \alpha(X) \\ \alpha(X) & 0 \end{pmatrix},
\end{equation*}
where $\alpha(X) = \cfrac{\sh(X)}{\ch(X) - 1}$. Note that $\alpha(X) = \cfrac{1}{\th(\frac{X}{2})}$.

\noindent Let $\mathscr{B} = \{e_{1}, e_{2}\}$ be a basis of $\V$ such that $\Mat_{\mathscr{B}}\left(\cdot, \cdot\right) = \Id_{1, 1}$. Then, using that $\mathscr{B}_{\mathbb{R}} = \{e_{1}, e_{2}, ie_{1}, ie_{2}\}$ is a basis of the real vector space $\V_{\mathbb{R}}$, it follows that:
\begin{equation*}
\det_{\mathbb{R}}(\Id - b(X)) = \det\begin{pmatrix} 1 & -\alpha(X) & 0 & 0 \\ -\alpha(X) & 1 & 0 & 0 \\ 0 & 0 & 1 & -\alpha(X) \\ 0 & 0 & -\alpha(X) & 1 \end{pmatrix} =  (1 - \alpha(X)^{2})^{2}.
\end{equation*}
Similarly, using that 
\begin{equation*}
\J = \Mat_{\mathscr{B}_{\mathbb{R}}} \left(\cdot, \cdot\right)_{\mathbb{R}} = \begin{pmatrix} 0 & 0 & 1 & 0 \\ 0 & 0 & 0 & -1 \\ -1 & 0 & 0 & 0 \\ 0 & 1 & 0 & 0 \end{pmatrix},
\end{equation*}
we get:
\begin{eqnarray*}
\det(i\Id - \J b(X)) & = & \det\left(\begin{pmatrix} 0 & 0 & -1 & 0 \\ 0 & 0 & 0 & -1 \\ 1 & 0 & 0 & 0 \\ 0 & 1 & 0 & 0 \end{pmatrix} - \begin{pmatrix} 0 & 0 & 1 & 0 \\ 0 & 0 & 0 & -1 \\ -1 & 0 & 0 & 0 \\ 0 & 1 & 0 & 0 \end{pmatrix} \begin{pmatrix} 0 & \alpha(X) & 0 & 0 \\ \alpha(X) & 0 & 0 & 0 \\ 0 & 0 & 0 & \alpha(X) \\ 0 & 0 & \alpha(X) & 0 \end{pmatrix}\right) \\ 
& = & \det\left(\begin{pmatrix} 0 & 0 & -1 & 0 \\ 0 & 0 & 0 & -1 \\ 1 & 0 & 0 & 0 \\ 0 & 1 & 0 & 0 \end{pmatrix} - \begin{pmatrix} 0 & 0 & 0 & \alpha(X) \\ 0 & 0 & -\alpha(X) & 0 \\ 0 & -\alpha(X) & 0 & 0 \\ \alpha(X) & 0 & 0 & 0 \end{pmatrix} \right) = \det \begin{pmatrix} 0 & 0 & -1 & -\alpha(X) \\ 0 & 0 & \alpha(X) & -1 \\ 1 & \alpha(X) & 0 & 0 \\ -\alpha(X) & 1 & 0 & 0 \end{pmatrix} \\
& = & (1 + \alpha(X)^{2})^{2}
\end{eqnarray*}
Using that 
\begin{equation*}
\cfrac{\th(\frac{X}{2})^{2} - 1}{\th(\frac{X}{2})^{2} - 1} = \cfrac{-1}{\ch(X)},
\end{equation*}
it follows from Remark \ref{RemarkAppendixB} that there exists $\C > 0$ such that:
\begin{equation*}
|\Omega(\tilde{c}(b(X)))| = \C\left|\cfrac{1 - \alpha(X)^{2}}{1 + \alpha(X)^{2}}\right| = \cfrac{\C}{\ch(X)}.
\end{equation*}
In particular, for the dual pair $(\G, \G') = (\U(1, 1), \U(n))$, we get for every $a(X) = c(b(X)) \in \A^{c} \subseteq \G$ that:
\begin{equation*}
|\Omega(\tilde{c}(b(X)))| = \cfrac{\C}{\ch(X)^{n}}.
\end{equation*}
From \cite[Theorem~4.5.3]{WAL}, we know that $\tilde{a}(X)^{-\rho} \leq \Xi(\tilde{a}(X))$, with $\tilde{a}(X) = \tilde{c}(b(X))$. In our case, we get that $\tilde{a}(X)^{-\rho} = e^{-X}$ and using that for every $n \geq 1$, $\ch(X)^{n} \geq \ch(X) \geq e^{X}$, it follows that:
\begin{equation*}
|\Omega(\tilde{c}(b(X)))| = \cfrac{\C}{\ch(X)^{n}} \leq \cfrac{\C}{\ch(X)} \leq \C e^{-X} \leq \C\Xi(\tilde{a}(X)).
\end{equation*}    
In particular, using the $\widetilde{\K}$-bi-invariance of $\Omega$ and $\Xi$, we get that $\Omega(\tilde{g}) \leq \C\Xi(\tilde{g})$ for every $\tilde{g} \in \widetilde{\G}$ for $(\G, \G') = (\U(1, 1), \U(n))$. One can easily check that $\G'$ can be replaced by $\U(r, s)$ and the computations are similar.

\noindent We can now extend it to $(\G, \G') = (\U(p, p), \U(n))$. In this case, 
\begin{equation*}
\A = \left\{\D = \begin{pmatrix} \D_{1}(X) & \D_{2}(X) \\ \D_{2}(X) & \D_{1}(X) \end{pmatrix}, X \in {\mathbb{R}^{*}_{+}}^{p}\right\}
\end{equation*}
where for $X = (X_{1}, \ldots, X_{p})$, $\D_{1}(X) = \diag(\ch(X_{1}), \ldots, \ch(X_{p}))$ and $\D_{2}(X) = \diag(\sh(X_{1}), \ldots, \sh(X_{p}))$. One can easily check that there exists $\C > 0$ such that
\begin{equation*}
|\Omega(\tilde{c}(b(X)))| = \cfrac{\C}{\prod\limits_{i=1}^{p}\ch(X_{i})^{n}}.
\end{equation*}
In this case, $\rho = \sum\limits_{i=1}^{2p} \cfrac{2p - 2i + 1}{2}e_{i}$, and from Equation \eqref{DiagonalHSi}, we get
\begin{equation*}
\tilde{a}(X)^{-\rho} = \diag(e^{-X_{1}}, \ldots, e^{-X_{p}}, e^{X_{1}}, \ldots, e^{X_{p}})^{-\rho} = \prod\limits_{k=1}^{p} e^{-2pX_{k}}
\end{equation*}
If $n \geq 2p$, it follows that $\ch(X_{k})^{n} \geq \ch(X_{k})^{2p} \geq e^{2pX_{k}}$ for every $k \in [|1, p|]$ and then,
\begin{equation*}
|\Omega(\tilde{c}(b(X)))| = \cfrac{\C}{\prod\limits_{i=1}^{p}\ch(X_{i})^{n}} \leq \C\prod\limits_{k=1}^{p} e^{-2pX_{k}} \leq \Xi(\tilde{a}(X)).
\end{equation*}
Again, $\U(n)$ can be replaced by $\U(r, s)$ as long as $r+s \geq 2p$. Finally, it $\G = \U(p, q)$, with $p \leq q$, we get that:
\begin{equation*}
\A = \left\{\D = \begin{pmatrix} \D_{3}(X) & \D_{4}(X) \\ \D^{t}_{4}(X) & \D^{\diamond}_{3}(X) \end{pmatrix}, X \in {\mathbb{R}^{*}_{+}}^{p}\right\}
\end{equation*}
where $\D_{3}(X) = \diag(\ch(X_{1}), \ldots, \ch(X_{p}))$, $\D_{4}(X) = (\diag(\sh(X_{1}), \ldots, \sh(X_{p})), 0_{p, q})$, where $0_{p, q}$ is the zero matrix of $\Mat(p, q-p)$, and 
\begin{equation*}
\D^{\diamond}_{3}(X) = \begin{pmatrix} \diag(\ch(X_{1}), \ldots, \ch(X_{p})) & 0_{p, q-p} \\ 0_{q-p, p} & 0_{q-p, q-p} \end{pmatrix}.
\end{equation*}
and one can check that the computations done for $\U(p, p)$ extends easily to $\U(p, q)$. The lemma follows.

\end{proof}

\begin{rema}

One can easily see that the condition $\dim_{\mathbb{C}}(V) \leq \dim_{\mathbb{C}}(\V')$ of Lemma \ref{LemmaAppendixB} does not mean that similar estimates cannot be obtained in some cases if $\dim_{\mathbb{C}}(V) > \dim_{\mathbb{C}}(\V')$. Indeed, it follows from the proof of Lemma \ref{LemmaAppendixB} that the inequality \eqref{EquationEstimatesAppendixB} can be obtained for $(\G, \G') = (\U(1, 1), \U(1))$.

\end{rema}

\bibliographystyle{plain}


\end{document}